\documentclass[leqno, 10pt, letterpaper]{amsart}

\usepackage{amsmath}
\usepackage{amsfonts}
\usepackage{amssymb}
\usepackage{graphicx}
\usepackage{amsthm,graphicx,color,yfonts}
\usepackage{pgfplots}
\pgfplotsset{compat=1.10}
\usepgfplotslibrary{fillbetween}
\usepackage{hyperref}
\usepackage{accents}

\newtheorem{theorem}{Theorem}
\newtheorem{definition}[theorem]{Definition}
\newtheorem{proposition}[theorem]{Proposition}
\newtheorem{lemma}[theorem]{Lemma}
\newtheorem{corollary}[theorem]{Corollary}

\theoremstyle{remark}
\newtheorem{remark}[theorem]{Remark}

\newcommand{\ls}{\lesssim}
\newcommand{\gs}{\gtrsim}
\newcommand{\la}{\langle}
\newcommand{\ra}{\rangle}
\newcommand{\R}{\mathbb{R}}
\newcommand{\T}{\mathbb{T}}
\newcommand{\C}{\mathbb{C}}
\newcommand{\Z}{\mathbb{Z}}
\newcommand{\N}{\mathbb{N}}
\newcommand{\pa}{\partial}

\newcommand{\supp}{{\mbox{supp}}}

\def\norm#1{\left\|#1\right\|}
\def\bra#1{\langle#1\rangle}

\def\norm#1{\|#1\|}
\def\bra#1{\langle#1\rangle}
\def\wt#1{\widetilde{#1}}
\def\wh#1{\widehat{#1}}

\DeclareMathOperator*{\med}{med}

\usepackage{wrapfig}
\usepackage{tikz}
\usetikzlibrary{arrows,calc,decorations.pathreplacing}
\definecolor{light-gray1}{gray}{0.90}
\definecolor{light-gray2}{gray}{0.80}
\definecolor{light-gray3}{gray}{0.60}

\numberwithin{equation}{section}

\numberwithin{theorem}{section}

\numberwithin{table}{section}

\numberwithin{figure}{section}

\ifx\pdfoutput\undefined
  \DeclareGraphicsExtensions{.pstex, .eps}
\else
  \ifx\pdfoutput\relax
    \DeclareGraphicsExtensions{.pstex, .eps}
  \else
    \ifnum\pdfoutput>0
      \DeclareGraphicsExtensions{.pdf}
    \else
      \DeclareGraphicsExtensions{.pstex, .eps}
    \fi
  \fi
\fi

\date{\today}
\begin{document}

\title[Periodic FPU]{Periodic FPU System: Continuum Limit to KdV via Regularization and Fourier Analysis}

\author[C. Kwak]{Chulkwang Kwak}
\address{Department of Mathematics, Ewha Womans University, Seoul 03760, Korea}
\email{ckkwak@ewha.ac.kr}

\author[C. Yang]{Changhun Yang}
\address{Department of Mathematics, Chungbuk National University, Cheongju-si 28644, Korea}
\email{chyang@chungbuk.ac.kr}

    \thanks{$*$It is confirmed that this work was contributed equally by all authors.}


\begin{abstract}
The Fermi-Pasta-Ulam (FPU) system, initially introduced by Fermi for numerical simulations, models vibrating chains with fixed endpoints, where particles interact weakly, nonlinearly with their nearest neighbors. Contrary to the anticipated ergodic behavior, the simulation revealed nearly periodic (quasi-periodic) motion of the solutions, a phenomenon later referred to as the \emph{FPU paradox}. A partial but remarkable explanation was provided by Zabusky and Kruskal \cite{ZK-1965}, who formally derived the continuum limit of the FPU system, connecting it to the Korteweg–de Vries (KdV) equation. This formal derivation was later rigorously justified by Bambusi and Ponno \cite{BP-2006}.

\medskip

In this paper, we revisit the problem studied in \cite{BP-2006}, specifically focusing on the continuum limit of the periodic FPU system for a broader class of initial data, as the number of particles $N$ tends to infinity within a fixed domain. Unlike the non-periodic case discussed in \cite{HKY2021}, periodic FPU solutions lack a (local) smoothing effect, posing a significant challenge in controlling one derivative in the nonlinearity. This control is crucial not only for proving the (uniform in $N$) well-posedness for rough data but also for deriving the continuum limit. The main strategies to resolve this issue involve deriving $L^4$-Strichartz estimates for FPU solutions, analogous to those previously derived for KdV solutions in \cite{B-1993KdV}, and regularizing the system via the normal form method introduced in \cite{Babin2011}.
\end{abstract}

\maketitle
	
\tableofcontents	

\section{Introduction}

\subsection{A short history of FPU system}
The Fermi-Pasta-Ulam (FPU) system has a rich history that spans from its origin in the 1950s to its lasting impact on nonlinear dynamics and statistical mechanics. The FPU system was introduced by physicists Enrico Fermi, John Pasta, and Stanislaw Ulam, with contributions from Mary Tsingou \cite{FPUT-1955}. The system was a chain of particles connected by springs, but with nonlinear force laws between the particles, differing from the typical linear harmonic oscillator models. Fermi and his collaborators expected that the system would eventually reach thermal equilibrium, with energy evenly distributed across all particles, which would resemble the ergodic behavior predicted by classical statistical mechanics.

\medskip

However, the simulation did not show the expected behavior. Instead of reaching thermal equilibrium, the system exhibited a quasi-periodic motion, with energy remaining largely localized in certain modes, and the system failed to reach a uniform distribution of energy. This unexpected behavior became known as the \emph{FPU paradox}. This paradox was a catalyst for deeper studies into nonlinear dynamics, chaos theory, and soliton solutions. It revealed that nonlinear systems could exhibit complex behavior, and that equilibrium might not always be reached in the way traditional statistical mechanics suggested.

\medskip

A partial explanation for the FPU paradox was provided by Zabusky and Kruskal \cite{ZK-1965}. They discovered that, for a specific class of nonlinear systems, the continuum limit of the FPU system could be described by the Korteweg-de Vries (KdV) equation, which describes the evolution of waves in a shallow water surface or other physical systems with similar dynamics. The KdV equation is known to admit soliton solutions\footnote{Such traveling waves were first discovered by Russell \cite{Russell1844}.}, which are stable, localized wave packets that interact elastically without changing their shape or speed.

\medskip

While Zabusky and Kruskal provided a formal derivation of KdV from FPU, it was not until the work of Bambusi and Ponno \cite{BP-2006} that a rigorous mathematical justification for the continuum limit of the periodic FPU system was established. They demonstrated that, under certain conditions, the periodic FPU system indeed converged to the KdV equation as the number of particles tends to infinity. This rigorous result confirmed the connection between the discrete FPU system and the continuum KdV system.

\medskip 

The study of the FPU system has received extensive attention and has progressed in several directions over the years, see also \cite{BI2005, G-2007} and references therein:

\begin{itemize}
\item Nonlinear dynamics and Chaos: The FPU paradox contributed significantly to the development of chaos theory and nonlinear dynamics, particularly by exploring how energy can be trapped in nonlinear systems and how small perturbations can lead to large-scale changes in behavior, see, for instance, \cite{Rink2001, PB-2005, BCMM2015} and references therein.

\item Solitons and KdV equation: The connection between the FPU system and the KdV equation sparked interest in the study of solitons, which are stable, localized wave solutions that propagate without changing shape. Solitons have since become a central theme in the study of nonlinear wave equations, see, for instance, \cite{FP-1999, Iooss2000, M-2013} and references therein. 

\item Extensions and Variants: Over the years, the FPU system has been generalized in various ways, including the study of higher-dimensional systems, different forms of nonlinearity, and the effects of varying boundary conditions, see, for instance, \cite{SW-2000, BKP2015, CH2018, HKY2021} and references therein.

\item Quantum Analogues: More recently, the FPU system has inspired investigations into quantum versions, where quantum effects are incorporated, and new questions about quantum thermalization and solitonic behavior have arisen, see, for instance, \cite{LT2014, YR2024} and references therein.
\end{itemize}

\subsection{Notations and definitions}\label{sec:Intro-Not and Def}
We introduce the following notations and definitions before discussing our problem.
\begin{itemize}
\item The discrete Laplacian $\Delta_1$ on $\Z$
\[(\Delta_1 u)(j) = u(j+1) + u(j-1) - 2u(j), \quad j \in \Z.\]
\item $h\Z$ represents the infinite lattice with a mesh size of $h$ for $h > 0$.
\item The periodic lattice domain $\T_h$
\[\mathbb{T}_h= h\Z/(2\pi\Z) =\{ x=hn \; | \; n=-N,\cdots, -1,0,1,\cdots,N-1\},\] 
that is, additive group of $(2N)$ points for a given $N \in \N$. The parameter $h$ denotes the distance between adjacent points in $\T_h$. Note that $h = \frac{\pi}{N}$.
\item The discrete Laplacian $\Delta_h$ on $\T_h$
\[(\Delta_h u)(x) = \frac{1}{h^2}\left(u(x+h) + u(x-h) - 2u(x)\right), \quad x \in \T_h.\]
\item For $1 \leq p \leq \infty$, the Lebesgue space $L^p(\T_h)$ (resp. $L^p(\T)$) is defined by the collection of real-valued functions on $\T_h$ (resp. on $\T$) equipped with the $L^p$-norm
\[\begin{aligned}
&\|f_h\|_{L^p(\T_h)}:=  \left\{\begin{aligned}
  &\left(  h \sum_{x\in \mathbb{T}_h} |f_h(x)|^p \right)^\frac1p,&&\textup{if }1\leq p<\infty,\\
  &\sup_{x\in\mathbb{T}_h} |f_h(x)|,&&\textup{if }p=\infty.
  \end{aligned}\right.\\
\Bigg(\mbox{resp}. \;&\|f\|_{L^p(\T)}:=  \left\{\begin{aligned}
  &\left(\int_{\T}|f_h(x)|^p \; dx \right)^\frac1p,&&\textup{if }1\leq p<\infty,\\
  &\sup_{x\in\T} |f(x)|,&&\textup{if }p=\infty.
  \end{aligned}\right. \Bigg)
\end{aligned}\]
\item With a particular $h=\frac{\pi}{N}$, $(\T_h)^*$ denotes the Pontryagin dual space of $\T_h$ defined by
\[(\T_h)^*:= \Z/(2N\Z) = \Z_{2N} = \left\{ -N,\cdots,-1,0,1,\cdots,N-1\right\},\]
that is, additive group of $(2N)$ points.
\item For $f \in L^1(\T)$ and $f_h\in L^1(\T_h)$, we define the Fourier coefficients of $f$ and $f_h$ by
 \begin{equation}\label{Fourier transform}
  \begin{aligned}
\mathcal{F} (f)(k):=&~{}\frac{1}{\sqrt{2\pi}}\int_{\T} f(x)e^{-ixk} \; dx,\quad \forall k\in \Z,\\
\mathcal{F}_h (f_h)(k):=&~{}\frac{h}{\sqrt{2\pi}}\sum_{x\in \mathbb{T}_h}  f_h(x)e^{-ixk},\quad\forall k\in (\T_h)^*.
\end{aligned}
 \end{equation}
We simply use $\; \widehat{\cdot} \;$ for both $\mathcal F$ and $\mathcal F_h$, if no confusion. We note that $\widehat{f_h}$ is periodic with period $2N$ satisfying $\widehat{f_h}(k+2N) = \widehat{f_h}(k)$ for all $-N \le k < N$.
\item For $1\le q\le \infty$, the Lebesgue space $\ell^q((\T_h)^*)$ (resp. $\ell^q(\Z)$) is defined by the collection of real-valued functions on $(\T_h)^*$ (resp. on $\Z$) equipped with the $\ell^q$-norm
\[\begin{aligned}
&\|g_h\|_{\ell^p((\T_h)^*)}:=  \left\{\begin{aligned}
  &\left( \sum_{k\in (\T_h)^*} |g_h(k)|^p \right)^\frac1p,&&\textup{if }1\leq p<\infty,\\
  &\sup_{k\in (\T_h)^*} |g_h(k)|,&&\textup{if }p=\infty.
  \end{aligned}\right.\\
\Bigg(\mbox{resp}. \;&\|g\|_{\ell^p(\Z)}:=  \left\{\begin{aligned}
  &\left(  \sum_{k \in \Z}|g(k)|^p  \right)^\frac1p,&&\textup{if }1\leq p<\infty,\\
  &\sup_{x\in\Z} |g(x)|,&&\textup{if }p=\infty.
  \end{aligned}\right. \Bigg)
\end{aligned}\]
\item For  $\widehat{f} \in \ell^1(\Z)$ and $\widehat{f_h} \in \ell^1((\T_h)^*)$, the Fourier series (or Fourier inversion formula) is given by
  \begin{equation}\label{Inverse Fourier transform}
\begin{aligned}
\mathcal F^{-1}(\widehat{f}) = f(x):=&~{}\frac{1}{\sqrt{2\pi}}\sum_{k \in \Z} \widehat{f}(k) e^{ixk},\quad\forall x\in \T, \\
\mathcal F_h^{-1}(\widehat{f_h}) = f_h(x):=&~{}\frac{1}{\sqrt{2\pi}} \sum_{k \in (\T_h)^*} \widehat{f_h}(k) e^{ixk},\quad\forall x\in \T_h. 
\end{aligned}
\end{equation}
\end{itemize}
There are several ways to define differentials on $\T_h$. Throughout the paper, we use the following three different types of differentials, all of which are consistent with differentiation on $\T$ in the sense that, as the Fourier multiplier, the symbol of discrete differentials formally converges to $ik$, as $h\rightarrow0$.
  \begin{definition}[Differentials on $\mathbb{T}_h$]\label{def: discrete differentials}
$\;$
\begin{enumerate}
\item $\nabla_h$ (resp., $|\nabla_h|$, $\langle \nabla_h\rangle$) denotes the discrete Fourier multiplier of the symbol $\frac{2i}{h}\sin(\frac{hk}{2})$ (resp., $|\frac{2}{h}\sin(\frac{hk}{2})|$, $\langle\frac{2}{h}\sin(\frac{hk}{2})\rangle$), where $\langle \cdot \rangle=(1 + |\cdot|^2)^{\frac12}$.\footnote{These definitions are consistent with the discrete Laplacian $\Delta_h$, because $(-\Delta_h)$ is the Fourier multiplier of the symbol $\frac{4}{h^2}\sin^2(\frac{hk}{2})$, thus, $|\nabla_h|=\sqrt{-\Delta_h}$ and $\langle\nabla_h\rangle=\sqrt{1-\Delta_h}$.}
\item $\partial_h$ (resp., $|\partial_h|$, $\langle \partial_h\rangle$) denotes the discrete Fourier multiplier of the symbol $ik$ (resp., $|k|$, $\langle k\rangle$).
\item $\partial_h^+$ denotes the discrete right derivative naturally defined by
  \begin{equation}\label{right derivative}
    (\partial_h^+ f_h)(x) := \frac{f_h(x+h)-f_h(x)}{h}, \quad \forall x \in \T_h.
  \end{equation}
\end{enumerate}
  \end{definition}
  \begin{remark}
Formally, we have the following convergences as $h \to 0$:
\[\mathbb{T}_h\to\mathbb{T},\quad(\mathbb{T}_h)^*\to\mathbb{Z},\quad\mathcal{F}_h\to \mathcal{F},\quad\mathcal{F}_h^{-1}\to\mathcal{F}^{-1}.\]
    \end{remark}

\subsection{Main result}
The FPU system consists of $2N$ particles arranged in a one-dimensional chain, where neighboring particles interact weakly through nonlinear springs. This system was originally proposed to study the distribution of energy and the dynamical behavior of nonlinear classical particle chains. It can be described using Hamiltonian mechanics, where the motion of particles is governed by a specific Hamiltonian. The total energy of the FPU system is the sum of the kinetic and potential energies, and the Hamiltonian $\mathcal H$ is given by
\begin{equation}\label{eq:FPU Hamiltonian_0}
\mathcal H(q,p)=\sum_{j=-N}^{N-1}\left(\frac{(p(j))^2}{2}+ V\big(q(j+1)-q(j)\big) \right),
\end{equation}
where $(q,p) : \Z_{2N} \to \R \times \R$. Here, $(q(j),p(j))$ denotes the position-momentum pair of the $j$-th particle, and $V\big(q(j+1)-q(j)\big)$ represents the potential energy between adjacent particles $j$ and $j+1$. The potential function $V$ typically includes nonlinear terms and can be expressed as:
\[V(x) = \frac12x^2 + \frac{\alpha}{3}x^3 + \frac{\beta}{4}x^4,\]
where $\alpha$ and $\beta$ are coefficients representing the nonlinear interactions. When $\beta = 0$, the system \eqref{eq:FPU Hamiltonian_0} is called the $\alpha$-FPU system, which is related to the KdV equation. Conversely, when $\alpha = 0$, the system is referred to as the $\beta$-FPU system, which is connected to the modified KdV equation. 

\medskip

In this paper, we fix $\displaystyle \alpha = \frac12$ and $\beta = 0$\footnote{Some results established in this paper {\bf cannot} immediately be obtained for the $\beta$-FPU system. This is due to the presence of stronger resonances compared to those of the the $\alpha$-FPU system.}. Accordingly, our Hamiltonian system is given by
\begin{equation}\label{original Hamiltonian}
\left\{
    \begin{aligned} 
&~{}\mathcal H(q,p)=\sum_{j=-N}^{N-1}\left(\frac{(p(j))^2}{2}+ V\big(q(j+1)-q(j)\big) \right), \\ 
  &~{}V(x)=\frac12x^2+\frac{1}6x^3, \\ 
  &~{}q(j+2N)=q(j), \quad p(j+2N)=p(j), \end{aligned}\right.  
\end{equation}
and this yields the following FPU system:
\begin{equation}\label{FPU before rescaling}
  \left\{
\begin{aligned}
&~{}\partial_t q(t,j) = \frac{\partial \mathcal H}{\partial p(j)} = p(t,j),\\
&~{}\partial_t p(t,j)=-\frac{\partial \mathcal H}{\partial q(j)} = V'\big(q(t,j+1)-q(t,j)\big)-V'\big(q(t,j)-q(t,j-1)\big).
\end{aligned}\right.
\end{equation}
We define the relative displacement between two adjacent points as
\begin{equation}\label{eq:displacement}
r(t,j) := q(t,j+1)-q(t,j).
\end{equation}
This allows us to rewrite the system \eqref{FPU before rescaling} in the form
\[\partial_t^2 r = \Delta_1 \left( V'(r) \right),\]
where $V'$ denotes the usual derivative of the function $V$. Suppose that $N$ is a (sufficiently large) number of particles, and let
\[
h := \frac{\pi}{N}
\]
denote the distance between adjacent points. We rescale the variable $r$ by defining
\[r_h(t,x)=\frac{1}{h^2}r\left(\frac{t}{h^3},\frac{x}{h}\right),  \quad r_h : \R \times \T_h \to \R,\]
which leads to the rescaled equation
\begin{equation}\label{eq:rescaled FPU wave}
h^6\partial_t^2r_h= \Delta_h\left( V'(h^2r_h)\right),
\end{equation}
where $\Delta_h$ is the discrete Laplacian on $\T_h$ (see Section \ref{sec:Intro-Not and Def} for the definitions of $\Delta_h$ and $\T_h$). Note that \eqref{eq:rescaled FPU wave} is now posed on the box $\T_h$ with $2N$ particles. For the potential function $V(x)=\tfrac12x^2+\tfrac16x^3$, one obtains the following initial value problem for the discrete nonlinear wave equation, referred to hereafter as the FPU system:
\begin{equation}\label{FPU}
  \left\{\begin{aligned}
  &~{}\partial_t^2 r_h-\frac{1}{h^4}\Delta_h r_h=\frac{1}{2h^2}\Delta_h\big(r_h^{\,2}\big),\\
  &~{}r_h(0)=r_{h,0}, \quad \partial_t r_h(0)=r_{h,1}.
  \end{aligned}\right.
  \end{equation}
\begin{remark}
The FPU system \eqref{FPU} remains a Hamiltonian equation with the Hamiltonian\footnote{It can be derived from \eqref{original Hamiltonian}.}
\begin{equation}\label{Hamiltonian}
\mathcal H_h(r_h)=h\sum_{x\in \mathbb{T}_h}\left(\frac{1}{2}\bigg(\frac{h^2}{\sqrt{-\Delta_h}}\partial_tr_h\bigg)^2+\frac{1}{h^4}V(h^2r_h)\right).
\end{equation}
\end{remark}

\begin{remark}\label{rem:mean zero}
From \eqref{eq:displacement} and \eqref{FPU before rescaling}, together with the periodicity of $q(j)$ and $p(j)$, a direct computation gives
\[h\sum_{x\in\T_h} r_h(t,x) = h\sum_{x \in \T_h} \frac{1}{h^2}\left(q\left(\frac{t}{h^3},\frac{x}{h}+1\right)-q\left(\frac{t}{h^3},\frac{x}{h}\right)\right) = 0\]
and
\[h\sum_{x\in\T_h} (\partial_t r_h)(t,x) = h\sum_{x \in \T_h} \frac{1}{h^5}\left(p\left(\frac{t}{h^3},\frac{x}{h}+1\right)-p\left(\frac{t}{h^3},\frac{x}{h}\right)\right) = 0\]
for all $t$. These observations are crucial for the formal derivation of the coupled system, particularly when imposing the mean-zero condition on solutions in our analysis.
\end{remark}

\medskip

The main purpose of this paper is to study the continuum limit of the FPU system \eqref{FPU} as $h \to 0$, that is, as the number of particles $N$ tends to infinity. According to Zabusky and Kruskal \cite{ZK-1965}, one can expect that solutions to \eqref{FPU} are approximated by counter-propagating KdV waves
\[  r_h(t,x)\approx w^+\left(t, x-\frac{t}{h^2}\right)+ w^-\left(t, x+\frac{t}{h^2}\right),\]
  in an appropriate sense as $h \to 0$. Here, $w^\pm=w^\pm(t,x):\R \times\T \to\R$ are solutions to
  \begin{equation}\label{KdV0}
\pa_t w^\pm \pm \frac{1}{24} \pa_x^3 w^\pm \mp \frac{1}{4}\pa_x((w^\pm)^2)=0.
  \end{equation}

\medskip

The precise statement of the theorem is as follows:
  \begin{theorem}[KdV limit for periodic FPU]\label{main theorem}
Let $0 < s \le 1$ and $R > 0$ be given. Suppose that the initial data $(r_{h,0},r_{h,1})$ satisfies
\[\sup_{h\in(0,1]}  \left\|\left( r_{h,0},h^2 \nabla_h^{-1}r_{h,1}\right)\right\|_{\mathbb H^s(\T_h)} \le R,\]
where $\mathbb H^s(\T_h)$-norm is defined as in \eqref{eq:H^s for vector} (see Section \ref{sec:DFA1} for details). Then, there exist $0 < h_0 = h_0(R) < 1$ and $T(R)>0$ sufficiently small such that the following holds: 

\medskip

Let $0 < h \le h_0$ be fixed. Let $r_h(t)\in C_t([-T,T]: H^s(\mathbb{T}_h))$ be the solution to FPU \eqref{FPU} with initial data $( r_h(0),\partial_t r_h(0))$
and let $w^\pm(t)\in C_t([-T,T]: H^s(\mathbb{T}))$ be the solution to KdV \eqref{KdV0} with the initial data 
\[w_0^{\pm} =\mathcal{L}_h r_{h,0}^{\pm}, \quad \mbox{where} \quad r_{h,0}^{\, \pm} = \frac12\left( r_{h,0}\mp h^2 \nabla_h^{-1}r_{h,1} \right).\]
Here $\mathcal L_h$ denotes the interpolation operator on $L^2(\T_h)$, defined as in \eqref{def:linear interpolation_0} (see Section \ref{sec:DisLin}). Then, there exists $C(R) > 0$ independent on $h$ such that the following estimate holds: 
\begin{equation}\label{continuum limit}
  \sup_{t \in [-T,T]} \left\|( \mathcal{L}_hr_h)(t,x)- w^+(t, x-\tfrac{t}{h^2})- w^-(t, x+\tfrac{t}{h^2})\right\|_{L^2(\mathbb{T})}\le C(R) h^{\frac{2}{5}s}.
  \end{equation} 
\end{theorem}

\begin{remark}
Theorem \ref{main theorem} states that the interpolated solutions of the FPU system \eqref{FPU} can be approximated by counter-propagating waves that satisfy the KdV equation \eqref{KdV0}, provided the number of particles $N$ is sufficiently large (as $h \to 0$). This result underscores the connection between the discrete FPU system and the continuous KdV equation, demonstrating that the dynamics of the FPU system can be effectively described by the KdV approximation in the continuum limit.
\end{remark}

\begin{remark}
Theorem \ref{main theorem} ensures that the FPU system can be approximated by the KdV equation only when $h > 0$ is sufficiently small, corresponding to a sufficiently large $N$. In contrast, \cite{HKY2021} assumes an infinite number of particles, making it unnecessary to impose further restrictions on $h > 0$, as the continuum limit is inherently satisfied.
\end{remark}

\begin{remark}
The assumption on the initial data is simplified compared to the previous work \cite{BP-2006}. Specifically, we only assume a uniform bound on the size of the initial data in a natural Sobolev norm (without any additional weight). Furthermore, the regularity requirement for the continuum limit is reduced to $s>0$, representing a better regularity result (closer to $s=0$) compared to \cite{HKY2021}, where the same problem was considered under the setting of infinite chains.
\end{remark}


\begin{remark}
Let $\T_N = \R/(2N\Z)$ denote a periodic domain of period $2N$, and define $\displaystyle T_N = \left(\frac{N}{\pi}\right)^3T$, where $T$ is found in Theorem \ref{main theorem}. By rescaling, the continuum limit \eqref{continuum limit} in Theorem \ref{main theorem} can be interpreted as the small amplitude limit
\begin{equation}\label{Small amplitude limit}\begin{aligned}
  \sup_{|t| \le T_N} 
  & \left\|( \mathcal{L}_1r)(t,x)- \left(\tfrac{\pi}{N}\right)^2\left(w^+\left(\left(\tfrac{\pi}{N}\right)^3t, \left(\tfrac{\pi}{N}\right)(x-t)\right) + w^-\left(\left(\tfrac{\pi}{N}\right)^3t, \left(\tfrac{\pi}{N}\right)(x+t)\right)\right)\right\|_{L_x^2(\T_N)} \\
  &\le C(R) N^{-\frac32-\frac{2}{5}s},
\end{aligned}  \end{equation} 
This result is consistent with the findings in \cite{HKY2021}, which demonstrate that, under the assumption of an infinite number of particles, the FPU system is described by the KdV equation in the small amplitude limit. It further emphasizes the relationship between the continuum limit of the FPU system and the dynamics governed by the KdV equation.

\smallskip

Moreover, \eqref{Small amplitude limit} confirms that the KdV approximation remains valid over a specific timescale $N^3$, which is consistent with the observations in \cite{BP-2006}.
\end{remark}

\begin{remark}
In the periodic domain, the dynamics of the FPU system are accurately captured by the decoupled KdV equation. The solutions to the decoupled KdV system represent traveling waves that propagate without significant changes in shape or amplitude. These solutions are closely related to the energy distribution among low-frequency modes in the FPU system. Despite the presence of nonlinear interactions, the KdV equation effectively describes the observed behaviors of the FPU system. Moreover, the specific structure of nonlinearity in the FPU system suppresses constructive interference between waves. This contrasts with the typical behaviors observed in short-wavelength systems and provides a precise explanation for the metastability of the FPU system, as introduced in \cite{BP-2006}.
\end{remark}

\begin{remark}
The FPU system does not admit higher regularity conservation laws, while the conservation of the Hamiltonian \eqref{Hamiltonian} ensures that $L^2$ FPU solutions remain well-defined globally in time. This observation suggests the possibility of establishing the continuum limit for $L^2$-data. Consequently, both \eqref{continuum limit} and \eqref{Small amplitude limit} would hold for arbitrary times, provided such low regularity convergence can be achieved. However, our approach does not attain this result, leaving it as an intriguing open question.
\end{remark}

\begin{remark}\label{rem:novelty}
The main novelty of this article lies in presenting an analytic framework that captures the KdV nature under periodic boundary conditions, enabling the establishment of the continuum limit of the FPU system in the low-regularity regime. This achievement, previously observed for the KdV equation in \cite{B-1993KdV}, allows us to overcome the lack of local smoothing effects highlighted in \cite{HKY2021}.
\end{remark}

\subsection{vs. non-periodic problem}
Local smoothing is a fundamental property of solutions to dispersive equations, arising from the dispersive nature of their governing equations. This property refers to the phenomenon, typically observed in dispersive equations on unbounded domains, where solutions become smoother (in terms of Sobolev regularity) over time within localized spatial regions. Remark that this effect is not uniform across the entire domain but depends on dispersive characteristics of the equations and the specific spatial region under consideration.

The local smoothing effect results from the interplay between dispersion and time integration. Specifically, dispersion causes waves of different frequencies to spread apart, reducing constructive interference and oscillatory behavior in localized areas. Simultaneously, the spreading wavefront smooths out oscillatory irregularities, producing a locally smoother profile over time. This effect is a crucial tool for studying dispersive equations, particularly those with nonlinearities involving derivatives.

In contrast, in the periodic domain, waves are confined to the bounded region and mapped onto the periodic boundary. This confinement prevents frequencies from separating indefinitely, as they do in unbounded domains, significantly weakening the dispersion effect. Furthermore, the bounded nature of the periodic domain often results in persistent constructive interference, thereby inhibiting the smoothing effect and sustaining oscillatory behavior. 

Dispersive estimates, on the other hand, describe how solutions to dispersive equations spread out over time. In unbounded domains $\R^d$, dispersion arises from waves with different frequencies traveling at different speeds, leading to a decay in the amplitude of the solution over time. In other words, waves can spread indefinitely in unbounded domains, causing the solution's amplitude to decay. For instance, for the linear Schr\"odinger equation $iu_t + \Delta u = 0$, the dispersive estimate for the solution is given by
\begin{equation}\label{eq:dispersive estimates}
\|u(t)\|_{L^{\infty}(\R^d)} \lesssim |t|^{-\frac{d}{2}}\|u(0)\|_{L^1(\R^d)},
\end{equation}
which shows that the $L^{\infty}$-norm of the solutions decays as time increases, with the rate determined by the spatial dimension $d$. Based on this property, Strichartz estimates, mixed-norm estimates that describe how solutions behave in terms of integrability and regularity over time and space, can be established. These estimates play an important role in the study of well-posedness, scattering, and stability theory for dispersive equations.

In periodic domains $\T^d$, however, dispersive behavior differs significantly due to the compactness of the domain. In other words,  waves are confined to the bounded domain, and their amplitudes do not decay as strongly over time. Instead of decay in $L^{\infty}$, estimates often focus on bounds for mixed-norm spaces or averaging effects over time:
\begin{equation}\label{eq:L^q estimate}
\|u\|_{L^{q}([0,T] \times \T^d)} \lesssim \|u(0)\|_{H^s(\T^d)},
\end{equation}
for some regularity $s > 0$, reflecting a loss of regularity. See, for instance, \cite{Zygmund1974, Vega1992, B-1993Sch, BGT2003} for the linear Schr\"odinger equation. In particular, $L^4$-Strichartz estimates ensure the $L^2$ global well-posedness of the one-dimensional (cubic) nonlinear Schr\"odinger equation (NLS) (\cite{B-1993Sch}).

\smallskip

In discrete settings, even for infinite lattices ($h\Z$), the standard conservative scheme fails to reproduce dispersive estimates \eqref{eq:dispersive estimates} uniformly with respect to the mesh parameter $h$, or to recover Strichartz estimates (see \cite{IZ2005, IZ-2009}). Nevertheless, uniform (in $h$) Strichartz estimates with a loss of regularity address these issues and enable the continuum limit of discrete NLS \cite{HY-2019DCDS, HY-2019SIAM}.

On lattice domains with periodic boundary conditions, the situation becomes more complex. Inspired by Bourgain \cite{B-1993Sch, B-1993KdV}, it is expected that the lack of local smoothing effect can be compensated by the \emph{dispersive smoothing effect}. The dispersive smoothing effect means an additional gain in regularity, which arises from the observation that the space-time Fourier coefficients of interactions among distinct frequency-localized solutions are concentrated in regions far from the hypersurface $\{(\tau,k) : \tau = \rho(k)\}$, where $\rho(k)$ is the Fourier symbol of the linear operator of the governing equation.

Focusing on the FPU system, the phase functions of its linear propagators (see also Definition \ref{def:Xsb_h} below)
\[s_h(k) = \frac{1}{h^2}\left(k - \frac{2}{h}\sin\left(\frac{hk}{2}\right)\right)\] 
and their derivatives (group velocities) are comparable to those of Airy propagators. Indeed, 
\[\frac{1}{h^2}\left(k - \frac{2}{h}\sin\left(\frac{hk}{2}\right)\right) \sim k^3 \quad \mbox{and} \quad s_h'(k) = \frac{1}{h} \left(1-\cos\left(\frac{hk}{2}\right)\right) \sim k^2\]
on the frequency domain $(\T_h)^*$. These observations allow us to recover the Strichartz, local smoothing and maximal function estimates on $h\Z$, as well as the dispersive smoothing effect in the form of the bilinear estimates (see Proposition 5.1 and Lemma 6.1 in \cite{HKY2021}). On $\T_h$, the dispersive smoothing effect is captured via $L^4$-Strichartz estimates and the bilinear estimates (see Proposition \ref{prop:L4} and Corollary \ref{Cor:Linear Bilinear}, respectively), which represent significant contributions of this work as mentioned in Remark \ref{rem:novelty}. Such results highlight the non-trivial properties of solutions to discrete equations. For instance, solutions to the discrete Schr\"odinger equation 
\[i\partial_t u_h + \Delta_h u_h = 0\]
on $h\Z$ fail to exhibit local smoothing due to the mismatch between two group velocities corresponding to the discrete Laplacian, $\displaystyle  -\frac{2}{h}\sin(hk)$, and the continuous Laplacian, $k$, particularly, near the high-frequency edges $\displaystyle \left(k = \pm\frac{\pi}{h}\right)$ (see \cite{IZ-2009}). For the same reason, solutions on $\T_h$ fail to satisfy $L^4$-Strichartz estimates \eqref{eq:L^q estimate}, due to the near-overlap of twisted annuli, given by
\[\left\{k_1 \in (\T_h)^* : M \le \left|\tau_1 + \frac{2}{h^2}\left(1-\cos\left(hk_1\right)\right)\right| \le 2M\right\}\]
and 
\[\left\{k_2 \in (\T_h)^* : N \le \left|\tau_2 + \frac{2}{h^2}\left(1-\cos\left(hk_2\right)\right)\right| \le 2N\right\},\]
under the constraints $k = k_1+k_2$ and $\tau = \tau_1 + \tau_2$. This overlap reflects a loss of regularity, consistent with the standard Sobolev estimate. Further details and discussions on this failure, see \cite{Ignat2007,HKNY2021}. 

\smallskip

In \cite{HKY2021}, the authors, in collaboration with Hong, proved the continuum limit of the FPU system as the mesh size $h$ tends to $0$. A critical component of this work was deriving local smoothing and maximal function estimates on $h\Z$. These estimates played a key role in controlling one derivative in the nonlinearity, enabling the continuum limit to hold in a broader function space than in the earlier work \cite{SW-2000}. This advancement highlights the broader applicability of the continuum limit in settings where stricter topological constraints are not feasible.

However, as previously discussed, local smoothing and maximal function estimates are no longer valid on $\T_h$. This issue can be resolved by employing dispersive smoothing effects, particularly when deriving uniform bounds for low-regularity solutions. Nonetheless, proving the continuum limit remains challenging because interpolated solutions cannot be included in the $X^{s,b}$ spaces associated with the KdV equation. To resolve this final issue, we employ the normal form reduction method for regularizing the system. Remarkably, not only do the original FPU system and the KdV equation correspond, but their regularized systems also exhibit the same correspondence.

One notable aspect of the FPU system under periodic boundary conditions is the presence of non-trivial resonances. Resonance occurs when energy is transferred or concentrated in specific frequencies or modes under certain conditions. On a periodic domain, resonance primarily arises in at least two scenarios. First, nonlinear terms facilitate interactions between different frequency modes, generating new modes or transferring energy between existing ones. This phenomenon particularly arises when Fourier modes satisfy specific conditions. Second, resonance can emerge when certain combinations of frequencies result in matching phase or group velocities, causing energy to focus in particular directions. As a consequence of resonances, energy transfer occurs between low-frequency and high-frequency modes, significantly influencing the long-term dynamics of nonlinear waves. In some cases, resonance can disrupt solitary waves (solitons) or give rise to complex dynamics such as turbulence. Mathematically, the interaction of different frequencies generates new frequency components, further complicating the analysis and interpretation of the system.

The standard resonance for the KdV equation (and similarly for the FPU system) occurs when at least one Fourier mode in the nonlinear interactions becomes zero. This resonance can be eliminated via renormalization, ensuring that solutions satisfy the mean-zero condition. For the FPU system, the structure of the Hamiltonian system allows the mean-zero condition to be directly imposed, effectively ignoring the first resonance case. However, when the normal form reduction method is applied to the FPU system (or the KdV equation), the nonlinear terms are redistributed, leading to the emergence of new resonances. These include a strong cubic resonance, expressed for the FPU system as
\begin{equation}\label{eq:RESONANCE FPU}
\mp2i \frac{\cos\left(\frac{hk}{2}\right)\cos\left(\frac{hk}{4}\right)}{\frac{4}{h}\sin\left(\frac{hk}{4}\right)}\left|\widehat{\mathcal{V}_h^{\pm}}(t,k)\right|^2\widehat{\mathcal{V}_h^{\pm}}(t,k)
\end{equation}
and for the KdV equation as
\begin{equation}\label{eq:RESONANCE KDV}
\mp\frac{2i}{k}\left|\widehat{\mathcal{W}^{\pm}}(t,k)\right|^2\widehat{\mathcal{W}^{\pm}}(t,k),
\end{equation}
under the specific condition of frequencies $k_1=-k_2=k_3=k$. In addition to these strong resonances, a weaker resonance arises, which is characterized by its dependence on the total mass (square integral) of the solutions. In the case of the KdV equation, this weaker resonance does not appear because the oddness of the multiplier causes cancellation. In contrast, the regularized FPU system lacks this cancellation property, resulting in the presence of a weaker resonance (see Section \ref{sec:regularization} for further details). This weaker term can be effectively controlled by the additional continuum parameter $h$ present in the resonant term. It is also worth noting that the strong resonances described in \eqref{eq:RESONANCE FPU} and \eqref{eq:RESONANCE KDV} are comparable.

\subsection{Idea of the proof}
The proof of Theorem \ref{main theorem} is as follows: First, inspired by Zabusky and Kruskal \cite{ZK-1965}, we reformulate the FPU Hamiltonian system \eqref{FPU} by separating its Duhamel formula into two coupled equations \eqref{coupled FPU'}. In the coupled system, a more intuitive approach for understanding the limit procedure can be achieved by analyzing the symbols of the linear propagators and their asymptotics (see Remark \ref{rem:formal convergence of propagator}). Moreover, the coupled terms can be treated as small perturbations that become negligible in the continuum limit, leading to a convergence scheme for the KdV equation via the decoupled FPU \eqref{decoupled FPU'}. This reformulation makes the problem well-suited for the dispersive PDE techniques presented below.

Introducing the Fourier restriction norm method, initially developed by Bourgain \cite{B-1993Sch, B-1993KdV}, this approach is applied to both coupled and decoupled FPU systems, facilitating the decoupling of the coupled system. Specifically, we first establish $L^4$-Strichartz estimates, which demonstrate that the FPU solutions exhibit dispersive properties under the periodic boundary condition (see Section \ref{subsec:linear estimates}). Based on these estimates, bilinear estimates are derived (see Section \ref{subsec:Bilinear estimates}). It is worth noting that, by utilizing the gap in regularities, the mixed term and the oppositely moving wave can be rendered negligible, enabling the extraction of the decoupled system (see Remark \ref{rem:remove mixed nonlinearities} and Lemmas \ref{lem:bilinear2} and \ref{lem:bilinear3}). As a by-product, uniform-in-$h$ bounds can also be obtained for both systems as well as for the KdV equation.

To establish the continuum limit of the decoupled FPU system to the KdV equation, an additional technical method is required, distinct from the approach based on $h\Z$ \cite{HKY2021}. As previously mentioned, the local smoothing and maximal function estimates are not applicable to the periodic setting. Instead, we apply the regularization mechanism, which was developed in its current form by Babin, Ilyin, and Titi \cite{Babin2011} (known as the \emph{normal form method}), to both the decoupled FPU system and the KdV equation. Surprisingly, not only do the original decoupled FPU system and the KdV equation align well in the continuum limit (as $h \to 0$), but their higher-degree regularized counterparts also exhibit a strong match under the same limit (see Section \ref{sec:regularization}). The standard multilinear estimates (Lemma \ref{lem:multilinear estimate}), combined with $L^4$-Strichartz estimates and certain properties of the linear interpolation operator (Section \ref{sec:DisLin}), enable us to complete the proof.

\subsection*{Organization }
The paper is organized as follows: Section \ref{sec:DFA} presents preliminary analyses of discrete Fourier analysis and explores properties of the linear interpolation operator. In Section \ref{sec: outline}, we derive the coupled and decoupled FPU systems and establish their well-posedness in $L^2(\T_h)$. We also provide a brief proof of Theorem \ref{main theorem} assuming bridge Propositions \ref{Prop:Coupled to Decoupled} and \ref{prop:from decoupled to kdv}. Section \ref{sec:part1} introduces the $X^{s,b}$ spaces tailored to our systems, along with $L^4$-Strichartz estimates and bilinear estimates to demonstrate the uniform-in-$h$ well-posedness of these systems. Collecting these results, we establish Proposition \ref{Prop:Coupled to Decoupled}, a key component of the paper. Section \ref{sec:regularization} develops regularized FPU and KdV systems as part of the continuum limit proof. Finally, Section \ref{sec:part2} provides the proof of Proposition \ref{prop:from decoupled to kdv}, another key result in this paper.

\subsection*{Acknowledgements}
The first author is grateful for support by the Open KIAS Center at Korea Institute for Advanced Study. C. K. was supported by Young Research Program, National Research Foundation of Korea(NRF) grant funded by the Korea government(MSIT) (No. RS-2023-00210210) and the Basic Science Research Program through the National Research Foundation of Korea (NRF) funded by the Ministry of Education (No. 2019R1A6A1A11051177). C. Yang was supported by the National Research Foundation of Korea(NRF) grant funded by the Korea government(MSIT) (No. 2021R1C1C1005700).

\section{Preliminary : Discrete Fourier analysis on $\T_h$}\label{sec:DFA}
\subsection{Notations and basic definitions}\label{sec:DFA1}
Throughout this paper, we deal with two different types of functions, i.e.,  functions defined on $\T$ and $\T_h$. To avoid possible confusion, we use the subscript $h$ for functions on $\T_h$ with no exception. For instance, the functions $u_h$, $v_h$, and $w_h$ are defined on $\T_h$, while $u$, $v$, and $w$ are defined on $\T$.

On the other hand, we assign small letters $x, y, z, ...$ to spatial variables regardless of whether they are on $\T$ or $\T_h$, if there is no confusion. Note that the subscript $h$ determines the space of the spatial variable.

Let $a, b \in \R_+$. We use $a \lesssim b$ when $a \le Cb$ for some $C >0$. Conventionally, $a \sim b$ means $a \lesssim b$ and $b \lesssim a$. Moreover, $a \lesssim_s b$ means that the implicit constant $C > 0$ depends on $s$.

For notational convenience, we may abbreviate the domain and codomain of a function in the norm. For example, for $f_h=f_h(x): \T_h\to \R$ (resp., $f=f(x):\T \to \R$), 
\[\|f_h\|_{L^p}=\|f_h\|_{L^p(\mathbb{T}_h)}\quad\left(\textup{resp., }\|f\|_{L^p}=\|f\|_{L^p(\mathbb{T})}\right)
\]
and for $F_h=F_h(t,x):\R\times \mathbb{T}_h\to\mathbb{R}$ (resp., $F=F(t,x):\R\times\mathbb{T}\to\mathbb{R}$), 
\[
\|F_h\|_{L_t^qL^r}=\|F_h\|_{L_t^q(\R:L^r(\mathbb{T}_h))}\quad\left(\textup{resp., }\|F\|_{L_t^qL^r}=\|F\|_{L_t^q(\R:L^r(\mathbb{T}))}\right).
\]
For $s \in \R$, we define the Sobolev space $H^s(\T_h)$ as the Hilbert space equipped with the norm
   \[ \|f_h\|_{H^s(\T_h)}:= \|\langle \nabla_h \rangle^sf_h\|_{L^2(\T_h)}.\]
\begin{remark}
In view of Definition \ref{def: discrete differentials}, $H^s$ norm is equivalent to $\displaystyle \|(1-\Delta_h)^{\frac{s}{2}}f_h\|_{L^2(\T_h)}$ or 
\begin{equation}\label{eq:Hs norm}
\left(\sum_{k \in (\T_h)^*}\langle k\rangle^{2s}\left|\widehat{f_h}(k)\right|^2\right)^{1/2}
\end{equation}
by Plancherel's theorem. In what follows, we use \eqref{eq:Hs norm} for $\|f_h\|_{H^s(\T_h)}$ 
\end{remark}
For a vector-valued function $(f_h,g_h): \T_h\to \R\times\R$ (resp. $(f,g): \T\to \R\times\R$) and $s \ge 0$, let $\mathbb{H}^s(\T_h)$ (resp. $\mathbb{H}^s(\T)$) denote the Sobolev space consisting of vector-valued functions whose components belong to $H^s(\T_h)$ (resp. $H^s(\T)$), i.e.,
\begin{equation}\label{eq:H^s for vector}
\begin{aligned}
&~{}\mathbb{H}^s(\T_h) := \{(f_h,g_h) : \|(f_h,g_h)\|_{\mathbb{H}^s(\T_h)}^2 = \|f_h\|_{H^s(\T_h)}^2 + \|g_h\|_{H^s(\T_h)}^2 < \infty\}.\\
&~{}(\mbox{resp.} \;\;  \mathbb{H}^s(\T) := \{(f,g) : \|(f,g)\|_{\mathbb{H}^s(\T)}^2 = \|f\|_{H^s(\T)}^2 + \|g\|_{H^s(\T)}^2 < \infty\})
\end{aligned}
\end{equation}

\subsection{Basic inequalities and Fourier transform on a lattice}
In this section, we summarize some useful inequalities for functions on $\T_h$. 
\begin{lemma}[Lemma 2.1 in \cite{HKNY2021}]
$\;$
\begin{enumerate}
\item (Duality) For $1< p,q < \infty$ satisfying $\frac{1}{p}+\frac{1}{q}=1$, we have 
\[  \| f_h\|_{L^p(\T_h)} = \sup_{\|g_h\|_{L^{q}(\T_h)}=1}h\left| \sum_{x\in\T_h} f_h(x)\overline{g_h(x)}\right|.\]
\item (H\"older's inequality) For $1\le p,q, r\le \infty$ satisfying $\frac1p+\frac1q=\frac{1}{r}$, we have 
\[\| f_hg_h\|_{L^r(\T_h)} \leq \| f_h\|_{L^p(\T_h)}\| g_h\|_{L^q(\T_h)}.\]
\end{enumerate}
\end{lemma}
\begin{lemma}[Properties of the Fourier transform on a periodic lattice, Lemma 2.3 in \cite{HKNY2021}]\label{prelim properties}
$\;$

\begin{enumerate}
\item (Inversion)
\[\mathcal{F}_h^{-1}\circ \mathcal{F}_h=\textup{Id}\textup{ on }L^2(\mathbb{T}_h),\quad\mathcal{F}_h\circ \mathcal{F}_h^{-1}=\textup{Id}\textup{ on }L^2((\mathbb{T}_h)^*).\]
\item (Plancherel's theorem) 
\[h\sum_{x\in\T_h}f_h(x)\overline{g_h(x)} = \sum_{k\in(\T_h)^*}\widehat{f_h}(k)\overline{\widehat{g_h}(k)}.\]
\item (Fourier transform of a product)
\begin{equation}\label{FT of product}
  \mathcal{F}_h(f_hg_h)(k)= \frac{1}{\sqrt{2\pi}}\sum_{k'\in(\T_h)^*}\widehat{f_h}(k')\widehat{g_h}(k-k'). 
\end{equation}
\end{enumerate}
\end{lemma}
\begin{remark}
On $\T$, one also has the same properties as Lemma \ref{prelim properties}:
\begin{enumerate}
\item (Inversion)
\[\mathcal{F}^{-1}\circ \mathcal{F}=\textup{Id}\textup{ on }L^2(\mathbb{T}),\quad\mathcal{F}\circ \mathcal{F}^{-1}=\textup{Id}\textup{ on }\ell^2(\Z).\]
\item (Plancherel's theorem) 
\[\int_{\T} f(x)\overline{g(x)} \; dx = \sum_{k\in\Z}\widehat{f}(k)\overline{\widehat{g}(k)}.\]
\item (Fourier transform of a product)
\begin{equation}\label{FT of product-conti}
  \mathcal{F}(fg)(k)= \frac{1}{\sqrt{2\pi}}\sum_{k'\in \Z}\widehat{f}(k')\widehat{g}(k-k'). 
\end{equation}
\end{enumerate}
\end{remark}

\begin{lemma}[Sobolev embedding, Lemma 2.7 in \cite{HKNY2021}]\label{lem:Bernstein inequality}
Suppose that $0 \le s \le \frac12$, $q \ge 2$, and $\frac1q = \frac12 - s$. Then, for any $\epsilon > 0$, we have
\[\|f_h\|_{L^q(\T_h)} \lesssim \|f_h\|_{H^{s+\epsilon}}.\]
\end{lemma}

\subsection{Linear interpolation}\label{sec:DisLin}
The linear interpolation operator $\mathcal L_h$ is defined by
\begin{equation}\label{def:linear interpolation_0}
  \begin{aligned}
  (\mathcal{L}_h f_h)(x):&= f_h(hk)+(\partial_h^+ f_h)(hk)\cdot(x-hk)\\
  &= f_h(hk)+\frac{f_h(hk+h)-f_h(hk)}{h}(x-hk)
  \end{aligned}
  \end{equation}
  for all $x\in [hk,hk+h)$, where $k\in(\T_h)^*$. Note that the linear interpolation converts a function on $\mathbb{T}_h$ into a continuous function on $\T$.
\begin{lemma}[Boundedness of linear interpolation, Lemma 5.1 in \cite{HKNY2021}]\label{Lem:discretization linearization inequality}
Let $0\leq s \leq 1$. Then, for any $f_h\in H^s(\mathbb{T}_h)$, we have 
\[
\| \mathcal{L}_h f_h \|_{{H}^s(\mathbb{T})} \lesssim \| f_h \|_{{H}^s(\mathbb{T}_h)} \label{ineq:discretization linearization inequality}\]
\end{lemma}
\begin{lemma}[Symbol of the linear interpolation, Lemma~5.4 in \cite{HKNY2021}]
The interpolation operator $\mathcal L_h$ is a Fourier multiplier in the sense that 
  \begin{equation}\label{linear interpolation multiplier}
    \widehat{\mathcal{L}_h f_h} (k)= \frac{4}{h^2k^2}\sin^2\left(\frac{hk}{2}\right) \widehat{f_h}\left(k-\frac{2\pi}{h}n\right), \quad \mbox{for} \;\; n\in\Z  \;\; \mbox{such that} \;\; \left|k-\frac{2\pi}{h}n\right|\le \frac{\pi}{h}.
   \end{equation}
\end{lemma}
\begin{lemma}\label{lem:LOWLh}
Let $0\le s\le 2$. Then for $f_h\in H^s(\T_h)$, we have 
\[\left(\sum_{|k| \le \frac{\pi}{h}}\left|\mathcal F(\mathcal{L}_h f_h) - \mathcal F_h(f_{h})\right|^2\right)^{\frac12} \lesssim h^s \|f_h\|_{H^{s}(\T_h)}.\]
\end{lemma}
\begin{proof}
By Taylor's remainder theorem, we know for $0 \le \alpha \le 2$ that
  \begin{equation}\label{Lh}
    \left| \frac{2}{h}\sin\left(\frac{hk}{2}\right) -k\right| \lesssim |k|\min\left(1,(h|k|)^2\right) \ls |k|(h|k|)^\alpha,
  \end{equation}
which yields
  \begin{equation}\label{Lh-1}
\left| \frac{4}{h^2k^2} \sin^2\left(\frac{hk}{2}\right)-1 \right| = \left| \frac{2}{hk} \sin\left(\frac{hk}{2}\right)-1 \right|\left| \frac{2}{hk} \sin\left(\frac{hk}{2}\right)+1 \right| \lesssim(h|k|)^\alpha,
  \end{equation}
for $0 \le \alpha \le 2$. Taking $\alpha = s$, we complete the proof.
\end{proof}
\begin{lemma}[Invariance of mean under the linear interpolation]\label{Lem:mean}
 We have 
\[  h\sum_{x\in \T_h} f_h(x)=\int_{-\pi}^{\pi}(\mathcal{L}_hf_h)(y)dy.\]
\end{lemma}
\begin{proof}
It follows from \eqref{def:linear interpolation_0} and \eqref{right derivative} that 
\[\begin{aligned}
  \int_{-\pi}^{\pi}(\mathcal{L}_hf_h)(y)dy
  &=\sum_{x\in \T_h} \bigg( h f_h(x) + (\partial_h^+ f_h)(x)\int_{x}^{x+h}(y-x)dy \bigg) \\  
  &=\sum _{x\in \T_h} \Big( h f_h(x) + \frac{h^2}{2}(\partial_h^+ f_h)(x) \Big)
  =\frac{h}{2}\sum_{x\in \T_h}\left(f_h(x+h)+f_h(x)\right).
\end{aligned}\]
\end{proof}
\section{Formal derivation and proof of Theorem \ref{main theorem}}\label{sec: outline}
This section is devoted to providing a short formal derivation of KdV system from the FPU system. For more details, see Section 2.1 in \cite{HKY2021}, but here we provide it for the self-containment.
\subsection{FPU to KdV}
The Duhamel principle allows us to rewrite \eqref{FPU} as an integral equation for the wave equation, and using the Euler formula, one has the following coupled system:
\begin{equation}\label{FPU system}
  \left\{\begin{aligned} \; &\partial_tr_h^{\pm}(t,x) \pm \frac{1}{h^2}\nabla_hr_h^{\pm}(t,x) \pm \frac{1}{4}\nabla_h\left( r_h(t,x)\right)^2=0,  \\ 
    &r_h(t,x)=r_h^+(t,x) +r_h^-(t,x),
  \end{aligned}  \right. \quad \big(r_h^+, r_h^-\big):\mathbb{R}\times \T_h\to\mathbb{R}\times\mathbb{R},
\end{equation}
with initial data
\[r_{h,0}^\pm= \frac12\left( r_{h,0}\mp h^2 \nabla_h^{-1}r_{h,1} \right),\]
where $\nabla_h$ is defined as in Definition \ref{def: discrete differentials}. 
\begin{remark}
By Remark \ref{rem:mean zero} (in particular, the mean-zero condition on $\partial_t r_h$), $\nabla_h^{-1}r_{h,1}$ is well-defined.
\end{remark}
From Duhamel's principle, \eqref{FPU system} is equivalent to the following integral equation: 
\begin{equation*}
  r_h^{\, \pm}(t) =e^{\mp\frac{t}{h^2}\nabla_h}r_{h,0}^{\, \pm} \mp\frac14\int_0^t e^{\mp\frac{(t-t')}{h^2}\nabla_h} \nabla_h\left(r_h(t')^2 \right) \; dt'.
  \end{equation*}
\begin{remark}\label{rem:Mean Conservation_1}
Note that for any function $f_h$ on $\T_h$, it is known that
\[h\sum_{x\in\T_h} \nabla_h f_h(x)=\mathcal{F}_h\Big( \nabla_h f_h \Big)(0)=0.\]
Using this observation, we obtain
\[\partial_t \Big( h\sum_{x\in\T_h}r_h^{\pm}(t,x)  \Big) = \mp h\sum_{x\in\T_h}\left(\frac{\nabla_h}{h^2}r_h^{\pm}(t,x) + \frac{\nabla_h}{4}\Big( r_h(t,x)\Big)^2\right) = 0,\]
which guarantees the conservation of the mean for $r_h^{\pm}$, that is,
\begin{equation}\label{eq:MC of r_h}
h\sum_{x\in\T_h}r_h^{\pm}(t,x) = \frac{h}2 \sum_{x\in\T_h} \left(r_{h,0} \mp h^2 \nabla_h^{-1} r_{h,1}\right)(x) \quad \text{ for all } t.
\end{equation}
On the other hand, by Remark \ref{rem:mean zero}, we know
\[h\sum_{x \in \T_h} r_h(0,x) = 0 \quad \mbox{and} \quad h\sum_{x \in \T_h} (\partial_t r_h)(0,x) = 0,\]
which, in addition to \eqref{eq:MC of r_h}, guarantees
\[  h\sum_{x\in\T_h}r_h^{\pm}(t,x) = 0 \quad \mbox{for all} \;\; t.\]
\end{remark}
Next, we introduce the phase translation operator by
\[u_h^\pm(t):=e^{\pm\frac{t}{h^2}\partial_h} r_h^{\, \pm}(t),\]
where $\partial_h$ is the discrete Fourier multiplier of the symbol $ik$ (see Definition \ref{def: discrete differentials}). Then, $u_h^{\pm}$ solves the following coupled integral equation:
\begin{equation}\label{coupled FPU'}
\begin{aligned}
  u_h^\pm(t)&= S_h^\pm(t)u_{h,0}^{\, \pm} \mp\frac14 \int_0^t S_h^\pm(t-t') \nabla_h \left(u_h^\pm(t')+e^{\pm\frac{2t'}{h^2}\partial_h}u_h^\mp(t') \right)^2 \; dt',
  \end{aligned}
\end{equation}
where $u_{h,0}^{\pm} = r_{h,0}^{\pm}$ and the linear FPU propagator $S_h^\pm(t)$ is defined by
  \begin{equation}\label{eq: linear FPU flow}
S_h^\pm(t) f_h := e^{\mp\frac{t}{h^2}(\nabla_h-\partial_h)} f_h,
  \end{equation}
for any function $f_h$ on $\T_h$. By construction, the FPU system \eqref{FPU} can be recovered from the equation \eqref{coupled FPU'} via
\[r_h(t,x)=e^{-\frac{t}{h^2}\partial_h}u_h^+(t,x)+e^{\frac{t}{h^2}\partial_h}u_h^-(t,x).\]
\begin{remark}\label{rem:Mean Conservation_3}
From Remark \ref{rem:Mean Conservation_1}, a direct computation gives 
\[h\sum_{x\in\T_h} u_h^{\pm}(t,x)=\sqrt{2\pi}\mathcal{F}_h(u_h^{\pm})(t,0)=\sqrt{2\pi}\mathcal{F}_h(r_h^{\, \pm})(t,0)=h\sum_{x\in\T_h}r_h^{\, \pm}(t,x)=0, \quad \text{ for all } t,\]
which asserts that $u_h^{\pm}$ satisfies the mean zero condition.
\end{remark}

\begin{remark}\label{rem:formal convergence of propagator}
The linear propagator $S_{h}^{\pm}(t)$ defined as in \eqref{eq: linear FPU flow} formally approximates to the Airy flow $S^{\pm}(t)$ given in \eqref{eq:propagator of KdV} below as $h \to 0$ in the Fourier space. Indeed, for each $k \in (\T_h)^*$, the Taylor expansion of the phase function reveals the following formal convergence
 \[ \mp\frac{1}{h^2}\left( \frac{2}{h}\sin\left(\frac{hk}{2}\right)-k\right) \rightarrow \pm\frac{1}{24}k^3 \quad \text{ as } h\to 0.\]
\end{remark}
\begin{remark}\label{rem:remove mixed nonlinearities}
Expanding the nonlinear terms in \eqref{coupled FPU'}, one identifies the mixed term and the oppositely moving wave, expressed as 
\[u_h^\pm(t')e^{\pm\frac{2t'}{h^2}\partial_h}u_h^\mp(t') \quad \mbox{and} \quad (e^{\pm\frac{2t'}{h^2}\partial_h}u_h^\mp(t'))^2,\]
respectively. These terms can be regarded as error terms as $h \to 0$ due to their asymptotic behavior. To illustrate, suppose the nonlinear solution $u_h^{\pm}(t)$ behaves similarly to linear solutions over a short time interval. In this case, the mixed term in \eqref{coupled FPU'} can be approximated by
\[\begin{aligned}
&~{}S_h^{\pm}(-t')\nabla_h (u_h^\pm(t')e^{\pm\frac{2t'}{h^2}\partial_h}u_h^\mp(t'))\\
  &\approx 
  S_h^{\pm}(-t')\nabla_h  \left(S_h^{\pm}(t')u_{h,0}^{\pm}(x)e^{\pm \frac{2t}{h^2}\partial_h} S_h^{\mp}(t')  u_{h,0}^{\mp}(x)\right) \\ 
  &=\sum_{k\in (\T_h)^*} e^{ikx}\frac{2i}{h}\sin\left(\frac{hk}{2}\right) \sum_{\substack{k_1 \in (\T_h)^* \\ k_1 \neq 0}}e^{\pm i\frac{t'}{h^2}\phi_h(k,k_1)}   \; \widehat{u_{h,0}^{\mp}}(k_1)\widehat{u_{h,0}^{\pm}}(k-k_1),
\end{aligned}\]
where 
\[\phi_h(k,k_1)= \frac{8}{h^3}\cos\left(\frac{hk}{4}\right)\sin\left(\frac{hk_1}{4}\right)\cos\left(\frac{h(k-k_1)}{4}\right).\]
By integrating this term over time, the Duhamel term produces a factor
\[\frac{1}{\phi_h(k,k_1)} \approx \frac{h^2}{k_1},\]
which vanishes as $h \to 0$.
Alternatively, this behavior can be explained by the fact that the exponential term $e^{\pm i\frac{t'}{h^2}\phi_h(k,k_1)}$ rapidly oscillates as $h \to 0$, causing the phase cancellation. This result confirms that the mixed term asymptotically vanishes, justifying its classification as an error term. Analogously, the oppositely moving wave can be analyzed and shown to behave in the same manner. Such observations are indeed demonstrated in Lemmas \ref{lem:bilinear2} and \ref{lem:bilinear3}.
\end{remark}
Inspired by Remark \ref{rem:remove mixed nonlinearities}, one further reduces the coupled system \eqref{coupled FPU'} as the following \textit{decoupled FPU system}: 
\begin{equation}\label{decoupled FPU'}
\begin{aligned}
  v_h^\pm(t)&= S_h^\pm(t)u_{h,0}^\pm\mp\frac14 \int_0^t S_h^\pm(t-t') \nabla_h\Big(v_h^{\pm}(t') \Big)^2 \; dt',
  \end{aligned}
  \end{equation}
where $v_h^{\pm}=v_h^{\pm}(t,x):\R\times\T_h\rightarrow\R$ satisfying the mean zero condition
\[h \sum_{x \in \T_h}v_h^{\pm}(t,x) = 0, \quad \mbox{for all} \;\; t.\]
\begin{remark}
Both the coupled \eqref{coupled FPU'} and the decoupled \eqref{decoupled FPU'} systems are well-posed in $L^2(\T_h)$ (see Proposition~\ref{prop:LWP} below), but the existence time depends on the number of lattice points $N$. This dependency poses a non-trivial challenge in demonstrating the continuum limit of the FPU system to the KdV system. Nevertheless, one can capture the dispersive properties (see Sections \ref{subsec:linear estimates} and \ref{subsec:Bilinear estimates}) for the FPU solutions, and this observation enables us to establish the well-posedness of both \eqref{coupled FPU'} and \eqref{decoupled FPU'} independently of the number of lattice points $N$ (see Proposition \ref{prop:uniform bound}).
\end{remark}
As a consequence, we have
  \begin{proposition}\label{Prop:Coupled to Decoupled}
Let $h\in(0,1]$, $0< s\le 1$, and $R > 0$ be given. Suppose that 
\[\sum_{x\in\T_h}u_{h,0}^{\pm}(x)=0 \quad \mbox{and} \quad \sup_{h\in(0,1]}  \left\|\left( u_{h,0}^{+},u_{h,0}^{-}\right)\right\|_{\mathbb H^s(\T_h)} \le R.\]
Then, there exist $C(R) >0$ and $T(R)>0$ independent on $h$ such that the following holds: Let $(u_h^+(t),u_h^-(t))\in C_t([-T,T]:\mathbb H^s(\T_h))$ (resp., $(v_h^+(t),v_h^-(t))\in C_t([-T,T]:\mathbb H^s(\T_h))$) be the solution to the coupled FPU \eqref{coupled FPU'} (resp., decoupled FPU \eqref{decoupled FPU'}) with an initial data $(u_{h,0}^+, u_{h,0}^-)$. Then,
\[      \left\| u_h^{\pm}(t) - v_h^{\pm}(t) \right\|_{C_t([-T,T]:L^2(\T_h))} \le  C(R)h^{s}.\]
    \end{proposition}
Moreover, due to Remark \ref{rem:formal convergence of propagator}, the decoupled system \eqref{decoupled FPU'} can be expected to converge to the following counter-propagating KdV system
\begin{equation}\label{kdv integral form}
w^{\pm}(t) = S^{\pm}(t)(\mathcal{L}_h u_{h,0}^{\pm} ) \mp \frac14 \int_0^t S^{\pm}(t-t') \partial_x\Big(w^{\pm}(t')\Big)^2 dt',
\end{equation}
where $w^{\pm}=w^{\pm}(t,x):\mathbb{R}\times\mathbb{T}\rightarrow \mathbb{R}$, $\mathcal L_h$ is the linear interpolation operator defined as in \eqref{def:linear interpolation_0}, and 
\begin{equation}\label{eq:propagator of KdV}
S^{\pm}(t)=e^{\mp\frac{t}{24}\partial_x^3}
\end{equation}
denotes the Airy flow. 
\begin{remark}
Note that, from the conservation of the mean for the KdV equation and Lemma \ref{Lem:mean}, the solution $w^{\pm}$ satisfies the mean-zero condition
\[
\int_{\T} w^{\pm}(t,x) \; dx = 0 \quad \mbox{for all} \; t.
\]
\end{remark}
Precisely, one has
\begin{proposition}\label{prop:from decoupled to kdv}
Let $h\in(0,1]$, $0< s\le 1$, and $R > 0$ be given. Suppose that 
\[\sum_{x\in\T_h}u_{h,0}^{\pm}(x)=0 \quad \mbox{and} \quad \sup_{h\in(0,1]}  \big\|\big( u_{h,0}^{+},u_{h,0}^{-}\big)\big\|_{\mathbb H^s(\T_h)} \le R.\]
Then, there exist $0 < h_0 = h_0(R) < 1$ and $T(R)>0$ sufficiently small such that the following holds: 

Let $(v_h^+(t),v_h^-(t)) \in C_t([-T,T]:\mathbb H^s(\T_h))$ (resp., $(w^+(t),w^-(t)) \in C_t([-T,T]:\mathbb H^s(\T))$) be the solution to the decoupled FPU \eqref{decoupled FPU'} (resp., KdV \eqref{kdv integral form}) with an initial data $(u_{h,0}^+, u_{h,0}^-)$  (resp., $(\mathcal L_hu_{h,0}^{+},\mathcal L_hu_{h,0}^{-})$). Then, there exists $C(R) > 0$ independent on $h$ such that
\[    \sup_{t\in [-T,T]}\left\| \mathcal{L}_hv_h^{\pm}(t) - w^{\pm}(t) \right\|_{L^2(\T)} \le  C(R)  h^{\frac{2s}{5}},\]
whenever $0 < h \le h_0$.
  \end{proposition}

  \begin{remark}
   The convergence rate $h^{\frac{2s}{5}}$ in Proposition~\ref{prop:from decoupled to kdv} is worse than that in Proposition~\ref{Prop:Coupled to Decoupled}. However, it seems to be optimal in the sense that the Airy flows can be approximated by the linear FPU flows with this convergence rate (see \eqref{difference2-1:symbol difference} below).
  \end{remark}
\subsection{Well-posedness of FPU}
\begin{proposition}[Local well-posedness of coupled and decouopled FPUs] \label{prop:LWP} Let $h\in(0,1]$ be fixed. For each 
  $$u_{h,0}^+,u_{h,0}^- \in L^2(\T_h),$$
there exists a time $T=T(h,\|u_{h,0}^+\|_{L^2(\T_h)},\|u_{h,0}^-\|_{L^2(\T_h)})$ for which a unique solution $(u_h^+,u_h^-)\in C([-T,T]:\mathbb{H}^0(\T_h))$ to the coupled FPU \eqref{coupled FPU'} (resp., $(v_h^+,v_h^-)\in C([-T,T]:\mathbb{H}^0(\T_h))$ to the decoupled FPU \eqref{decoupled FPU'}) exists.
\end{proposition}
\begin{proof}
We only deal with the coupled FPU, since the decoupled FPU follows similarly. Define a nonlinear map $\Phi=(\Phi^+,\Phi^-)$ by 
\[  \Phi^{\pm}(u_h^+,u_h^-)
  :=S_h^\pm(t)u_{h,0}^\pm\mp\frac14 \int_0^t S_h^\pm(t-t') \nabla_h \left(u_h^\pm(t')+e^{\pm\frac{2t'}{h^2}\partial_h}u_h^\mp(t')\right)^2dt'.\]
Then, for $T>0$ (to be chosen later), 
\[  \|   \Phi(u_h^+,u_h^-) \|_{C_t([-T,T]:\mathbb{H}^0(\T_h))} \ls \| u_{h,0}^{\pm}\|_{L^2(\T_h)} + T \left\| \nabla_h \left(u_h^\pm+e^{\pm\frac{2t}{h^2}\partial_h}u_h^\mp\right)^2\right\|_{C_t([-T,T]:L^2(\T_h))}.\]
Note by the continuous embedding $\ell^p \subset \ell^q$ that
\begin{equation}\label{reversed L^p-L^q inequality}
\|u_h\|_{L^q(\T_h)}\lesssim h^{-(\frac{1}{p}-\frac{1}{q})}\|u_h\|_{L^p(\T_h)}\quad\mbox{for all} \;\; q> p.
\end{equation}
Note also that the differential operator is bounded on a lattice
\begin{equation}\label{discrete operator is bounded}
  \| \nabla_h u_h \|_{L^2(\T_h)}\ls h^{-1} \|u_h\|_{L^2(\T_h)}.
\end{equation}
Using \eqref{discrete operator is bounded} and \eqref{reversed L^p-L^q inequality}, we estimate the nonlinear term by
\[\begin{aligned}
  \left\|\nabla_h \left(u_h^\pm+e^{\pm\frac{2t}{h^2}\partial_h}u_h^\mp\right)^2\right\|_{L^2(\T_h)}
  &\le 
h^{-1}\left\| \left(u_h^\pm+e^{\pm\frac{2t}{h^2}\partial_h}u_h^\mp\right)^2\right\|_{L^2(\T_h)} \\ 
 &\le 2h^{-1} \left(  \left\|  u_h^\pm\right\|_{L^4(\T_h)}^2 
 + \left\| e^{\pm\frac{2t}{h^2}\partial_h}u_h^\mp\right\|_{L^4(\T_h)}^2 \right)\\
 &\le 2Ch^{-\frac32}\left(  \left\|  u_h^\pm\right\|_{L^2(\T_h)}^2 
  + \left\| e^{\pm\frac{2t}{h^2}\partial_h}u_h^\mp\right\|_{L^2(\T_h)}^2 \right)\\
  &\le 2Ch^{-\frac32} \left\| (u_h^\pm,u_h^\mp)\right\|_{\mathbb{H}^0(\T_h)}^2,
\end{aligned}\]
where the $C > 0$ appears in the estimate \eqref{reversed L^p-L^q inequality} for $(p,q) = (2,4)$. Thus, we obtain 
\[\|   \Phi(u_h^+,u_h^-) \|_{C_t([-T,T]:\mathbb{H}^0(\T_h))} 
  \le \| (u_{h,0}^+,u_{h,0}^-)\|_{\mathbb H^0(\T_h)} +2C h^{-\frac32} T\left\| (u_h^\pm,u_h^\mp)\right\|_{C_t([-T,T]:\mathbb{H}^0(\T_h))}^2.\]
Similarly, one can verify that 
\[\begin{aligned}
 & \|   \Phi(u_{1,h}^+,u_{1,h}^-) -\Phi(u_{2,h}^+,u_{2,h}^-)  \|_{C_t([-T,T]:\mathbb{H}^0(\T_h))} \\
\le&~{} 4C h^{-\frac32} T
  \left(\sum_{j=1}^2 \left\| (u_{j,h}^\pm,u_{j,h}^\mp)\right\|_{C_t([-T,T]:\mathbb{H}^0(\T_h))}\right)
  \left\| (u_{1,h}^\pm-u_{2,h}^\pm,u_{1,h}^\mp-u_{2,h}^\mp) \right\|_{C_t([-T,T]:\mathbb{H}^0(\T_h))}.
\end{aligned}\]
Taking $T= \frac{h^\frac32}{16C\left\| (u_{h,0}^\pm,u_{h,0}^\mp)\right\|_{\mathbb{H}^0(\T_h)}}$, we prove that $\Phi$ is contractive on the ball in $C_t([-T,T]:\mathbb{H}^0(\T_h))$ of radius $2\left\| (u_{h,0}^\pm,u_{h,0}^\mp)\right\|_{\mathbb{H}^0(\T_h))}$ centered at zero. 
\end{proof}

\subsection{Proof of Theorem \ref{main theorem}}
We end this section with a brief proof of Theorem \ref{main theorem}, assuming Propositions \ref{Prop:Coupled to Decoupled} and \ref{prop:from decoupled to kdv}. A direct computation shows 
\[\begin{aligned}
&~{}\left\|( \mathcal{L}_hr_h)(t,\cdot)- w^+(t, \cdot-\tfrac{t}{h^2})- w^-(t, \cdot+\tfrac{t}{h^2})\right\|_{L^2(\mathbb{T})}\\
 \le&~{} \|\mathcal L_h(e^{-\frac{t}{h^2}\partial_h}u_h^+)(t,\cdot) - e^{-\frac{t}{h^2}\partial_x}w^+(t,\cdot)\|_{L^2(\T)} + \|\mathcal L_h(e^{\frac{t}{h^2}\partial_h}u_h^-)(t,\cdot)- e^{\frac{t}{h^2}\partial_x}w^-(t,\cdot)\|_{L^2(\T)}.
\end{aligned}\]
Now we deal only with the "$+$" part, as the rest follows analogously.

By Lemma \ref{Lem:discretization linearization inequality}, it follows that
\[\begin{aligned}
&~{} \|\mathcal L_h(e^{-\frac{t}{h^2}\partial_h}u_h^+)(t) - e^{-\frac{t}{h^2}\partial_x}w^+(t)\|_{L^2(\T)}\\
\le&~{}\|\mathcal L_h(e^{-\frac{t}{h^2}\partial_h}u_h^+)(t) - \mathcal L_h(e^{-\frac{t}{h^2}\partial_h}v_h^+)(t)\|_{L^2(\T)} + \|\mathcal L_h(e^{-\frac{t}{h^2}\partial_h}v_h^+)(t) - e^{-\frac{t}{h^2}\partial_x}w^+(t)\|_{L^2(\T)}\\
\le&~{} \|u_h^+(t)-v_h^+(t)\|_{L^2(\T_h)} + \|\mathcal L_h(v_h^+)(t) - w^+(t)\|_{L^2(\T)} + \|\mathcal L_h(e^{-\frac{t}{h^2}\partial_h}v_h^+)(t) - e^{-\frac{t}{h^2}\partial_x}\mathcal L_h(v_h^+)(t)\|_{L^2(\T)}.
\end{aligned}\]
Assuming Propositions \ref{Prop:Coupled to Decoupled} and \ref{prop:from decoupled to kdv}, the first two terms can be bounded by $h^{\frac{2s}{5}}$. For the remaining term, we split the $L^2$-norm into high- and low-frequency components as follows:
\[\begin{aligned}
\|\mathcal L_h(e^{-\frac{t}{h^2}\partial_h}v_h^+)(t) - e^{-\frac{t}{h^2}\partial_x}\mathcal L_h(v_h^+)(t)\|_{L^2(\T)} \le&~{} \|P_{\ge \frac{\pi}{h}}\left(\mathcal L_h(e^{-\frac{t}{h^2}\partial_h}v_h^+) - e^{-\frac{t}{h^2}\partial_x}\mathcal L_h(v_h^+)\right)(t)\|_{L^2(\T)}\\ 
&~{}+ \|P_{\le \frac{\pi}{h}}\left(\mathcal L_h(e^{-\frac{t}{h^2}\partial_h}v_h^+) - e^{-\frac{t}{h^2}\partial_x}\mathcal L_h(v_h^+)\right)(t)\|_{L^2(\T)}.
\end{aligned}\]
By Proposition \ref{prop:exterior estimates} and Corollary \ref{Cor:Uniform bounds} below, the first term is bounded by $h^s$. Note by \eqref{linear interpolation multiplier} that the second term vanishes. Collecting all, we complete the proof.

\section{Coupled and decoupled FPU systems}\label{sec:part1}
\subsection{Function spaces for solutions}
In this subsection, we define resolution spaces for the FPU system. 
We introduce the $X^{s,b}$ spaces, introduced by Bourgain \cite{B-1993Sch} and further developed by Kenig, Ponce, and Vega \cite{KPV-1996} and Tao \cite{T-2001}, adapted to our lattice setting.

First, we define the function space in a general setting. In subsequent applications, the spatial domain $\Xi$ will be either the torus $\mathbb{T}$ or the periodic lattice domain $\mathbb{T}_h$, and the associated symbol $P$ is chosen according to the model considered. Since the following are stated in a general setting, they can be applied in a unified way. We refer to \cite{T-2006} for the details and proofs.

\begin{definition}[$X^{s,b}$ spaces]\label{def:Xsb}
Let $\Xi$ be either $\mathbb{T}$ or $\mathbb{T}_h$. Let $p$ be a real-valued continuous function. For $s,b\in\R$, we define the $X_{\tau=\rho(k)}^{s,b}(\R \times \Xi)$ spaces ($X_{\tau=\rho(k)}^{s,b}$ in short) as the completion of $\mathcal S(\R\times \R)$ or $\mathcal S(\R\times h\Z)$ with respect to the norm 
\[\| u \|_{X_{\tau=\rho(k)}^{s,b}}^2 := \sum_{k\in (\Xi)^*}\int_{\R}\la k \ra^{2s} \la \tau -\rho(k) \ra^{2b} |\widetilde{u}(\tau,k)|^2  \;d\tau,\]
where $\tilde{u}$ denotes the space-time Fourier transform of $u$ defined by\footnote{In particular, when $\Xi = \T_h$, $\widetilde{u}$ (as in Definition \ref{def:Xsb}) is defined by 
\[\widetilde{u}_h(\tau, k) = \frac{h}{2\pi}\sum_{x \in \T_h} \int_{\R} e^{-it\tau} e^{-ix k} u_h(t,x) \; dt.\]
}
\[\widetilde{u} (\tau, k) = \frac{1}{2\pi}\int_{\R \times \Xi} e^{-it\tau} e^{-ix k} u(t,x) \; dt dx\]
and $(\Xi)^*$ is the Pontryagin dual space of $\Xi$, i.e, $(\T)^* = \Z$ and $(\T_h)^* = \Z\setminus (2N\Z)$. We particularly use $\mathcal F_{t,x}$ and $\mathcal F_{t,h}$ for $\widetilde{u}$ and $\widetilde{u_h}$, respectively, to avoid confusion.
\end{definition}

\begin{lemma}[\cite{T-2006,ET2016}]\label{lem:properties}
  Let $s,b \in \R$ and $X_{\tau=\rho(k)}^{s,b}$ spaces be defined as in Definition \ref{def:Xsb}. Let $\eta \in C_0^\infty(\R)$ be a cut-off function. Then, the following properties hold true:
  \begin{enumerate}
  \item (Nesting) $X_{\tau=\rho(k)}^{s',b'} \subset X_{\tau=\rho(k)}^{s,b}$ whenever $s \le s',~b\le b'$.
  \item (Duality) $X_{-\tau=p(-k)}^{-s,-b}$ is the dual space of $X_{\tau=\rho(k)}^{s,b}$.
  \item (Well-defined for linear solutions) For any $f \in H^s(\Xi)$, we have\footnote{ When $\Xi=\T$, $\partial = \partial_x$, otherwise $\partial=\partial_h$.}
  \[\left\|\eta(t) e^{ itp(-i\partial)} f \right\|_{X_{\tau=\rho(k)}^{s,b}} \ls_{\eta,b} \| f\|_{H^s(\Xi)}.\]
  \item (Transference principle) Let $X$ be a Banach space for which the inequality
  \[ \big\|e^{it\tau_0}e^{itp(-i\partial)}f\big\|_{X} \lesssim \|f\|_{H^s(\Xi)}\]
  holds for all $f \in H^s(\Xi)$ and $\tau_0 \in \R$. If, additionally, $b > \frac12$, then we have the embedding
  \[\|u\|_{X} \lesssim_b \|u\|_{X_{\tau=\rho(k)}^{s,b}}.\]
  In particular, we have
  \begin{equation}\label{eq:embedding}
  \|u\|_{C_tH^s(\Xi)} \lesssim \|u\|_{X_{\tau=\rho(k)}^{s,b}}.
  \end{equation}
  \item (Stability with respect to time localization) Let $0 < T < 1$, $b > \frac12$ and $f \in H^s(\Xi)$. We have
  \[\big\|\eta(\tfrac{t}{T}) e^{itp(-i\partial)} f \big\|_{X_{\tau=\rho(k)}^{s,b}} \ls_{\eta, b} T^{\frac12 - b} \|f\|_{H^s(\Xi)}.\]
  If $-\frac12<b' \le b<\frac12$, then we have
  \begin{align}\label{ineq:time localization}
    \big\|\eta(\tfrac{t}{T}) u_h\big\|_{X_{\tau=\rho(k)}^{s,b'}} \ls_{\eta, b, b'} T^{b-b'} \| u_h\|_{X_{\tau=\rho(k)}^{s,b}}.
  \end{align}
\item (Inhomogeneous estimate) Let $-\frac12 < b' \le 0 \le b \le b'+1$ and $T \le 1$. Then, we have 
\[  \left\| \eta(\tfrac{t}{T}) \int_0^t e^{ i(t-t')p(-i\partial)} F(t')dt' \right\|_{X_{\tau=\rho(k)}^{s,b}}  \ls_{\eta, b,b'} T^{1-b-b'}\| F\|_{X_{\tau=\rho(k)}^{s,b'}}.\]
  \end{enumerate}
  \end{lemma}
From Lemma \ref{lem:properties}, it is known that $X^{s,b}$ for some $b > \frac12$ seems be an appropriate auxiliary function space for both FPU and KdV. However, the nonlinearity of KdV cannot be controlled in $X^{s,b}$ except for $b = \frac12$ (see Theorem 1.2 in \cite{KPV-1996})\footnote{An analogous argument to that in Theorem 1.2 of \cite{KPV-1996} may yield the same result for the FPU system, but proving it here is unnecessary.}. In order to overcome this issue, we use an additional auxiliary function space $Y_{\tau=\rho(k)}^{s}=Y_{\tau=\rho(k)}^{s}(\R\times \Xi)$ equipped with the norm (see, for instance, \cite{B-1993KdV})
\[ \| u\|_{Y_{\tau=\rho(k)}^{s}}:= \|u\|_{X_{\tau=\rho(k)}^{s,\frac12}}+\|\la k\ra^s \widetilde{u}(\tau,k)\|_{\ell_{k\in (\Xi)^*}^2L_{\tau\in \R}^1}.\]
We list well-known properties of $Y_{\tau=\rho(k)}^{s}$. For the proofs, we refer to \cite{ET2016}
\begin{lemma}\label{Lem:Y}
We have 
\begin{enumerate}
\item (Embedding) For $u\in C_t^0H_x^s(\R\times\Xi)$, 
\begin{equation}\label{Embedding Y}
  \|  u \|_{C_tH_x^s} \ls \| u\|_{Y_{\tau=\rho(k)}^{s}}.
\end{equation}
\item (Well-defined for linear solutions) For any $f \in H^s(\Xi)$, we have 
\[  \left\| \eta(\tfrac{t}{T})e^{itp(-i\partial)} f \right\|_{Y_{\tau=\rho(k)}^{s}}
  \ls \|f\|_{H^s(\Xi)}.\]
\item (Inhomogeneous estimate)
For $Z_{\tau=\rho(k)}^s$ norm defined by 
\[  \| u\|_{Z_{\tau=\rho(k)}^s}:=\| u\|_{X_{\tau=\rho(k)}^{s,-\frac12}} +\left\| \langle k\rangle^s \langle \tau -\rho(k) \rangle^{-1} \widetilde{u}(\tau,k)\right\|_{\ell_{k\in \Xi^*}^2 L_{\tau\in\R}^1},\]
we have 
\[\left\|\eta(\tfrac{t}{T})\int_0^t e^{ i(t-t')p(-i\partial)}F(t')dt'\right\|_{Y_{\tau=\rho(k)}^s} \ls \| F\|_{Z_{\tau=\rho(k)}^s}.\]
\end{enumerate}
\end{lemma}
For the particular FPU solutions $u_{h}^{\pm}$, its corresponding $X^{s,b}$ space is defined by
\begin{definition}[$X_{h,\pm}^{s,b}$ spaces]\label{def:Xsb_h}
  For $s,b\in \R$, we define the discrete $X^{s,b}$ spaces $X_{h,\pm}^{s,b}=X_{h,\pm}^{s,b}(\R\times \T_h)$ as 
\[  X_{h,\pm}^{s,b} := X_{\tau=\pm s_h(k)}^{s,b},\]
where 
\[s_h(k) = \frac{1}{h^2} \left(k - \frac{2}{h}\sin\left(\frac{hk}{2}\right)\right).\]
Analogously, $Y_{h,\pm}^s$ and $Z_{h,\pm}^s$ can be defined by replacing $X_{\tau=\rho(k)}^{s,\frac12}$ and $X_{\tau=\rho(k)}^{s,-\frac12}$ by $X_{h,\pm}^{s,\frac12}$ and $X_{h,\pm}^{s,-\frac12}$, respectively. Moreover, in a standard manner, we localize these spaces in time for $T \in (0,1]$ as
\[Y_{h,\pm}^{s,T} = \left\{f_h : [-T,T] \times \T_h : \|f_h\|_{Y_{h,\pm}^{s,T}} = \inf\limits_{g_h = f_h \;\; \textit{\emph{in}} \;\; [-T, T] \times \T_h} \|g_h\|_{Y_{h,\pm}^s}\right\},\]
\[Z_{h,\pm}^{s,T} = \left\{f_h : [-T,T] \times \T_h : \|f_h\|_{Z_{h,\pm}^{s,T}} = \inf\limits_{g_h = f_h \;\; \textit{\emph{in}} \;\; [-T, T] \times \T_h} \|g_h\|_{Z_{h,\pm}^s}\right\}.\]
\end{definition}

\subsection{Linear estimates}\label{subsec:linear estimates}
\begin{proposition}[$L^4$-Strichartz estimates]\label{prop:L4}
Let $0 < h \le 1$.  For $b > \frac13$, we have
\[  \norm{f_h}_{L_t^4(\R:L^4(\T_h))} \lesssim \norm{f_h}_{X_{h,\pm}^{0,b}(\R \times \T_h)}\]
  for any function $f_h$ in $X_{h,\pm}^{0,b}(\R \times \T_h)$.
  \end{proposition}

  \begin{lemma}\label{lem:cos}
  Let $\alpha \in \R$ and $x \in [-\frac{\pi}{2}, \frac{\pi}{2}]$. Then, the following statements are valid:
  \begin{enumerate}
  \item $|1- \cos x - \alpha| \ge \frac{x^2}{4}$, if $\alpha \le 0$.
  \item $|1- \cos x - \alpha| \ge \frac{1}{4}|x-\beta||x+\beta|$, if $0 < \alpha < 1$, where $\beta = \cos^{-1}(1-\alpha) \in \left(0, \frac{\pi}{2}\right)$
  \item $|1- \cos x - \alpha| \ge \frac{1}{4}\left|x-\frac{\pi}{2}\right|\left|x+\frac{\pi}{2}\right|$, if $\alpha \ge 1$.
  \end{enumerate}
  \end{lemma}
  
  \begin{proof} 
  When $\alpha \le 0$, $|1-\cos x -\alpha| \ge 1-\cos x \ge 0$ in $[-\frac{\pi}{2}, \frac{\pi}{2}]$. Let $g(x) = 1 - \cos x - \frac14x^2$. Note that $g(0) = 0$. Since $g'(x) = \sin x - \frac12 x \ge 0$ for $x \in [0, \frac{\pi}{2}]$, we know $g$ is increasing, thus $g(x) \ge 0$ for $x \in [0, \frac{\pi}{2}]$. Therefore, we have Item (1) due to the evenness of $g$. 
  
  \medskip
  
  \noindent When $\alpha \ge 1$, $|1-\cos x -\alpha| = \cos x + (\alpha-1) \ge \cos x \ge 0$. Now we set $g(x) = \cos x - \frac14\left(\frac{\pi^2}{2} - x^2\right)$. Note that $g(0) = 1 - \frac{\pi^2}{16} \ge 0$ and $g(\frac{\pi}{2}) = 0$. Since $g'(x) = -\sin x +\frac12x \le 0$ for $x \in [0, \frac{\pi}{2}]$, we know $g$ is decreasing, thus $g(x) \ge 0$ for $x \in [0, \frac{\pi}{2}]$. Therefore, we have Item (3) due to the evenness of $g$.  
  
  \medskip
  
  \noindent When $0 < \alpha <1$, let $\beta = \cos^{-1}(1-\alpha) \in (0, \frac{\pi}{2})$. Let 
  \[\begin{aligned}
  g(x) =&~{} |1-\cos x -\alpha| -  \frac{1}{4}|x-\beta||x+\beta| \\
  =&~{} \begin{cases} -1 + \cos x  + \alpha + \frac14(x-\beta)(x+\beta), \quad &\mbox{if} \;\; 0 \le x \le \beta,\\ 1- \cos x - \alpha - \frac14(x-\beta)(x+\beta), \quad &\mbox{if} \;\; \beta \le x \le \frac{\pi}{2}.\end{cases}
  \end{aligned}\]
  We have already shown that $g(0) = \alpha - \frac{\beta^2}{4} = 1- \cos \beta - \frac{\beta^2}{4} > 0$ and $g(\frac{\pi}{2}) = 1-\alpha - \frac14\left(\frac{\pi^2}{4} - \beta^2\right) = \cos \beta - \frac14\left(\frac{\pi^2}{4} - \beta^2\right) > 0$ for $\beta \in \left(0, \frac{\pi}{2}\right)$. We also note that $g(\beta) = 0$. Since
  \[g'(x) = \begin{cases} -\sin x + \frac12 x < 0, \quad &\mbox{if} \;\; 0 \le x \le \beta,\\ \sin x - \frac12x >0 , \quad &\mbox{if} \;\; \beta \le x \le \frac{\pi}{2},\end{cases}\]
  we conclude that $g(x) \ge 0$ for $x \in [0, \frac{\pi}{2}]$. Therefore, we have Item (2) due to the evenness of $g$. 
  \end{proof}

  \begin{proof}[Proof of Proposition \ref{prop:L4}]
 We only prove "+" part, and the other part follows analogously. Let $f_h = f_{h,1} + f_{h,2}$, where 
  \[\widehat{f}_{h,1}(k) = 0, \quad \mbox{if} \quad |k| > 1.\]
  Note that $|\{k \in (\mathbb{T}_h)^* : k \in \supp (\widehat{f}_1)\}| = 3$. Since $(f_h)^2 \le 2(f_{h,1})^2 + 2(f_{h,2})^2$, it suffices to treat $\norm{(f_{h,1})^2}_{L_t^2(\R:L^2(\T_h))}$ and $\norm{(f_{h,2})^2}_{L_t^2(\R:L^2(\T_h))}$ separately. Moreover, let $f_{h,2} = f_{h,2,1} + f_{h,2,2}$, where
  \[\widehat{f}_{h,2,1}(k) = 0 \quad \mbox{if} \quad k > 1.\]
  Then, similarly as before, $\norm{(f_{h,2})^2}_{L^2}^2 \le 2\norm{(f_{h,2,1})^2}_{L^2}^2 + 2\norm{(f_{h,2,2})^2}_{L^2}^2$ and $\norm{(f_{h,2,1})^2}_{L^2} = \norm{(\overline{f_{h,2,1}})^2}_{L^2} = \norm{(f_{h,2,2})^2}_{L^2}$ which enables us to assume that $\supp (\widehat{f}_{h,2}) = \{2,3, \cdots, N-1\}$.
  
  \medskip
  
  \textbf{$(f_{h,1})^2$ case.} A computation gives
  \begin{equation}\label{eq:L^4_1}
  \norm{(f_{h,1})^2}_{L_t^2(\R:L^2(\T_h))}^2 \le \sum_{k \in (\mathbb{T}_h)^*} \int_{\R} \left| \sum_{k_1 \in (\mathbb{T}_h)^*}\int_\R  |\wt{f}_{h,1}(\tau_1,k_1)||\wt{f}_{h,1}(\tau -\tau_1,k-k_1)|\; d\tau_1 \right|^2 \; d\tau .
  \end{equation}
  From the support property, the right-hand side of \eqref{eq:L^4_1} vanishes unless $|k| \le 2$. Let 
  \[\wt{F}_{h,1}(\tau, k) = \bra{\tau - s_h(k)}^{b}|\wt{f}_{h,1}(\tau,k)|.\]
  By the Cauchy-Schwarz inequality and the Minkowski inequality, we see that for $b > \frac14$,
  \[\begin{aligned}
  \mbox{RHS of } \eqref{eq:L^4_1} \lesssim&~{}\sum_{ |k| \le 2} \int_{\R} \Bigg( \sum_{|k_1| \le 1}\left(\int_\R \bra{\tau-\tau_1 - s_h(k-k_1)}^{-2b}\bra{\tau_1- s_h(k_1)}^{-2b} \; d\tau_1 \right)^{\frac12}\\
  & \hspace{7em} \times \left(\int_\R  |\wt{F}_1(\tau_1,k_1)|^2|\wt{F}_1(\tau -\tau_1,k-k_1)|^2\; d\tau_1\right)^{\frac12} \Bigg)^2 \; d\tau \\
  \lesssim&~{}\sum_{|k| \le 2} \Bigg(\sum_{|k_1|\le 1}\left(\int_{\R^2}  |\wt{F}_{h,1}(\tau_1,k_1)|^2|\wt{F}_{h,1}(\tau -\tau_1,k-k_1)|^2\; d\tau_1 d \tau \right)^{\frac12}\Bigg)^2\\
  \lesssim&~{} \norm{f_{h,1}}_{X_{h,+}^{0,b}}^4 \lesssim \norm{f_h}_{X_{h,+}^{0,b}}^4.
  \end{aligned}\]
  
  \medskip
  
  \textbf{$(f_{h,2})^2$ case.} Analogous to \eqref{eq:L^4_1}, we have
  \[
  \begin{aligned}
  \norm{(f_{h,2})^2}_{L^2(\R \times \T_h)}^2 \le&~{} \sum_{k \in (\mathbb{T}_h)^*} \int_{\R} \left| \sum_{k_1 \in (\mathbb{T}_h)^*} \int_{\R} |\wt{f}_{h,2}(\tau_1,k_1)||\wt{f}_{h,2}(\tau -\tau_1,k-k_1)|\; d\tau_1 \right|^2 \; d\tau\\
  =&~{}\sum_{\substack{k \in (\mathbb{T}_h)^* \\ k  > 1}} \int_{\R} \left| \sum_{\substack{k_1 \in (\mathbb{T}_h)^* \\ k-k_1, k_1 >1}}\int_\R  |\wt{f}_{h,2}(\tau_1,k_1)||\wt{f}_{h,2}(\tau -\tau_1,k-k_1)|\; d\tau_1 \right|^2 \; d\tau =:I.
  \end{aligned}\]
  Similarly as \eqref{eq:L^4_1}, we have
  \[\begin{aligned}
  I \lesssim&~{} \sum_{\substack{k \in (\mathbb{T}_h)^* \\ k > 1}} \int_{\R} \Bigg( \Bigg(\sum_{\substack{k_1 \in (\mathbb{T}_h)^* \\ k-k_1, k_1 > 1}}\int_\R \bra{\tau-\tau_1 - s_h(k-k_1)}^{-2b}\bra{\tau_1- s_h(k_1)}^{-2b} \; d\tau_1 \Bigg)^{\frac12}\\
  & \hspace{7em} \times \Bigg(\sum_{\substack{k_1 \in (\mathbb{T}_h)^* \\ k-k_1, k_1 > 1}}\int_\R  |\wt{F}_{h,2}(\tau_1,k_1)|^2|\wt{F}_{h,2}(\tau -\tau_1,k-k_1)|^2\; d\tau_1\Bigg)^{\frac12}  \Bigg)^2 \; d\tau \\
  \lesssim&~{} L\norm{f_{h,2}}_{X_{h,+}^{0,b}}^4,
  \end{aligned}\]
  where 
  \[L = \sup_{\substack{\tau \in \R, k \in (\mathbb{T}_h)^* \\ k > 1}}\sum_{\substack{k_1 \in (\mathbb{T}_h)^* \\ k-k_1, k_1 > 1}}\int_\R \bra{\tau-\tau_1 - s_h(k-k_1)}^{-2b}\bra{\tau_1- s_h(k_1)}^{-2b} \; d\tau_1.\]
  Thus, it is enough to show that $L \lesssim 1$ whenever $b > \frac13$.
  
  \medskip
  
  By a direct computation
  \[\int_{\R} \bra{a}^{-\alpha} \bra{b-a}^{-\alpha} \; da \lesssim \bra{b}^{1-2\alpha},\]
  for $\frac12 < \alpha < 1$, we estimate
  \[L \lesssim \sup_{\substack{\tau \in \R, k \in (\mathbb{T}_h)^* \\ k > 1}} \sum_{\substack{k_1 \in (\mathbb{T}_h)^* \\ k_1, k-k_1 > 1}} \bra{\tau-s_h(k_1) - s_h(k-k_1)}^{1-4b}.\]
 Fix $\tau \in \R$ and $k \in (\mathbb{T}_h)^*$ with $k>1$. We compute
\[\begin{aligned}
  &s_h(k_1) + s_h(k-k_1) - \tau \\
  =&~{} \frac{1}{h^2}\left(k - \frac2h\left(\sin\left(\frac{hk_1}{2}\right) + \sin\left(\frac{h(k-k_1)}{2}\right)\right)\right) - \tau\\
  =&~{} \frac{4}{h^3}\sin\left(\frac{hk}{4}\right)\left(1 - \cos\left(\frac{h(2k_1-k)}{4}\right)\right) - \left(\tau - \frac{1}{h^2}\left(k - \frac4h\sin\left(\frac{hk}{4}\right)\right)\right)\\
  =&~{} \frac{4}{h^3}\sin\left(\frac{hk}{4}\right)\left(1 - \cos\left(\frac{h(2k_1-k)}{4}\right) - C_h(\tau, k)\right)\\
  =:&~{} \frac{4}{h^3}\sin\left(\frac{hk}{4}\right)g(k_1),
  \end{aligned}\]
  where 
\[C_h(\tau, k) = \frac{\tau - \frac{1}{h^2}\left(k - \frac4h\sin\left(\frac{hk}{4}\right)\right)}{\frac{4}{h^3}\sin\left(\frac{hk}{4}\right)}.
\]
  Now we observe $g(x)$ for $x \in \left[\frac{k}{2} - \frac{\pi}{h}, \frac{k}{2} + \frac{\pi}{h}\right]$. When  $1 < x, k-x < \frac{\pi}{h}$, we know 
  \[-\frac{\pi}{2} \le \frac{h(2k_1-k)}{4} \le \frac{\pi}{2} \quad \mbox{and thus} \quad0 \le 1 - \cos\left(\frac{h(2x-k)}{4}\right) \le 1.\]
  Note also that 
  \begin{equation}\label{eq:sin}
  |\sin x| \ge \frac{|x|}{2},
  \end{equation}
  for all $|x| \le \frac{\pi}{2}$. Then, by Lemma \ref{lem:cos},
  \[|g(k_1)| \ge \frac14\left|\frac{h(2k_1-k)}{4} - \gamma\right| \left|\frac{h(2k_1-k)}{4} + \gamma \right|,\]
  where 
  \[\gamma = \begin{cases}0, \quad &\mbox{if} \;\; C_(\tau, k) \le 0, \\ \cos^{-1}(1-C_h(\tau,k)), \quad &\mbox{if} \;\; 0 < C(\tau, k) < 0, \\ \frac{\pi}{2}, \quad &\mbox{if} \;\; C(\tau, k) \ge 1. \end{cases}\]
  This, in addition to \eqref{eq:sin}, implies
  \[|\tau-s_h(k_1) - s_h(k-k_1)| \ge \frac{1}{16}k\left|k_1 - \tilde{\gamma}_+\right|\left|k_1 - \tilde{\gamma}_-\right|,\]
  where 
  \[\tilde{\gamma}_{\pm} = \frac{k}{2} \pm \frac{2\gamma}{h}.\]
  Let $\Omega_{\pm} := \{k_1 \in (\mathbb{T}_h)^* : \left|k_1 - \tilde{\gamma}_{\pm} \right| \le 1\}$. Then $|\Omega_{\pm}| \le 3$, and we have 
  \[\begin{aligned}
  L \lesssim&~{} \sum_{\substack{k_1 \in (\mathbb{T}_h)^* \setminus \Omega_{\pm} \\ k-k_1, k_1 >1}}\left(k\left|k_1 - \tilde{\gamma}_+\right|\left|k_1 - \tilde{\gamma}_-\right|\right)^{1-4b} \\
  \lesssim&~{} \left(\sum_{|k_1| > 1}|k_1|^{3(1-4b)}\right)^{\frac13}\left(\sum_{\left|k_1 - \tilde{\gamma}_+\right| > 1}\left|k_1 - \tilde{\gamma}_+\right|^{3(1-4b)}\right)^{\frac13}\left(\sum_{\left|k_1 - \tilde{\gamma}_-\right| > 1}\left|k_1 - \tilde{\gamma}_-\right|^{3(1-4b)}\right)^{\frac13}\\
  \lesssim&~{} 1,
  \end{aligned}\]
  if $b > \frac13$, which completes the proof.
  \end{proof}
  
\subsection{Bilinear estimates}\label{subsec:Bilinear estimates}
By Remark \ref{rem:Mean Conservation_3}, throughout this section, we may assume that $u_h^{\pm}$ and $v_{h}^{\pm}$ satisfy the mean zero condition. Define Hyper-surfaces by
\[\Theta : = \{(\tau_1, \tau_2, \tau_3) \in \R^3 : \tau_1 + \tau_2 + \tau_3 =0\}\]
and
\[\Delta := \{(k_1, k_2, k_3) \in ((\T_h)^*)^3 : k_1+k_2+k_3 =0, \;\; k_1k_2k_3 \neq 0\}.\] 
\begin{lemma}[Bilinear estimate I]\label{lem:bilinear1}
  Let $s\ge0$. Then, we have
\begin{equation}\label{Bilinear estimate I}
  \left\| \nabla_h (u_h^{\pm}v_h^{\pm})\right\|_{Z_{h,\pm}^s}
  \ls \|u_h^{\pm}\|_{X_{h,\pm}^{s,\frac12}}
  \|v_h^{\pm}\|_{X_{h,\pm}^{s,\frac12}},
\end{equation}
for all $u_h^{\pm}, v_h^{\pm} \in X_{h,\pm}^{s,\frac12}$.
\end{lemma}
\begin{proof}
We prove \eqref{Bilinear estimate I} only for $u_h^+,v_h^+$ and drop the "+" sign for the simplicity. We consider the first part of the $Z$-norm. By duality, we have
\[\begin{aligned}
 \left\| \nabla_h ( u_h \cdot v_h) \right  \|_{X_{h,+}^{s,-\frac12}} &=\sup_{\|w_h\|_{X_{h,+}^{-s,\frac12}}=1}
  \left| \int_{\R}\sum_{\T_h} w_h \nabla_h 
  \left( u_h \cdot v_h\right) dt\right| \\ 
  &=\sup_{\|w_h\|_{X_{h,+}^{-s,\frac12}}=1}\left| \int_{\Theta} \sum_{\Delta} \frac{2}{h}\sin\left(\frac{hk_3}{2}\right) \widetilde{u_h}(\tau_1,k_1)\widetilde{v_h}(\tau_2,k_2)\widetilde{w_h}(\tau_3,k_3) \right|.
\end{aligned}\]
Let 
\begin{equation}\begin{aligned}\label{notation}
  U_h(\tau,k)&=  \la k\ra^s\la \tau-s_h(k) \ra^\frac12 \left| \widetilde{u_h}(\tau,k) \right|, \\
  V_h(\tau,k)&=  \la k\ra^s \la \tau-s_h(k) \ra^\frac12 \left| \widetilde{v_h}(\tau,k)\right|,  \\ 
  W_h(\tau,k)&=  \la k\ra^{-s} \la \tau-s_h(k) \ra^\frac12 \left| \widetilde{w_h}(\tau,k) \right|.
\end{aligned}\end{equation}
Then, it suffices to control
\begin{equation}\begin{aligned}\label{Dual integral}
\int_{\Theta} \sum_{\Delta} \mathcal M_h(\tau_1,\tau_2,\tau_3,k_1,k_2,k_3)U_h(\tau_1,k_1)V_h(\tau_2,k_2)W_h(\tau_3,k_3),
\end{aligned}\end{equation}
where
\[\mathcal M_h(\tau_1,\tau_2,\tau_3,k_1,k_2,k_3) = \frac{ |k_3|^s\left| \frac{2}{h}\sin\left(\frac{hk_3}{2}\right) \right|}{|k_1|^s | k_2|^s \la \tau_1-s_h(k_1) \ra^\frac12 \la \tau_2-s_h(k_2) \ra^\frac12 \la \tau_3-s_h(k_3) \ra^\frac12  }.\]
For $k_1, k_2, k_3 \in (\T_h)^*$, let 
\[G(k_1,k_2,k_3) = s_h(k_1) + s_h(k_2) + s_h(k_3)\]
be the resonance function, which plays an important role in the bilinear $X^{s,b}$-type estimates. Note that 
\[G(k_1,k_2,k_3)=-\frac{8}{h^3}\sin\left(\frac{hk_1}{4}\right)\sin\left(\frac{hk_2}{4}\right)\sin\left(\frac{hk_3}{4}\right).\]
From the identities
\[k_1 + k_2+k_3 = 0 \quad \mbox{and} \quad \tau_1-s_h(k_1) + \tau_2-s_h(k_2) +\tau_3-s_h(k_3) = - G(k_1,k_2,k_3),\]
we know that\footnote{Once applying (spacetime) Littlewood-Paley decomposition to the summand in \eqref{Dual integral}, one knows \eqref{Dual integral} vanishes unless the conditions \eqref{eq:support properties} hold.} 
\begin{equation}\label{eq:support properties}
\max_{j=1,2,3}|k_j| \sim \med_{j=1,2,3}|k_j| \quad \mbox{and} \quad \max_{j=1,2,3}|\tau_j-s_h(k_j)| \sim \max \left(\med_{i=1,2,3}|\tau_j-s_h(k_j)|, |G(k_1,k_2,k_3)|\right).
\end{equation}
With mean-zero condition($k_1k_2k_3 \neq 0$), we further know
\begin{equation}\label{lowboundofmod}
  \max_{j=1,2,3} \la\tau_j-s_h(k_j)\ra^\frac12
\gs |k_1|^\frac12|k_2|^\frac12|k_3|^\frac12.
\end{equation}
Without loss of generality, we may assume that $\la\tau_1-s_h(k_1)\ra$ is the largest one in $\mathcal M_h(\tau_1,\tau_2,\tau_3,k_1,k_2,k_3)$. By \eqref{eq:support properties} and \eqref{lowboundofmod}, the multiplier $\mathcal M_h(\tau_1,\tau_2,\tau_3,k_1,k_2,k_3)$ is bounded by 
\begin{equation}\label{eq:multiplier estimates}
   \frac{ |k_3|^{1+s} }{ |k_1|^{\frac12+s}|k_2|^{\frac12+s} |k_3|^\frac12   \la \tau_2-s_h(k_2) \ra^\frac12 \la \tau_3-s_h(k_3) \ra^\frac12  } 
  \ls \frac{1}{\la \tau_2-s_h(k_2) \ra^\frac12 \la \tau_3-s_h(k_3) \ra^\frac12 }.
\end{equation}
Then, by Parseval's identity, Proposition \ref{prop:L4} and Lemma \ref{lem:properties}, we obtain
\[\begin{aligned}
  \eqref{Dual integral} &\ls \int_{\Theta} \sum_{\Delta} \frac{1}{\la \tau_2-s_h(k_2) \ra^\frac12 \la \tau_3-s_h(k_3) \ra^\frac12 } U_h(\tau_1,k_1)V_h(\tau_2,k_2)W_h(\tau_3,k_3) \\ 
  &\ls\int_{\R} \sum_{\T_h}\mathcal{F}_{t,h}^{-1}(U_h)
  \mathcal{F}_{t,h}^{-1} \left( \frac{V_h}{\la \tau-s_h(k) \ra^\frac12 } \right)
  \mathcal{F}_{t,h}^{-1}\left(\frac{W_h}{\la \tau-s_h(k) \ra^\frac12 }  \right) dt\\
  &\ls\left\| \mathcal{F}_{t,h}^{-1}(U_h) \right\|_{L_t^2L^2}
  \left\| \mathcal{F}_{t,h}^{-1} \left( \frac{V_h}{\la \tau-s_h(k) \ra^\frac12 } \right) \right\|_{L_t^4L^4}
  \left\| \mathcal{F}_{t,h}^{-1} \left( \frac{W_h}{\la \tau-s_h(k) \ra^\frac12 } \right) \right\|_{L_t^4L^4} \\ 
  &\ls \left\| U_h \right\|_{L_{t,x}^2}
  \left\| \mathcal{F}_{t,h}^{-1} \left( \frac{V_h}{\la \tau-s_h(k) \ra^\frac12 } \right) \right\|_{X_{h,+}^{0,\frac13+}}
  \left\| \mathcal{F}_{t,h}^{-1} \left( \frac{W_h}{\la \tau-s_h(k) \ra^\frac12 } \right) \right\|_{X_{h,+}^{0,\frac13+}}\\
  &\lesssim \| u_h\|_{X_{h,+}^{s,\frac12}}\| v_h \|_{X_{h,+}^{s,\frac12}}\| w_h \|_{X_{h,+}^{-s,\frac12}}.
\end{aligned}\]

\medskip

For the second part of the $Z_{h,+}^s$-norm, we have from duality that
\[\begin{aligned}
  &~{}\left\| \frac{\la k\ra^s \mathcal{F}_{t,x}\left( \nabla_h (u_h^{+}v_h^{+})\right)(\tau,k)}{\la \tau-s_h(k)\ra } \right\|_{ \ell_{k \in (\T_h)^*}^2L_\tau^1 } \\
  \lesssim&~{}\sup_{\|\wt{w_h}\|_{ \ell_{k \in (\T_h)^*}^2L_\tau^\infty}\le1}
  \int_{\Theta} \sum_{\Delta} \tilde{\mathcal M}_h(\tau_1,\tau_2,\tau_3,k_1,k_2,k_3)U_h(\tau_1,k_1)
  V_h(\tau_2,k_2)
  |\wt{w_h}(\tau_3,k_3)|, 
\end{aligned}\]
where $U_h$ and $V_h$ are defined as in \eqref{notation} and
\[\tilde{\mathcal M}_h(\tau_1,\tau_2,\tau_3,k_1,k_2,k_3)=\frac{|k_3|^s\left|\frac{2}{h}\sin(\frac{hk_3}{2})\right|} {|k_1|^s |k_2|^s \langle \tau_1-s_h(k_1)\ra^\frac12 \langle \tau_2-s_h(k_2)\ra^\frac12  \langle \tau_3-s_h(k_3)\ra}.\]

{\bf Case 1.} $ \max\limits_{i=1,2,3} \la\tau_i-s_h(k_i) \rangle = \la\tau_1-s_h(k_1) \rangle$. Note from Proposition \ref{prop:L4} that
\begin{equation}\label{eq:embedding}
\left\|\mathcal{F}_{t,h}^{-1} \left( \frac{|\wt{w_h}(\tau,k)|}{\la \tau-s_h(k) \ra } \right) \right\|_{L_t^4L^4} \lesssim \left\|\frac{\wt{w_h}(\tau,k)}{\la \tau-s_h(k) \ra^{\frac23-} } \right\|_{ \ell_{k \in (\T_h)^*}^2L_\tau^2} \lesssim \|\wt{w_h}\|_{ \ell_{k \in (\T_h)^*}^2L_\tau^\infty}.
\end{equation}
Similarly as \eqref{eq:multiplier estimates}, the multiplier is bounded by 
\[\frac{ |k_3|^{\frac12+s} }{|k_1|^{\frac12+s} |k_2|^{\frac12+s}   \la \tau_2-s_h(k_2) \ra^\frac12 \la \tau_3-s_h(k_3) \ra } 
 \ls \frac{1}{\la \tau_2-s_h(k_2) \ra^\frac12 \la \tau_3-s_h(k_3) \ra }.\]
Then, by Proposition \ref{prop:L4} and \eqref{eq:embedding}, we have
\[\begin{aligned}
&~{}\int_{\Theta} \sum_{\Delta} \tilde{\mathcal M}_h(\tau_1,\tau_2,\tau_3,k_1,k_2,k_3)U_h(\tau_1,k_1)
  V_h(\tau_2,k_2)
  |\wt{w_h}(\tau_3,k_3)|\\
  \ls&~{} \left\| \mathcal{F}_{t,h}^{-1}(U_h) \right\|_{L_t^2L^2}
  \left\| \mathcal{F}_{t,h}^{-1} \left( \frac{V_h}{\la \tau-s_h(k) \ra^\frac12 } \right) \right\|_{L_t^4L^4}
  \left\| \mathcal{F}_{t,h}^{-1} \left( \frac{|\wt{w_h}(\tau,k)|}{\la \tau-s_h(k) \ra } \right) \right\|_{L_t^4L^4} \\ 
  \ls&~{} \| u_h\|_{X_{h,+}^{s,\frac12}}\| v_h\|_{X_{h,+}^{s,\frac12}}\|\wt{w_h}\|_{ \ell_{k \in (\T_h)^*}^2L_\tau^\infty}.
\end{aligned}\]
Note by symmetry that the case $ \max\limits_{i=1,2,3} \la\tau_i-s_h(k_i) \rangle = \la\tau_2-s_h(k_2) \rangle$ follows analogously.

{\bf Case 2.} $\displaystyle  \max\limits_{i=1,2,3} \la\tau_i-s_h(k_i) \rangle = \la\tau_3-s_h(k_3) \rangle$. By \eqref{lowboundofmod}, we have
\[\la \tau_3-s_h(k_3)\ra 
 \gtrsim |k_1|^{\frac23}|k_2|^{\frac23}|k_3|^{\frac23}\la \tau_1 - s_h(k_1) \ra^{\frac16}\la \tau_2 - s_h(k_2) \ra^{\frac16}.\]
Without loss of generality, we further assume that $|k_1| \le |k_2|$. Then, the multiplier is bounded by 
\[\frac{|k_3|^{1+s} }{ |k_1|^{s+\frac23} |k_2|^{s+\frac23} |k_3|^{\frac23}\langle \tau_1-s_h(k_1)\ra^\frac23 \langle \tau_2-s_h(k_2)\ra^\frac23}\lesssim \frac{1}{|k_1|^{\frac23}\la \tau_1-s_h(k_1)\ra^{\frac23}\la \tau_2-s_h(k_2)\ra^{\frac23}}.\]
Since both $\la \tau - s_h(k) \ra^{-\frac23}$ and $|k|^{-\frac23}$ are square integrable in $\tau$ and $k$, respectively, we have
\[\begin{aligned}
  &~{}\int_{\tau_1, \tau_2} \sum_{\substack{k_1, k_2 \\ k_1k_2 \neq 0}} \tilde{\mathcal M}_h(\tau_1,\tau_2,-\tau_1-\tau_2,k_1,k_2,-k_1 - k_2)U_h(\tau_1,k_1)
  V_h(\tau_2,k_2)
  |\wt{w_h}(-\tau_1-\tau_2,-k_1-k_2)|\\
\lesssim&~{} \sum_{\substack{k_1, k_2 \\ k_1k_2 \neq 0}} \frac{1}{|k_1|^{\frac23}}\|U_h(k_1)\|_{L_{\tau}^2}\|V_h(k_2)\|_{L_{\tau}^2}\|\wt{w_h}(-k_1-k_2)\|_{L_{\tau}^\infty}\\
  \ls&~{} \|U_h\|_{\ell_{k \in (\T_h)^*}^2L_{\tau}^2}\|V_h\|_{\ell_{k \in (\T_h)^*}^2L_{\tau}^2}\|\wt{w_h}\|_{\ell_{k \in (\T_h)^*}^2L_{\tau}^\infty}\\
\lesssim&~{} \| u_h\|_{X_{h,+}^{s,\frac12}}\| v_h\|_{X_{h,+}^{s,\frac12}}\|\wt{w_h}\|_{\ell_{k \in (\T_h)^*}^2L_\tau^\infty}.
\end{aligned}\]
\end{proof}

\begin{lemma}[Bilinear estimate II]\label{lem:bilinear2}
  Let $0\le s \le s'\le s+1$. Then,
  \[\left\| \nabla_h ( u_h^\pm \cdot e^{\pm\frac{2t}{h^2}\partial_h}v_h^\mp) \right\|_{Z_{h,\pm}^s}
    \ls h^{s'-s}\|u_h^{\pm}\|_{X_{h,\pm}^{s',\frac12}}
    \|v_h^{\mp}\|_{X_{h,\mp}^{s',\frac12}},\]
for all $u_h^{\pm} \in X_{h,\pm}^{s',\frac12}, v_h^{\mp} \in X_{h,\mp}^{s',\frac12}$.
\end{lemma}
\begin{proof}
We only prove 
\[\left\| \nabla_h ( u_h^+ \cdot e^{+\frac{2t}{h^2}\partial_h}v_h^-) \right\|_{Z_{h,+}^s}
    \ls h^{s'-s}\|u_h^{+}\|_{X_{h,+}^{s',\frac12}}
    \|v_h^{-}\|_{X_{h,-}^{s',\frac12}}.\]
Similarly as in the proof of Lemma \ref{lem:bilinear1}, we write by duality that  
\begin{equation}\label{eq:duality}
\begin{aligned}
&\left\| \nabla_h ( u_h^+ \cdot e^{\frac{2t}{h^2}\partial_h}v_h^-) \right\|_{X_{h,+}^{s,\frac12}} \\
=&~{}\sup_{\|w_h\|_{X_{h,+}^{-s,-\frac12}}\le 1}\left|\int_{\Theta} \sum_{\Delta}\frac{2}{h}\sin\left(\frac{hk_3}{2}\right)\widetilde{u_h^+}(\tau_1,k_1) 
\widetilde{v_h^-}(\tau_2-\tfrac{2k_2}{h^2},k_2) 
\widetilde{w_h}(\tau_3,k_3)\right|
\end{aligned}
\end{equation}
Define
\begin{equation}\label{eq:UVW}
\begin{aligned}
  U_h(\tau,k)&=  \la k\ra^{s'} \la \tau-s_h(k) \ra^\frac12 \left| \widetilde{u_h^+}(\tau,k) \right| \\
  V_h(\tau,k)&=  \la k\ra^{s'}  \la \tau-\tfrac{2k}{h^2}+s_h(k) \ra^\frac12 \left| \widetilde{v_h^-}(\tau-\tfrac{2k}{h^2},k)\right|  \\ 
  W_h(\tau,k)&=  \la k\ra^{-s}\la \tau-s_h(k) \ra^\frac12 \left| \widetilde{w_h}(\tau,k) \right|.
\end{aligned}
\end{equation}
Then, the right-hand side of \eqref{eq:duality} is bounded by
\begin{equation}\label{eq:bi2}
  \int_{\Theta} \sum_{\Delta} \mathcal M_h(\tau_1,\tau_2,\tau_3,k_1,k_2,k_3) U_h(\tau_1,k_1)V_h(\tau_2,k_2)W_h(\tau_3,k_3),
\end{equation}
where
\[\mathcal M_h(\tau_1,\tau_2,\tau_3,k_1,k_2,k_3) = \frac{|k_3|^s \left|\frac{2}{h} \sin\left(\frac{hk_3}{2}\right)\right|}
  { |k_1|^{s'}|k_2|^{s'}\la \tau_1-s_h(k_1) \ra^\frac12 \la \tau_2-\tfrac{2k_2}{h^2}+s_h(k_2) \ra^\frac12
  \la \tau_3-s_h(k_3) \ra^\frac12 }.\]
For $k_1, k_2, k_3 \in (\T_h)^*$, define the resonant function as 
\[G(k_1,k_2,k_3) = s_h(k_1) - s_h(k_2) + s_h(k_3) + \frac{2k_2}{h^2}.\]
Note that 
\[G(k_1,k_2,k_3)=\frac{8}{h^3}\cos\left(\frac{hk_1}{4}\right)\sin\left(\frac{hk_2}{4}\right)\cos\left(\frac{hk_3}{4}\right).\]
From the identities
\[k_1 + k_2+k_3 = 0 \quad \mbox{and} \quad \tau_1-s_h(k_1) + \tau_2 - \frac{2k_2}{h^2} + s_h(k_2) +\tau_3-s_h(k_3) = - G(k_1,k_2,k_3),\]
we know that
\[\max_{j=1,2,3}|k_j| \sim \med_{j=1,2,3}|k_j|\]
and
\begin{equation}\label{lowbound of mod}
  \begin{aligned}
  M:=&~{}\max\left(|\tau_1-s_h(k_1)|  ,\left|\tau_2-\frac{2k_2}{h}+s_h(k_2)\right|  , |\tau_3-s_h(k_3)|\right) \\ 
  \gs&~{} \frac{1}{h^3}\cos\left(\frac{hk_1}{4}\right)\cos\left(\frac{hk_3}{4}\right)\left|\sin\left(\frac{hk_2}{4}\right)\right|.
  \end{aligned}\end{equation}
Note that $\displaystyle \cos\left(\frac{hk}{4}\right) \ge \frac{1}{\sqrt{2}}$ for all $k \in (\T_h)^*$. From \eqref{lowbound of mod}, we immediately know
\[(1+M)^\frac12\gtrsim \left(1+\frac{|k_2|}{h^2}\right)^\frac12 \gtrsim h^{-1}|k_2|^\frac12.\]
Without loss of generality, we may assume that $M=|\tau_1-s_h(k_1)|$ in $\mathcal M_h(\tau_1,\tau_2,\tau_3,k_1,k_2,k_3)$. Note from \eqref{lowbound of mod} that dispersive smoothing effect is efficient when $|k_2|$ is the maximum frequency, thus it suffices to consider the worst case when $|k_2| \le |k_1| \sim |k_3|$. Since $h|k| \lesssim 1$ for all $k \in (\T_h)^*$ and $1+s-s' \ge 0$, for non-zero frequencies, the multiplier $\mathcal M_h(\tau_1,\tau_2,\tau_3,k_1,k_2,k_3)$ is bounded by
\[\frac{ h|k_3|^{1+s}|k_2|^{-\frac12}}{ |k_1|^{s'} |k_2|^{s'} \la \tau_2-\tfrac{2k_2}{h^2}+s_h(k_2) \ra^\frac12 \la \tau_3-s_h(k_3) \ra^\frac12 } \lesssim \frac{h^{s'-s}}{ \la \tau_2-\tfrac{2k_2}{h^2}+s_h(k_2) \ra^\frac12 \la \tau_3-s_h(k_3) \ra^\frac12 }.\]
Then, by Parseval's identity, Proposition \ref{prop:L4} and Lemma \ref{lem:properties}, we obtain
\[\begin{aligned}
  \eqref{eq:bi2} &\ls \int_{\Theta} \sum_{\Delta}  \frac{h^{s'-s}}{ \la \tau_2-\tfrac{2k_2}{h^2}+s_h(k_2) \ra^\frac12 \la \tau_3-s_h(k_3) \ra^\frac12 }  U_h(\tau_1,k_1)V_h(\tau_2,k_2)W_h(\tau_3,k_3) \\ 
  &\ls h^{s'-s}\int_{\R} \sum_{\T_h}\mathcal{F}_{t,h}^{-1}(U_h)
  \mathcal{F}_{t,h}^{-1} \left( \frac{V_h}{\la \tau-\tfrac{2k}{h^2}+s_h(k) \ra^\frac12 } \right)
  \mathcal{F}_{t,h}^{-1}\left(\frac{W_h}{\la \tau-s_h(k) \ra^\frac12 }  \right) dt,\\
  &\ls h^{s'-s}\left\| \mathcal{F}_{t,h}^{-1}(U_h) \right\|_{L_t^2L^2}
  \left\| \mathcal{F}_{t,h}^{-1} \left( \frac{V_h}{\la \tau-\tfrac{2k}{h^2}+s_h(k) \ra^\frac12 } \right) \right\|_{L_t^4L^4}
  \left\| \mathcal{F}_{t,h}^{-1} \left( \frac{W_h}{\la \tau-s_h(k) \ra^\frac12 } \right) \right\|_{L_t^4L^4} \\ 
  &\ls h^{s'-s}\| u_h^{+} \|_{X_{h,+}^{s',\frac12}}
  \left\| \mathcal{F}_{t,h}^{-1} \left( \frac{V_h}{\la \tau-\tfrac{2k}{h^2}+s_h(k) \ra^\frac12 } \right) \right\|_{X_{h,-}^{0,\frac13}}
  \left\| \mathcal{F}_{t,h}^{-1} \left( \frac{W_h}{\la \tau-s_h(k) \ra^\frac12 } \right) \right\|_{X_{h,+}^{0,\frac13}} \\ 
  &\ls h^{s'-s} \| u_h^{+} \|_{X_{h,+}^{s',\frac12}}
  \| v_h^{-} \|_{X_{h,-}^{s',\frac12}}
  \| w_h \|_{X_{h,+}^{-s,\frac12}}.
\end{aligned}\]

For the second part of $Z_{h,+}^s$, by duality, we have
\[\begin{aligned}
&~{}\left\| \frac{\la k\ra^s \mathcal{F}_{t,x} \left(\nabla_h ( u_h^+ \cdot e^{\frac{2t}{h^2}\partial_h}v_h^-)\right)  (\tau,k)}{\la \tau-s_h(k)\ra } \right\|_{ \ell_{k \in (\T_h)^*}^2L_\tau^1 }    \\ 
\lesssim&~{}\sup_{\|\wt{w_h}\|_{\ell_{k \in (\T_h)^*}^2L_\tau^\infty}\le1}
\int_{\Theta} \sum_{\Delta} \tilde{\mathcal M}_h(\tau_1,\tau_2,\tau_3,k_1,k_2,k_3){U_h}(\tau_1,k_1){V_h}(\tau_2,k_2)|\widetilde{w_h}(\tau_3,k_3)|,
\end{aligned}\]
where $U_h$ and $V_h$ are defined as in \eqref{eq:UVW}, and
\[\tilde{\mathcal M}_h(\tau_1,\tau_2,\tau_3,k_1,k_2,k_3) = \frac{|k_3|^s\left|\frac{2}{h}\sin\left(\frac{hk_3}{2}\right)\right|}
{ |k_1|^{s'} |k_2|^{s'} \langle \tau_1-s_h(k_1)\ra^\frac12 \langle \tau_2-\tfrac{2k_2}{h^2}+s_h(k_2)\ra^\frac12 \langle \tau_3-s_h(k_3)\ra }.\]

{\bf Case 1.}  $ M=|\tau_1-s_h(k_1)|$. Similarly as above in addition to \eqref{lowbound of mod} under the worst case $|k_2| \le |k_1| \sim |k_3|$,  $\tilde{\mathcal M}_h(\tau_1,\tau_2,\tau_3,k_1,k_2,k_3)$ is bounded by  
\[  \frac{ h|k_3|^{1+s}|k_2|^{-\frac12} }{ |k_1|^{s'}|k_2|^{s'} \la \tau_3 - s_h(k_3)\ra  \la \tau_2-\tfrac{2k_2}{h^2}+s_h(k_2) \ra^\frac12 } \ls  \frac{h^{s'-s}}{\la \tau_3 - s_h(k_3)\ra  \la \tau_2-\tfrac{2k_2}{h}+s_h(k_2) \ra^\frac12}.\]
Therefore, by Proposition \ref{prop:L4} and \eqref{eq:embedding}, we have
\[\begin{aligned}
&~{}\int_{\Theta} \sum_{\Delta} \tilde{\mathcal M}_h(\tau_1,\tau_2,\tau_3,k_1,k_2,k_3)U_h(\tau_1,k_1)
  V_h(\tau_2,k_2)
  |\wt{w_h}(\tau_3,k_3)|\\
  \ls&~{} h^{s'-s}\left\| \mathcal{F}_{t,h}^{-1}(U_h) \right\|_{L_t^2L^2}
  \left\| \mathcal{F}_{t,h}^{-1} \left( \frac{V_h}{\la \tau-s_h(k) \ra^\frac12 } \right) \right\|_{L_t^4L^4}
  \left\| \mathcal{F}_{t,h}^{-1} \left( \frac{|\wt{w_h}(\tau,k)|}{\la \tau-s_h(k) \ra } \right) \right\|_{L_t^4L^4} \\ 
  \ls&~{} h^{s'-s}\| u_h^+\|_{X_{h,+}^{s,\frac12}}\| v_h^-\|_{X_{h,-}^{s,\frac12}}\|\wt{w_h}\|_{\ell_{k \in (\T_h)^*}^2L_\tau^\infty}.
\end{aligned}\]
Note by symmetry that the case $  M=|\tau_2-\tfrac{2k_2}{h^2}+s_h(k_2)|$ follows analogously.

{\bf Case 2.}  $ M=|\tau_3-s_h(k_3)|$. By  \eqref{lowbound of mod}, we have
\[\la \tau_3-s_h(k_3)\ra 
 \gtrsim  h^{-\frac43}|k_2|^{\frac23}\la \tau_1 - s_h(k_1) \ra^{\frac16}\la \tau_2-\tfrac{2k_2}{h^2}+s_h(k_2)) \ra^{\frac16}.\]
Since the worst case occurs when $|k_2|$ is the smallest one, we further assume that $|k_2| \le |k_1| \sim |k_3|$. Then, the multiplier is bounded by 
\[\frac{|k_3|^{1+s}h^{\frac43} }{ |k_1|^{s'} |k_2|^{s'+\frac23} \la \tau_1-s_h(k_1) \ra^\frac23 \la \tau_2-\tfrac{2k_2}{h^2}+s_h(k_2) \ra^\frac23}\lesssim \frac{h^{\frac13+s'-s}}{|k_2|^{\frac23}\la  \tau_1-s_h(k_1)\ra^{\frac23}\la \tau_2-\tfrac{2k_2}{h^2}+s_h(k_2)\ra^{\frac23}}.\]
Since both $\la \tau - s_h(k) \ra^{-\frac23}$ and $k^{-\frac23}$ are square integrable in $\tau$ and $k$, respectively, we have
\[\begin{aligned}
  &~{}\int_{\tau_1, \tau_2} \sum_{\substack{k_1, k_2 \\ k_1k_2 \neq 0}} \tilde{\mathcal M}_h(\tau_1,\tau_2,-\tau_1-\tau_2,k_1,k_2,-k_1 - k_2)U_h(\tau_1,k_1)
  V_h(\tau_2,k_2)
  |\wt{w_h}(-\tau_1-\tau_2,-k_1-k_2)|\\
\lesssim&~{} \sum_{\substack{k_1, k_2 \\ k_1k_2 \neq 0}} \frac{h^{\frac13+s'-s}}{|k_2|^{\frac23}}\|U_h(k_1)\|_{L_{\tau}^2}\|V_h(k_2)\|_{L_{\tau}^2}\|\wt{w_h}(-k_1-k_2)\|_{L_{\tau}^\infty}\\
  \ls&~{} h^{\frac13+s'-s}\|U_h\|_{\ell_{k \in (\T_h)^*}^2L_{\tau}^2}\|V_h\|_{\ell_{k \in (\T_h)^*}^2L_{\tau}^2}\|\wt{w_h}\|_{\ell_{k \in (\T_h)^*}^2L_{\tau}^\infty}\\
\lesssim&~{} h^{\frac13+s'-s}\| u_h^+\|_{X_{h,+}^{s,\frac12}}\| v_h^-\|_{X_{h,-}^{s,\frac12}}\|\wt{w_h}\|_{\ell_{k \in (\T_h)^*}^2L_\tau^\infty}.
\end{aligned}\]
\end{proof}

\begin{lemma}[Bilinear estimate III]\label{lem:bilinear3}
Let $0\le s \le s'\le s+1$. Then,
\[ \left\| \nabla_h ( e^{\pm\frac{2t}{h^2}\partial_h}u_h^\mp \cdot e^{\pm\frac{2t}{h^2}\partial_h}v_h^\mp) \right\|_{Z_{h,\pm}^s}
    \ls h^{\frac12+s'-s}\|u_h^{\mp}\|_{X_{h,\mp}^{s',\frac12}}
    \|v_h^{\mp}\|_{X_{h,\mp}^{s',\frac12}},  \]
for all $u_h^{\mp}, v_h^{\mp} \in X_{h,\mp}^{s',\frac12}$.
\end{lemma}
\begin{proof}
In view of the proof of Lemma \ref{lem:bilinear2}, it suffices to prove 
\[\left\| \nabla_h ( e^{\frac{2t}{h^2}\partial_h}u_h^- \cdot e^{\frac{2t}{h^2}\partial_h}v_h^-) \right\|_{Z_{h,+}^0}
    \ls h^{\frac12}\|u_h^{-}\|_{X_{h,-}^{0,\frac12}}
    \|v_h^{-}\|_{X_{h,-}^{0,\frac12}}.\]
  We start with the first part of $Z_{h,+}^0$ norm 
\begin{equation}\label{eq:dualitybi3}
\begin{aligned}
  &~{}\left\| \nabla_h ( e^{\frac{2t}{h^2}\partial_h}u_h^- \cdot e^{\frac{2t}{h^2}\partial_h}v_h^-) \right  \|_{X_{h,+}^{0,-\frac12}} \\
=&~{}\sup_{\|w_h\|_{X_{h,+}^{0,\frac12}}=1}\left| \int_{\Theta} \sum_{\Delta}
  \frac{2}{h}\sin(\tfrac{hk_3}{2})
  \widetilde{u_h^-}(\tau_1-\tfrac{2k}{h^2},k_1)
  \widetilde{v_h^-}(\tau_2-\tfrac{2k}{h^2},k_2)
  \widetilde{w_h}(\tau_3,k_3) \right|
\end{aligned}
\end{equation} 
Define
\begin{equation}\label{eq:UVW3}
\begin{aligned}
  U_h(\tau,k)&=  \la \tau-\tfrac{2k}{h^2}+s_h(k) \ra^\frac12 \left| \widetilde{u_h^-}(\tau-\tfrac{2k}{h^2},k) \right| \\
  V_h(\tau,k)&=  \la \tau-\tfrac{2k}{h^2}+s_h(k) \ra^\frac12 \left| \widetilde{v_h^-}(\tau-\tfrac{2k}{h^2},k)\right|  \\ 
  W_h(\tau,k)&=  \la \tau-s_h(k) \ra^\frac12 \left| \widetilde{w_h}(\tau,k) \right|.
\end{aligned}
\end{equation}
Then, the right-hand side of \eqref{eq:dualitybi3} is bounded by 
\begin{equation}\label{Dual integral2}
  \int_{\Theta} \sum_{\Delta}
  \mathcal M_h(\tau_1,\tau_2,\tau_3,k_1,k_2,k_3)U_h(\tau_1,k_1)V_h(\tau_2,k_2)W_h(\tau_3,k_3),
\end{equation}
where
\[\mathcal M_h(\tau_1,\tau_2,\tau_3,k_1,k_2,k_3) = \frac{\frac{2}{h}\left| \sin(\frac{hk_3}{2})\right| }
  { \la \tau_1-\tfrac{2k_1}{h^2}+s_h(k_1) \ra^\frac12 \la \tau_2-\tfrac{2k_2}{h^2}+s_h(k_2) \ra^\frac12
  \la \tau_3-s_h(k_3) \ra^\frac12  }.\]
For $k_1, k_2, k_3 \in (\T_h)^*$, define the resonant function as 
\[G(k_1,k_2,k_3) = s_h(k_1)-\frac{2k_1}{h^2}+ s_h(k_2) -\frac{2k_2}{h^2} -s_h(k_3).\]
Note that 
\[G(k_1,k_2,k_3)=-\frac{8}{h^3}\cos\left(\frac{hk_1}{4}\right)\cos\left(\frac{hk_2}{4}\right)\sin\left(\frac{hk_3}{4}\right).\]
From the identities
\[k_1 + k_2+k_3 = 0 \quad \mbox{and} \quad \tau_1-\frac{2k_1}{h^2}+s_h(k_1)   +\tau_2-\frac{2k_2}{h^2}+s_h(k_2)   +\tau_3-s_h(k_3) = - G(k_1,k_2,k_3),\]
we know that
\[\max_{j=1,2,3}|k_j| \sim \med_{j=1,2,3}|k_j|\]
and
\begin{equation} \label{lowbound of mod2}
\begin{aligned}
  M:=&~{}\max\left(\left|\tau_1-\frac{2k_1}{h^2}+s_h(k_1)\right|, \left|\tau_2-\frac{2k_2}{h^2}+s_h(k_2)\right|, |\tau_3-s_h(k_3)|\right) \\ 
  \gs&~{} \frac{1}{h^3}\cos\left(\frac{hk_1}{4}\right)\cos\left(\frac{hk_2}{4}\right)\left|\sin\left(\frac{hk_3}{4}\right)\right|,
  \end{aligned}
\end{equation}
Note that $\displaystyle \cos\left(\frac{hk}{4}\right) \ge \frac{1}{\sqrt{2}}$ for all $k \in (\T_h)^*$. From \eqref{lowbound of mod2}, we immediately know
\[(1+M)^\frac12\gtrsim   (1+M)^\frac12\ge \left(1+\frac{|k_3|}{h^2}\right)^\frac12 \ge h^{-1}|k_3|^\frac12,\]
which says that dispersive smoothing effect is efficient when $|k_3|$ is the maximum frequency, thus it suffices to consider the worst case when $|k_3| \le |k_1| \sim |k_2|$. Without loss of generality, we may assume that $M= |\tau_1-\tfrac{2k_1}{h^2}+s_h(k_1)|$ in $\mathcal M_h(\tau_1,\tau_2,\tau_3,k_1,k_2,k_3)$. Since $h|k| \lesssim 1$ for all $k \in (\T_h)^*$, for non-zero frequencies, the multiplier $\mathcal M_h(\tau_1,\tau_2,\tau_3,k_1,k_2,k_3)$ is bounded by
\[\frac{ h|k_3|^{\frac12} }{ \la \tau_2-\tfrac{2k_2}{h^2}+s_h(k_2) \ra^\frac12 \la \tau_3-s_h(k_3) \ra^\frac12 } \lesssim \frac{h^{\frac12}}{ \la \tau_2-\tfrac{2k_2}{h^2}+s_h(k_2) \ra^\frac12 \la \tau_3-s_h(k_3) \ra^\frac12 }.\]
Then, by Parseval's identity, Proposition \ref{prop:L4} and Lemma \ref{lem:properties}, we obtain
\[\begin{aligned}
  \eqref{Dual integral2} \ls&~{}  \int_{\Theta} \sum_{\Delta}\frac{h^\frac12}{ \la \tau_2-\tfrac{2k_2}{h^2}+s_h(k_2) \ra^\frac12 \la \tau_3-s_h(k_3) \ra^\frac12 } U_h(\tau_1,k_1)V_h(\tau_2,k_2)W_h(\tau_3,k_3) \\ 
  =&~{}h^\frac12 \int_{\R} \sum_{\T_h}\mathcal{F}_{t,h}^{-1}(U_h)
  \mathcal{F}_{t,h}^{-1} \left( \frac{V_h}{\la \tau-\tfrac{2k}{h^2}+s_h(k) \ra^\frac12 } \right)
  \mathcal{F}_{t,h}^{-1}\left(\frac{W_h}{\la \tau-s_h(k) \ra^\frac12 }  \right) dt\\
  \lesssim&~{}h^\frac12\left\| \mathcal{F}_{t,h}^{-1}(U_h) \right\|_{L_t^2L^2}
  \left\| \mathcal{F}_{t,h}^{-1} \left( \frac{V_h}{\la \tau-\tfrac{2k}{h^2}+s_h(k) \ra^\frac12 } \right) \right\|_{L_t^4L^4}
  \left\| \mathcal{F}_{t,h}^{-1} \left( \frac{W_h}{\la \tau-s_h(k) \ra^\frac12 } \right) \right\|_{L_t^4L^4} \\ 
  \ls&~{} h^\frac12\| u_h^{-} \|_{X_{h,-}^{0,\frac12}}
  \left\| \mathcal{F}_{t,h}^{-1} \left( \frac{V_h}{\la \tau-\tfrac{2k}{h^2}+s_h(k) \ra^\frac12 } \right) \right\|_{X_{h,-}^{0,\frac13}}
  \left\| \mathcal{F}_{t,h}^{-1} \left( \frac{W_h}{\la \tau-s_h(k) \ra^\frac12 } \right) \right\|_{X_{h,+}^{0,\frac13}} \\ 
  \lesssim&~{} h^\frac12\| u_h^{-} \|_{X_{h,-}^{0,\frac12}}
  \| v_h^{-} \|_{X_{h,-}^{0,\frac12}}
  \| w_h \|_{X_{h,+}^{0,\frac12}}.
\end{aligned}\]

Next, we consider the second part of the $Z_{h,+}^s$ norm. By duality, we have
\[\begin{aligned}
&\left \|  \frac{\mathcal F_{t,x} \left(\nabla_h ( e^{\frac{2t}{h^2}\partial_h}u_h^- \cdot e^{\frac{2t}{h^2}\partial_h}v_h^-)\right)(\tau,k)}{\la \tau - s_h(k)\ra}  \right\|_{\ell_{k \in (\T_h)^*}^2L_\tau^1} \\ 
\lesssim&\sup_{\|\wt{w_h}\|_{\ell_{k \in (\T_h)^*}^2L_\tau^\infty}\le1 } 
  \int_{\Theta} \sum_{\Delta} \tilde{\mathcal M}_h(\tau_1,\tau_2,\tau_3,k_1,k_2,k_3) U_h(\tau_1,k_1)V_h(\tau_2,k_2)|\wt{w_h}(\tau_3,k_3)|,
\end{aligned}\]
where $U_h$ and $V_h$ are defined as in \eqref{eq:UVW3}, and
\[\tilde{\mathcal M}_h(\tau_1,\tau_2,\tau_3,k_1,k_2,k_3) = \frac{ |\frac{2}{h}\sin\frac{hk_3}{2}| }{\la \tau_3 - s_h(k_3)\ra \la \tau_1-\tfrac{2k_1}{h^2}+s_h(k_1) \ra^\frac12 \la \tau_2-\tfrac{2k_2}{h^2}+s_h(k_2) \ra^\frac12}.\]

{\bf Case 1.} $ M=|\tau_1-\tfrac{2k_1}{h}+s_h(k_1)|$. Since $h|k| \lesssim 1$ for all $k \in (\T_h)^*$, from \eqref{lowbound of mod2}, the multiplier $\tilde{\mathcal M}_h(\tau_1,\tau_2,\tau_3,k_1,k_2,k_3)$ is bounded by  
\[\frac{ h|k_3|^{\frac12} }{ \la \tau_2-\tfrac{2k_2}{h^2}+s_h(k_2) \ra^\frac12\la \tau_3 - s_h(k_3)\ra}  \ls \frac{h^{\frac12}}{ \la \tau_2-\tfrac{2k_2}{h^2}+s_h(k_2) \ra^\frac12 \la \tau_3 - s_h(k_3)\ra }.\]
Therefore, by Proposition \ref{prop:L4} and \eqref{eq:embedding}, we have
\[\begin{aligned}
&~{}\int_{\Theta} \sum_{\Delta} \tilde{\mathcal M}_h(\tau_1,\tau_2,\tau_3,k_1,k_2,k_3)U_h(\tau_1,k_1)
  V_h(\tau_2,k_2)
  |\wt{w_h}(\tau_3,k_3)|\\
  \ls&~{} h^{\frac12}\left\| \mathcal{F}_{t,h}^{-1}(U_h) \right\|_{L_t^2L^2}
  \left\| \mathcal{F}_{t,h}^{-1} \left( \frac{V_h}{\la \tau-s_h(k) \ra^\frac12 } \right) \right\|_{L_t^4L^4}
  \left\| \mathcal{F}_{t,h}^{-1} \left( \frac{|\wt{w_h}(\tau,k)|}{\la \tau-s_h(k) \ra } \right) \right\|_{L_t^4L^4} \\ 
  \ls&~{} h^{\frac12}\| u_h^-\|_{X_{h,-}^{0,\frac12}}\| v_h^-\|_{X_{h,-}^{0,\frac12}}\|\wt{w_h}\|_{\ell_{k \in (\T_h)^*}^2L_\tau^\infty}.
\end{aligned}\]
Note by symmetry that the case $  M=|\tau_2-\tfrac{2k_2}{h^2}+s_h(k_2)|$ follows analogously.

{\bf Case 2.} $M=|\tau_3-s_h(k_3)| $. By  \eqref{lowbound of mod2}, we have
\[\la \tau_3-s_h(k_3)\ra 
 \gtrsim  h^{-\frac32}|k_3|^{\frac34}\la \tau_1-\tfrac{2k_1}{h^2}+s_h(k_1)) \ra^{\frac18}\la \tau_2-\tfrac{2k_2}{h^2}+s_h(k_2)) \ra^{\frac18}.\]
Since $h|k| \lesssim 1$ for all $(\T_h)^*$, the multiplier is bounded by 
\[\frac{h^{\frac32}|k_3|^{\frac14} }{ \la \tau_1-s_h(k_1) \ra^\frac58 \la \tau_2-\tfrac{2k_2}{h^2}+s_h(k_2) \ra^\frac58}\lesssim \frac{h^{\frac12}}{|k_2|^{\frac34}\la  \tau_1-s_h(k_1)\ra^{\frac58}\la \tau_2-\tfrac{2k_2}{h^2}+s_h(k_2)\ra^{\frac58}}.\]
Since both $\la \tau - s_h(k) \ra^{-\frac58}$ and $k^{-\frac34}$ are square integrable in $\tau$ and $k$, respectively, we have
\[\begin{aligned}
  &~{}\int_{\tau_1, \tau_2} \sum_{\substack{k_1, k_2 \\ k_1k_2 \neq 0}} \tilde{\mathcal M}_h(\tau_1,\tau_2,-\tau_1-\tau_2,k_1,k_2,-k_1 - k_2)U_h(\tau_1,k_1)
  V_h(\tau_2,k_2)
  |\wt{w_h}(-\tau_1-\tau_2,-k_1-k_2)|\\
\lesssim&~{} \sum_{\substack{k_1, k_2 \\ k_1k_2 \neq 0}} \frac{h^{\frac12}}{|k_2|^{\frac34}}\|U_h(k_1)\|_{L_{\tau}^2}\|V_h(k_2)\|_{L_{\tau}^2}\|\wt{w_h}(-k_1-k_2)\|_{L_{\tau}^\infty}\\
  \ls&~{} h^{\frac12}\|U_h(k_1)\|_{\ell_{k \in (\T_h)^*}^2L_{\tau}^2}\|V_h(k_2)\|_{\ell_{k \in (\T_h)^*}^2L_{\tau}^2}\|\wt{w_h}(-k_1-k_2)\|_{\ell_{k \in (\T_h)^*}^2L_{\tau}^\infty}\\
\lesssim&~{} h^{\frac12}\| u_h^-\|_{X_{h,-}^{0,\frac12}}\| v_h^-\|_{X_{h,-}^{0,\frac12}}\|\wt{w_h}\|_{\ell_{k \in (\T_h)^*}^2L_\tau^\infty}.
\end{aligned}\]
\end{proof}
As an immediate corollary, we have
\begin{corollary}\label{Cor:Linear Bilinear}
Let $0 \le s \le s' \le s+1$. Suppose $u_h^{\pm}(t),v_h^{\pm}(t)\in Y_{h,\pm}^{s',T}$ are supported in $[0,T]$ and satisfy mean zero assumption. Then, we have
\[\begin{aligned}
  \left\| \nabla_h (u_h^{\pm}v_h^{\pm})\right\|_{Z^{s,T}_{h,\pm}} &\ls T^{\frac16-}\|u_h^{\pm}\|_{Y_{h,\pm}^{s,T}}\|u_h^{\pm}\|_{Y_{h,\pm}^{s,T}}, \\ 
\left\| \nabla_h ( u_h^\pm \cdot e^{\pm\frac{2t}{h^2}\partial_h}v_h^\mp) \right\|_{Z^{s,T}_{h,\pm}} &\ls h^{s'-s}T^{\frac16-}\|u_h^{\pm}\|_{Y_{h,\pm}^{s',T}} \|v_h^{\mp}\|_{Y_{h,\mp}^{s',T}},\\  
\left\| \nabla_h ( e^{\pm\frac{2t}{h^2}\partial_h}u_h^\mp \cdot e^{\pm\frac{2t}{h^2}\partial_h}v_h^\mp) \right\|_{Z_{h,\pm}^{s,T}} &\ls h^{\frac12+s'-s}T^{\frac16-}\|u_h^{\mp}\|_{Y_{h,\mp}^{s',T}} \|v_h^{\mp}\|_{Y_{h,\mp}^{s',T}},
\end{aligned}\]
where all implicit constants are independent of $h>0$.
\end{corollary}

\begin{proof}
For given $0 < T  \le 1$, take $\eta\in C_0^\infty$ as a non-negative cut-off function supported on $[-2,2]$ and equals to $1$ in $[-1,1]$ such that 
\[\|\eta_T(t)u_h^{\pm}\|_{Y^s_{h,\pm}} \le 2 \|u_h^{\pm}\|_{Y^{s,T}_{h,\pm}}, \]
where $\eta_T(t) = \eta(\tfrac{t}{T})$. Then, by Lemma \ref{lem:bilinear1}\footnote{In view of the proof, one can find at least $1/6$ more dispersive smoothing effect.} and \eqref{ineq:time localization}, we have
\[\begin{aligned}
\left\| \nabla_h (u_h^{\pm}v_h^{\pm})\right\|_{Z^{s,T}_{h,\pm}} \lesssim&~{} \left\| \nabla_h ((\eta_Tu_h^{\pm}) \cdot (\eta_T v_h^{\pm}))\right\|_{Z^{s}_{h,\pm}}\\
\lesssim&~{} T^{\frac16-}\|\eta_T u_h^{\pm}\|_{X_{h,\pm}^{s,\frac12}}\|\eta_T u_h^{\pm}\|_{X_{h,\pm}^{s, \frac12}}\\
\lesssim&~{} T^{\frac16-}\|u_h^{\pm}\|_{Y_{h,\pm}^{s,T}}\| u_h^{\pm}\|_{Y_{h,\pm}^{s,T}}.
\end{aligned}\]
The rest follow analogously.
\end{proof}

\subsection{Uniform bounds for coupled and decoupled FPU}

\begin{proposition}[Uniform bounds for coupled and decoupled FPU]\label{prop:uniform bound}
Let $s\ge0$. For given initial data
\[   (u_{h,0}^+, u_{h,0}^-) \in \mathbb H^s(\T_h) \quad \mbox{with} \quad h\sum_{x \in \T_h} u_{h,0}^{\pm} = 0,\]
there exist $T=T(\| (u_{h,0}^+, u_{h,0}^-)\|_{\mathbb H^s(\T_h)})$ independent of $h\in(0,1]$ and a unique solution $(u_h^+,u_h^-)$ to the coupled FPU \eqref{coupled FPU'} ( resp., $(v_h^+,v_h^-)$ to the decoupled FPU \eqref{decoupled FPU'} ) with 
\[\begin{aligned}
  &\left\| u_h^{\pm}(t)\right\|_{Y_{h,\pm}^{s,T}}
  \ls \| (u_{h,0}^+, u_{h,0}^-)\|_{\mathbb H^s(\T_h)}, \\ 
  &\left( \text{resp.,} \left\| v_h^{\pm}(t)\right\|_{Y_{h,\pm}^{s,T}}
  \ls \| (u_{h,0}^+, u_{h,0}^-)\|_{\mathbb H^s(\T_h)} \right).
\end{aligned}\]
\end{proposition}

\begin{proof}
For sufficiently small $0<T\ll1$ to be chosen later, consider the nonlinear map $\Phi=(\Phi^+,\Phi^-)$ by 
\[\begin{aligned}
  \Phi^{\pm}(u_h^+,u_h^-)
  &:=S_h^\pm(t)u_{h,0}^\pm \\ 
  &\quad \mp\frac14 \int_0^t S_h^\pm(t-t') \nabla_h \Big(u_h^\pm(t')+e^{\pm\frac{2t'}{h^2}\partial_h}u_h^\mp(t')\Big)^2 dt'.
\end{aligned}\]
Let $\mathbb Y_h^s$ (analogously define $\mathbb Y^{s,T}_h$ as a time localized version) denote the solution space for $(u_h^+,u_h^-)$ equipped with the norm
\[ \|(u_h^+,u_h^-)\|_{\mathbb Y_h^s}^2:= \|u_h^+\|_{Y_{h,+}^s}^2 +\|u_h^-\|_{Y_{h,-}^s}^2.\]
By Lemma~\ref{Lem:Y} and Corollary~\ref{Cor:Linear Bilinear}, we have 
\[\left\| \Phi(u_h^+,u_h^-) \right\|_{\mathbb Y_{h,\pm}^{s,T}}  \le C\|(u_{h,0}^+,u_{h,0}^-)\|_{\mathbb H_h^s(\T_h)} + C'T^{\frac16-}\|(u_h^+,u_h^-)\|_{\mathbb Y_h^{s,T}}^2.\]
We, here, emphasize that the constants $C$ and $C'$ are independent of $h>0$. By taking $0 < T < 1$ such that 
\begin{equation}\label{eq:time choice}
CC'T^{\frac16-}\|(u_{h,0}^+,u_{h,0}^-)\|_{\mathbb H_h^s(\T_h)} < \frac14,
\end{equation}
one verifies that  $\Phi$ maps from the set 
\[\mathcal Y=\left\{ (u_h^+,u_h^-)\in \mathbb Y_h^{s,T} : \|(u_h^{+},u_h^{-})\|_{\mathbb Y_{h}^{s,T}} \le 2 C\|(u_{h,0}^+,u_{h,0}^-)\|_{\mathbb H_h^s(\T_h)} \right\}\]
to itself. The difference of two solutions can be treated similarly. Let $\mathcal N^{\pm}(u_h^+,u_h^-)(t)$ denote 
\[\mathcal{N}^{\pm}(u_h^+,u_h^-)(t)=\nabla_h \bigg\{\Big(u_h^\pm(t)\Big)^2+\Big( e^{\pm\frac{2t}{h^2}\partial_h}u_h^\mp(t)\Big)^2+2u_h^\pm(t)e^{\pm\frac{2t}{h^2}\partial_h}u_h^\mp(t)\bigg\}.\]
For $(u_h^+,u_h^-),(w_h^+,w_h^-) \in \mathcal Y$, we write
\[\begin{aligned}
  &\mathcal{N}^{\pm}(u_h^+,u_h^-)-\mathcal{N}^{\pm}(w_h^+,w_h^-) \\ 
  &=\nabla_h\Big\{ (u_h^{\pm}(t)+w_h^{\pm})(u_h^{\pm}(t)-w_h^{\pm})\Big\}   
  + \nabla_h\Big\{ \Big( e^{\pm\frac{2t}{h^2}\partial_h}(u_h^\mp(t)+w_h^\mp(t)) \Big) \Big( e^{\pm\frac{2t}{h^2}\partial_h}(u_h^\mp(t)-w_h^\mp(t)) \Big) \Big\} \\ 
  &\quad +\nabla_h\Big\{ \Big(u_h^\pm(t)-w_h^\pm(t)\Big) e^{\pm\frac{2t}{h^2}\partial_h}u_h^\mp(t)  \Big\}
  +\nabla_h\Big\{ u_h^\pm(t) e^{\pm\frac{2t}{h^2}\partial_h}\big( u_h^\mp(t)-w_h^\mp(t)\big)  \Big\},
\end{aligned}\]
which guarantees 
\[\begin{aligned}
&\le C'T^{\frac16-}(\|(u_h^+,u_h^-)\|_{\mathbb Y_h^s} + \|(w_h^+,w_h^-)\|_{\mathbb Y_h^s})\|(u_h^+,u_h^-)-(w_h^+,w_h^-)\|_{\mathbb Y_{h}^{s,T}}  \\ 
&<\frac12\|(u_h^{+},u_h^{-})-(w_h^{+},w_h^{-})\|_{\mathbb Y_{h}^{s,T}},
\end{aligned}\]
by taking the same $T > 0$ as in \eqref{eq:time choice}. Thus, we conclude that $\Phi$ is contractive on $\mathcal Y$. $(u_h^{+},u_h^{-})$ is a solution to \eqref{coupled FPU'} which, by uniqueness, coincides in $C_t([-T,T]:H^s(\T_h))$ with the solution constructed in Proposition~\ref{prop:LWP}.
\end{proof}

As a direct consequence of \eqref{Embedding Y} and proposition~\ref{prop:L4}, we have
\begin{corollary}[Uniform bounds for FPU solutions]\label{Cor:Uniform bounds}
Let $s\ge0$.  For given initial data
\[   (u_{h,0}^+, u_{h,0}^-) \in \mathbb H^s(\T_h) \quad \mbox{with} \quad h\sum_{x \in \T_h} u_{h,0}^{\pm} = 0,\]
Let $(u_h^+(t),u_h^-(t))$  (resp., $(v_h^+(t),v_h^-(t))$) be the solutions to the coupled FPU \eqref{coupled FPU'} (resp., to the decoupled FPU \eqref{decoupled FPU'}) given in Proposition~\ref{prop:uniform bound}. Then,
\[\begin{aligned}
  &~{}\left\| u_h^{\pm}(t) \right\|_{C([-T,T]:H^s(\T_h))} \ls \| u_{h,0}^{\pm}  \|_{H^s(\T_h)}, \qquad \| u_{h}^{\pm}\|_{L^4([0,T] \times \T_h)} \ls \|u_{h,0}^{\pm}\|_{H^s(\T_h)}, \\
  &~{}\left\| v_h^{\pm}(t) \right\|_{C([-T,T]:H^s(\T_h))} \ls \| u_{h,0}^{\pm}  \|_{H^s(\T_h)}, \qquad \| v_{h}^{\pm}\|_{L^4([0,T] \times \T_h)} \ls \|u_{h,0}^{\pm}\|_{H^s(\T_h)}. 
\end{aligned}\]
\end{corollary}

\subsection{Decoupling the FPU system: Proof of Proposition~\ref{Prop:Coupled to Decoupled}}

\begin{proof}[Proof of Proposition~\ref{Prop:Coupled to Decoupled}]
Let $T=T(\|(u_{h,0}^+,u_{h,0}^-)\|_{\mathbb H^s(\T_h)})>0$ be the common lifespan of solutions $(u_h^+(t),u_h^-(t))$ to the coupled FPU and $(v_h^+(t),v_h^-(t))$ to the decoupled FPU with initial data $(u_{h,0}^+,u_{h,0}^-)$ constructed in Proposition~\ref{prop:uniform bound}. Moreover, the solutions are uniformly (in $h$) bounded by the size of initial data
\begin{equation}\label{uniform bounds of sol}
 \|(u_h^{+},u_h^{-})\|_{\mathbb Y_{h}^{s,T}}, \|(v_h^{+},v_h^{-})\|_{\mathbb Y_{h}^{s,T}} \le 2 C\|(u_{h,0}^+,u_{h,0}^-)\|_{\mathbb H^s(\T_h)}.
\end{equation}
By subtracting \eqref{coupled FPU'} from \eqref{decoupled FPU'}, we obtain 
\[\begin{aligned}
u_h^{\pm}(t) &- v_h^{\pm}(t)
=\mp\int_0^t S_h^{\pm}(t-t')\nabla_h \left\{ (u_{h}^{\pm}(t')+v_h^{\pm}(t'))(u_{h}^{\pm}(t')-v_h^{\pm}(t')) \right\}dt'  \\ 
&\mp\int_0^t S_h^{\pm}(t-t')\nabla_h \bigg\{ \Big( e^{\pm\frac{2t'}{h^2}\partial_h}v_h^\mp(t')\Big)^2+2v_h^\pm(t')e^{\pm\frac{2t'}{h^2}\partial_h}v_h^\mp(t')\bigg \} dt'.
\end{aligned}\] 
By taking $Y_{h,\pm}^{0,T}$ norm and then applying Corollary~\ref{Cor:Linear Bilinear} to the quadratic terms, we estimate 
\[\begin{aligned}
  \|u_h^{\pm} - v_h^{\pm}\|_{Y_{h,\pm}^{0,T}}
  &\le C'T^{\frac16-}\Big(\|u_h^{+}\|_{Y_{h,+}^{0,T}} 
  + \|u_h^{-}\|_{Y_{h,-}^{0,T}} \Big) \|u_h^{\pm} - v_h^{\pm}\|_{Y_{h,\pm}^{0,T}} \\ 
  &\quad +C'T^{\frac16-}h^{\frac12+s}
  \|v_{h}^{\mp}\|_{Y_{h,\mp}^{s,T}}^2
  +C'T^{\frac16-}h^{s}
  \|v_{h}^{\pm}\|_{Y_{h,\pm}^{s,T}}
  \|v_{h}^{\mp}\|_{Y_{h,\mp}^{s,T}},
\end{aligned}\]
for some $C'>0$ uniform in $h\in(0,1]$. 
By shrinking the time interval, if necessary, such that 
$$CC'T^{\frac16-}\|(u_{h,0}^+,u_{h,0}^-)\|_{\mathbb H^s(\T_h)}<\frac18,$$
we obtain from the uniform bounds of solutions \eqref{uniform bounds of sol} that 
\[  \|u_h^{\pm} - v_h^{\pm}\|_{Y_{h,\pm}^{0,T}}
  \ls h^s \|(u_{h,0}^+,u_{h,0}^-)\|_{\mathbb H^s(\T_h)}^2.\]
Then, the embedding $Y_{h,\pm}^{0,T}\hookrightarrow C_t([-T,T]:L^2(\T_h))$ gives the desired results.

\end{proof}

\section{Reformulations : Normal form reduction method}\label{sec:regularization}
Due to the lack of smoothing effect, it  cannot be shown directly that the decoupled system \eqref{decoupled FPU'} converges to the KdV system \eqref{kdv integral form} similarly as in \cite{HKY2021}. Thus, we reformulate the decoupled FPU \eqref{decoupled FPU'} using the argument in \cite{Babin2011} as well as the corresponding KdV system \eqref{kdv integral form}. 

By definitions of the Fourier transforms \eqref{Fourier transform} and their inversion formula \eqref{Inverse Fourier transform}, the constant $\displaystyle \frac{1}{\sqrt{2\pi}}$ appears when taking the Fourier transform to $(\mathcal{V}_h^{\pm})^2$ and $(\mathcal{W}^{\pm})^2$, see \eqref{FT of product} and \eqref{FT of product-conti}. Throughout this section, we drop the constant $\displaystyle \frac{1}{\sqrt{2\pi}}$ from \eqref{FT of product} and \eqref{FT of product-conti} for the sake of convenience. 
\subsection{Regularizing the FPU system via the Normal form approach}
Let us recall the decoupled FPU \eqref{decoupled FPU'} on $\T_h$ with initial data $u_{h,0}^{\pm}$ as
\[ v_h^{\pm}(t) = S_h^{\pm}(t) u_{h,0}^{\pm} \mp \frac14 \int_0^t S_h^{\pm}(t-t') \nabla_h\left(v_h^{\pm}(t')\right)^2 dt', \]
where $S_h^\pm(t)=e^{\mp \frac{t}{h^2}(\nabla_h-\partial_h)}$. The corresponding KdV system \eqref{kdv integral form} on $\T$ with initial data $\mathcal{L}_hu_{h,0}^{\pm}$ is given by
\[  w^{\pm}(t) = S^{\pm}(t)(\mathcal{L}_h u_{h,0}^{\pm} ) \mp \frac14 \int_0^t S^{\pm}(t-t') \partial_x\left(w^{\pm}(t')\right)^2 dt',\]
where $S^{\pm}(t)=e^{\mp\frac{t}{24}\partial_x^3 }$. Note from Lemma~\ref{Lem:discretization linearization inequality} and Lemma~\ref{Lem:mean} that  
 \[ \|\mathcal{L}_h u_{h,0}^{\pm}\|_{H^s(\T)} \ls \|u_{h,0}^{\pm}\|_{H^s(\T_h)}, \; \text{ and } \; \int_{\T} \mathcal{L}_h u_{h,0}^{\pm}(x) \; dx= 0.\]
Let us define a profile $\mathcal{V}_h^{\pm}$ by 
\begin{equation}\label{eq:renormalization}
  \widehat{\mathcal{V}_h^{\pm} }(t,k)=e^{\pm it\frac{1}{h^2}\left(\frac{2}{h}\sin(\frac{hk}{2})-k\right)}\widehat{v}_h^{\pm}(t,k), \quad \text{ equivalently, } \quad \mathcal{V}^{\pm}_h(t)=S_h^{\pm}(-t)v_h^{\pm}(t).
\end{equation}
Note that 
\[  \widehat{\mathcal{V}_h^{\pm}}(t,0) 
  =\widehat{v_h^{\pm}}(t,0) 
  =h\sum_{x\in\T_h}v_h^{\pm}(t,x)=0 \; \text{ for all } t.\]
Then, $\mathcal{V}_h^{\pm}$ (indeed, $\widehat{\mathcal{V}_h^{\pm}}$) solves the differential equation
\begin{equation}\label{time derivative of v}
  \partial_t   \widehat{\mathcal{V}_h^{\pm}}(t,k) =  \mp \frac14 \frac{2i}{h}\sin\left(\frac{hk}{2}\right) \sum_{\substack{k=k_1+k_2 \\ k_1k_2\neq0}}  e^{\pm it \Phi_h^2(k,k_1,k_2)}   \widehat{\mathcal{V}_h^{\pm}}(t,k_1)   \widehat{\mathcal{V}_h^{\pm}}(t,k_2)=:\mathcal F_h \textup{B}_1^{\pm}(\mathcal{V}_h^{\pm},\mathcal{V}_h^{\pm})(t,k),
\end{equation}
or the integral equation
\begin{equation}\label{eq:FT of V}
\begin{aligned}
  \widehat{\mathcal{V}_h^{\pm}}(t,k) &=  \widehat{u_{h,0}^{\pm}}(k) \mp  \frac14 \frac{2i}{h}\sin\left(\frac{hk}{2}\right) \int_0^t \sum_{ \substack{k=k_1+k_2 \\ k_1k_2\neq0}}  e^{\pm it' \Phi_h^2(k,k_1,k_2) }   \widehat{\mathcal{V}_h^{\pm}}(t',k_1)  \widehat{\mathcal{V}_h^{\pm}}(t',k_2) \; dt',
\end{aligned}
\end{equation}
where
\begin{equation}\label{eq:Resonance_2}
\begin{aligned}
\Phi_h^2(k,k_1,k_2) =&~{}\frac{1}{h^2}\left\{\left(\frac{2}{h}\sin\left(\frac{hk}{2}\right)-k\right) -\left(\frac{2}{h}\sin\left(\frac{hk_1}{2}\right)-k_1\right) -\left(\frac{2}{h}\sin\left(\frac{hk_2}{2}\right)-k_2\right)\right\}\\
=&~{} -\frac{8}{h^3}\sin\left(\frac{hk_1}{4}\right)\sin\left(\frac{hk_2}{4}\right) \sin\left(\frac{hk}{4}\right).
\end{aligned}
\end{equation}
Note that $\Phi_h^2$ is symmetric over the variables. Let us define a bilinear form
\[\mathcal F_h \textup{B}_2^{\pm}(\mathcal{V}_h^{\pm},\mathcal{V}_h^{\pm})(t,k) = \sum_{\substack{k=k_1+k_2 \\ k_1k_2 \neq 0}}\frac{\frac{2i}{h}\sin\left(\frac{hk}{2}\right)}{i\Phi_h^2(k,k_1,k_2)}e^{\pm it\Phi_h^2(k,k_1,k_2)}\wh{\mathcal{V}_h^{\pm}}(t,k_1)\wh{\mathcal{V}_h^{\pm}}(t,k_2).\]
Then, by symmetry, a direct computation gives
\[\begin{aligned}
\widehat{\mathcal{V}_h^{\pm}}(t,k) =&~{}  \widehat{u_{h,0}^{\pm}}(k) -\frac14\mathcal F_h\textup{B}_2^{\pm}(\mathcal{V}_h^{\pm},\mathcal{V}_h^{\pm})(t,k) + \frac14\mathcal F_h\textup{B}_2^{\pm}(\mathcal{V}_h^{\pm},\mathcal{V}_h^{\pm})(0,k) \\
&~{}+ \frac12\int_0^t \mathcal F_h\textup{B}_2^{\pm}(\mathcal{V}_h^{\pm}, \partial_t \mathcal{V}_h^{\pm})(t',k) \; dt'.
\end{aligned}\]
Using \eqref{time derivative of v}  and \eqref{eq:Resonance_2}, one has  
\[\begin{aligned}
&~{}\int_0^t\mathcal{F}_h\textup{B}_2^{\pm}(\mathcal{V}_h^{\pm},\partial_t\mathcal{V}_h^{\pm})(t',k) \; dt'\\
=&~{} \int_0^t \sum_{\substack{k=k_1+k_2 \\ k_1k_2 \neq 0}}\frac{\frac{2i}{h}\sin\left(\frac{hk}{2}\right)}{i\Phi_h^2(k,k_1,k_2)}e^{\pm it'\Phi_h^2(k,k_1,k_2)}\widehat{\mathcal{V}_h^{\pm}}(t',k_1)\partial_t\widehat{\mathcal{V}_h^{\pm}}(t',k_2) \; dt'\\
=&~{}\mp \frac14 \int_0^t \sum_{\substack{k=k_1+k_2 \\ k_1k_2 \neq 0}}\frac{\frac{2i}{h}\sin\left(\frac{hk}{2}\right)}{i\Phi_h^2(k,k_1,k_2)}e^{\pm it'\Phi_h^2(k,k_1,k_2)}\widehat{\mathcal{V}_h^{\pm}}(t',k_1)\\
&~{} \hspace{3em}\times \frac{2i}{h}\sin\left(\frac{hk_2}{2}\right) \sum_{\substack{k_2=k_{21}+k_{22} \\ k_{21}k_{22}\neq0}}  e^{\pm it' \Phi_h^2(k_2,k_{21},k_{22})}   \widehat{\mathcal{V}_h^{\pm}}(t',k_{21})   \widehat{\mathcal{V}_h^{\pm}}(t',k_{22}) \; dt' \\
=&~{}\int_0^t \sum_{\substack{k=k_1+k_2+k_3\\k_1k_2k_3 \neq 0 \\ k_2 + k_3 \neq 0}} \frac{\pm 2i \cos\left(\frac{hk}{4}\right)\cos\left(\frac{h(k_2+k_3)}{4}\right)}{\frac4h\sin\left(\frac{hk_1}{4}\right)}e^{\pm it' \Phi_h^3(k,k_1,k_2,k_3)}\widehat{\mathcal{V}_h^{\pm}}(t',k_1)\widehat{\mathcal{V}_h^{\pm}}(t',k_2)\widehat{\mathcal{V}_h^{\pm}}(t',k_3)\;dt'\\
=:&~{} \int_0^t \mathcal{F}_h \mathcal{N}_{h,3}^{\pm}(\mathcal{V}_h^{\pm},\mathcal{V}_h^{\pm},\mathcal{V}_h^{\pm})(t',k) \; dt',
\end{aligned}\]
where 
\begin{equation}\label{eq:resonance_3}
\begin{aligned}
\Phi_h^3(k,k_1,k_2,k_3) :=&~{} \Phi_h^2(k,k_1,k_2+k_3) + \Phi_h^2(k_2+k_3,k_2,k_3)\\
=&~{} -\frac{8}{h^3}\sin\left(\frac{h(k_1+k_2)}{4}\right)\sin\left(\frac{h(k_2+k_3)}{4}\right) \sin\left(\frac{h(k_3+k_1)}{4}\right).
\end{aligned}
\end{equation}
Note that $\Phi_h^3$ is symmetric over the variables. We further decompose the summation in $\mathcal{F}_h\mathcal{N}_{h,3}^{\pm}(\mathcal{V}_h^{\pm},\mathcal{V}_h^{\pm},\mathcal{V}_h^{\pm})$ into two parts:
\[\sum_{\substack{k=k_1+k_2+k_3\\k_1k_2k_3 \neq 0 \\ k_2 + k_3 \neq 0}} = \sum_{\substack{k=k_1+k_2+k_3,  \;k_1k_2k_3\neq0 \\ k_2+k_3\neq0 ,  \; (k_1+k_2)(k_1+k_3)=0}} \; + \; \sum_{\substack{k=k_1+k_2+k_3,  \;k_1k_2k_3\neq0 \\  (k_1+k_2)(k_2+k_3)(k_1+k_3)\neq0 }},\]
and their corresponding parts of $\mathcal{F}_h\mathcal{N}_{h,3}^{\pm}(\mathcal{V}_h^{\pm},\mathcal{V}_h^{\pm},\mathcal{V}_h^{\pm})$ are respectively denoted by $(\mathcal{F}_h\mathcal{N}_{h,3}^{\pm})_\mathcal{R}$ and $(\mathcal{F}_h\mathcal{N}_{h,3}^{\pm})_\mathcal{NR}$, thus, we write
\begin{equation}\label{decomposition of B}
  \mathcal{F}_h\mathcal{N}_{h,3}^{\pm}(\mathcal{V}_h^{\pm},\mathcal{V}_h^{\pm},\mathcal{V}_h^{\pm})(t',k) = (\mathcal{F}_h\mathcal{N}_{h,3}^{\pm})_\mathcal{R}(\mathcal{V}_h^{\pm},\mathcal{V}_h^{\pm},\mathcal{V}_h^{\pm})(t',k)
  +(\mathcal{F}_h\mathcal{N}_{h,3}^{\pm})_{\mathcal{NR}}(\mathcal{V}_h^{\pm},\mathcal{V}_h^{\pm},\mathcal{V}_h^{\pm})(t',k).
\end{equation}
Here, by the subscript $\mathcal{R}$ and $\mathcal{NR}$, we mean the resonant and nonresonant terms, respectively.

First, consider the resonant term $(\mathcal{F}_h\mathcal{N}_{h,3}^{\pm})_\mathcal{R}$. The set of frequencies under
\[k_2+k_3\neq0, (k_1+k_2)(k_1+k_3)=0\]
can be divided into the following three sets
\[\begin{aligned}
  \mathcal{R}_{h,1} & =\{ k_1+k_2=0 \} \cap \{ k_1+k_3=0 \} \cap \{ k_2+k_3\neq 0 \}, \\ 
  \mathcal{R}_{h,2} & =\{ k_1+k_2=0 \} \cap \{ k_1+k_3\neq 0 \} \cap \{ k_2+k_3\neq 0 \}, \\ 
  \mathcal{R}_{h,3}& =\{ k_1+k_2\neq 0 \} \cap \{ k_1+k_3=0 \} \cap \{ k_2+k_3\neq 0 \}.
\end{aligned}\]

Over $\mathcal{R}_{h,1}$, we know $k_1=-k_2=-k_3$ and $k=-k_1$, which assure 
\[(\mathcal{F}_h\mathcal{N}_{h,3}^{\pm})_{\mathcal{R}}(\mathcal{V}_h^{\pm},\mathcal{V}_h^{\pm},\mathcal{V}_h^{\pm})(t,k)|_{\mathcal R_{h,1}} = \mp2i \frac{\cos\left(\frac{hk}{2}\right)\cos\left(\frac{hk}{4}\right)}{\frac{4}{h}\sin\left(\frac{hk}{4}\right)}\left|\widehat{\mathcal{V}_h^{\pm}}(t,k)\right|^2\widehat{\mathcal{V}_h^{\pm}}(t,k),\]
due to $\displaystyle \widehat{\mathcal{V}_h^{\pm}}(t,-k) = \overline{\widehat{\mathcal{V}_h^{\pm}}(t,k)}$, where $\overline{z}$, $z \in \C$, is the complex conjugate of $z$. Here $f|_A$ denotes the restriction of $f$ on the set $A$.

Under the symmetry over frequencies, we know
\[(\mathcal{F}_h\mathcal{N}_{h,3}^{\pm})_\mathcal{R}(\mathcal{V}_h^{\pm},\mathcal{V}_h^{\pm},\mathcal{V}_h^{\pm})(t,k)|_{\mathcal{R}_{h,2}} = (\mathcal{F}_h\mathcal{N}_{h,3}^{\pm})_\mathcal{R}(\mathcal{V}_h^{\pm},\mathcal{V}_h^{\pm},\mathcal{V}_h^{\pm})(t,k)|_{\mathcal{R}_{h,3}}.\]
Over $\mathcal{R}_{h,2}$, we know $k_1=-k_2\neq \pm k_3$ and $k=k_3$, which assure 
\[(\mathcal{F}_h\mathcal{N}_{h,3}^{\pm})_\mathcal{R}(\mathcal{V}_h^{\pm},\mathcal{V}_h^{\pm},\mathcal{V}_h^{\pm})(t,k)|_{\mathcal{R}_{h,2}} = \pm 2i \cos\left(\frac{hk}{4}\right)\widehat{\mathcal{V}_h^{\pm}}(t,k)\sum_{\substack{k_1 \in (\T_h)^* \\ k_1 \neq \pm k}} \frac{\cos\left(\frac{h(k-k_1)}{4}\right)}{\frac{4}{h}\sin\left(\frac{hk_1}{4}\right)}|\widehat{\mathcal{V}_h^{\pm}}(t,k_1)|^2.\]
From the angle difference identity for the cosine function, and symmetry and anti-symmetry for cosine and sine functions, respectively, we have
\[\begin{aligned}
&~{}\sum_{\substack{k_1 \in (\T_h)^* \\ k_1 \neq \pm k}} \frac{\cos\left(\frac{h(k-k_1)}{4}\right)}{\frac{4}{h}\sin\left(\frac{hk_1}{4}\right)}|\widehat{\mathcal{V}_h^{\pm}}(t,k_1)|^2 \\
=&~{} \sum_{\substack{k_1 \in (\T_h)^* \\ k_1 \neq \pm k}} \frac{\cos\left(\frac{hk}{4}\right)\cos\left(\frac{hk_1}{4}\right)+\sin\left(\frac{hk}{4}\right)\sin\left(\frac{hk_1}{4}\right)}{\frac{4}{h}\sin\left(\frac{hk_1}{4}\right)}|\widehat{\mathcal{V}_h^{\pm}}(t,k_1)|^2\\
=&~{} \frac{h}{4}\sin\left(\frac{hk}{4}\right)\sum_{\substack{k_1 \in (\T_h)^* \\ k_1 \neq \pm k}} |\widehat{\mathcal{V}_h^{\pm}}(t,k_1)|^2,
\end{aligned}\]
which concludes
\[(\mathcal{F}_h\mathcal{N}_{h,3}^{\pm})_\mathcal{R}(\mathcal{V}_h^{\pm},\mathcal{V}_h^{\pm},\mathcal{V}_h^{\pm})(t,k)|_{\mathcal{R}_{h,2} \cup \mathcal{R}_{h,3}} =\pm \frac{ih}{2}\sin\left(\frac{hk}{2}\right)\widehat{\mathcal{V}_h^{\pm}}(t,k)\sum_{\substack{k_1 \in (\T_h)^* \\ k_1 \neq \pm k}} |\widehat{\mathcal{V}_h^{\pm}}(t,k_1)|^2.
\]
Collecting all, we obtain 
\[\begin{aligned}
(\mathcal{F}_h\mathcal{N}_{h,3}^{\pm})_\mathcal{R}(\mathcal{V}_h^{\pm},\mathcal{V}_h^{\pm},\mathcal{V}_h^{\pm})(t,k)
  =&~{} \mp2i \frac{\cos\left(\frac{hk}{2}\right)\cos\left(\frac{hk}{4}\right)}{\frac{4}{h}\sin\left(\frac{hk}{4}\right)}\left|\widehat{\mathcal{V}_h^{\pm}}(t,k)\right|^2\widehat{\mathcal{V}_h^{\pm}}(t,k)  \\ 
&~{}\pm \frac{ih}{2}\sin\left(\frac{hk}{2}\right)\widehat{\mathcal{V}_h^{\pm}}(t,k)\sum_{\substack{k_1 \in (\T_h)^* \\ k_1 \neq \pm k}} |\widehat{\mathcal{V}_h^{\pm}}(t,k_1)|^2.
\end{aligned}\]

Now, we consider the nonresonant term $(\mathcal{F}_h\mathcal{N}_{h,3}^{\pm})_{\mathcal{NR}}$ in \eqref{decomposition of B}. Let us define the (cubic) non-resonant set
\begin{equation}\label{eq:A(k)}
  \mathcal{A}(k)=\big\{ (k_1,k_2,k_3)\in ((\T_h)^*)^3: \; k=k_1+k_2+k_3,\; k_1k_2k_3\neq0, \; (k_1+k_2)(k_2+k_3)(k_1+k_3)\neq0\; \big\}
\end{equation}
and a trilinear form
\[\mathcal{F}_h \textup{B}_3^{\pm}(\mathcal{V}_h^{\pm},\mathcal{V}_h^{\pm},\mathcal{V}_h^{\pm})(t,k) := \sum_{(k_1,k_2,k_3)\in\mathcal{A}(k)}
  \frac{2i\cos(\frac{hk}{4})\cos\left(\frac{h(k_2+k_3)}{4}\right)}
  {\frac{4i}{h}\sin(\frac{hk_1}{4})\Phi_h^3(k,k_1,k_2,k_3)} e^{ \pm it\Phi_h^3(k,k_1,k_2,k_3)}\prod_{j=1}^3\widehat{\mathcal{V}_h^{\pm}}(t,k_j).\]
Then, a direct computation under the symmetry on $k_2$ and $k_3$ variables in the multiplier gives 
\[\begin{aligned}
 \int_0^t (\mathcal{F}_h\mathcal{N}_{h,3}^{\pm})_{\mathcal{NR}}(\mathcal{V}_h^{\pm},\mathcal{V}_h^{\pm},\mathcal{V}_h^{\pm})(t',k)dt'
  =&~{}\mathcal{F}_h{\textup{B}_3^{\pm}}(\mathcal{V}_h^{\pm},\mathcal{V}_h^{\pm},\mathcal{V}_h^{\pm})(0,k)
  +\mathcal{F}_h{\textup{B}_3^{\pm}}(\mathcal{V}_h^{\pm},\mathcal{V}_h^{\pm},\mathcal{V}_h^{\pm})(t,k) \\ 
  &~{}- \int_0^t\mathcal{F}_h{\textup{B}_3^{\pm}}(\partial_t\mathcal{V}_h^{\pm},\mathcal{V}_h^{\pm},\mathcal{V}_h^{\pm})(t',k) \; dt'\\
&~{}-2 \int_0^t\mathcal{F}_h{\textup{B}_3^{\pm}}(\mathcal{V}_h^{\pm},\mathcal{V}_h^{\pm},\partial_t\mathcal{V}_h^{\pm})(t',k) \; dt'.
\end{aligned}\]
Using \eqref{time derivative of v} and \eqref{eq:resonance_3}, one has 
\[\begin{aligned}
  \mathcal{F}_h\textup{B}_3^{\pm}(\partial_t\mathcal{V}_h^{\pm},\mathcal{V}_h^{\pm},\mathcal{V}_h^{\pm})(t,k)
   &=\mp\frac{i}{2} \sum_{(k_1,k_2,k_3,k_3)\in\mathcal{B}(k)}
   \frac{\cos\left(\frac{hk}{4}\right)\cos\left(\frac{h(k_1+k_2)}{4}\right)\cos\left(\frac{h(k_3+k_4)}{4}\right)}{\Phi_h^3(k,k_1+k_2,k_3,k_4)} \\   
&~{}\times e^{ \pm it\Phi_h^4(k,k_1,k_2,k_3,k_4) } \widehat{\mathcal{V}_h^{\pm}}(t,k_1) \widehat{\mathcal{V}_h^{\pm}}(t,k_2) \widehat{\mathcal{V}_h^{\pm}}(t,k_3) \widehat{\mathcal{V}_h^{\pm}}(t,k_4),
 \end{aligned}\]
 where the set $\mathcal{B}(k)$ of frequencies is given by
\begin{equation}\label{eq:B(k)}
\mathcal{B}(k)=  \left\{
(k_1,k_2,k_3,k_4) \in ((\T_h)^*)^4 : \begin{array}{l}
k= k_1+k_2+k_3+k_4, \;  k_1k_2k_3k_4 \neq 0,\\
(k_1+k_2)(k_1+k_2+k_3)(k_1+k_2+k_4)(k_3+k_4)\neq0
\end{array} \right\},
\end{equation} 
and the quartic resonant function is given by
\[
\Phi_h^4(k,k_1,k_2,k_3,k_4)= \Phi_h^3(k,k_1+k_2,k_3,k_4) + \Phi_h^2(k_1+k_2,k_1,k_2).\\
\]   
Similarly, we also have (by changing the variables)
\[\begin{aligned}
  \mathcal{F}_h\textup{B}_3^{\pm}(\mathcal{V}_h^{\pm},\mathcal{V}_h^{\pm},\partial_t\mathcal{V}_h^{\pm})(t,k)
   &=\mp \frac{i}{4} \sum_{(k_1,k_2,k_3,k_3)\in \mathcal{B}(k)}
   \frac{\cos\left(\frac{hk}{4}\right)\cos\left(\frac{h(k_1+k_2+k_4)}{4}\right)\sin\left(\frac{h(k_1+k_2)}{2}\right)}{\sin\left(\frac{hk_3}{4}\right)\Phi_h^3(k,k_1+k_2,k_3,k_4)} \\   
&~{}\times e^{ \pm it\Phi_h^4(k,k_1,k_2,k_3,k_4) } \widehat{\mathcal{V}_h^{\pm}}(t,k_1) \widehat{\mathcal{V}_h^{\pm}}(t,k_2) \widehat{\mathcal{V}_h^{\pm}}(t,k_3) \widehat{\mathcal{V}_h^{\pm}}(t,k_4).
 \end{aligned}\]
Thus,
\[\begin{aligned}
&~{}\mathcal{F}_h\textup{B}_3^{\pm}(\partial_t\mathcal{V}_h^{\pm},\mathcal{V}_h^{\pm},\mathcal{V}_h^{\pm})(t,k) + 2\mathcal{F}_h\textup{B}_3^{\pm}(\mathcal{V}_h^{\pm},\mathcal{V}_h^{\pm},\partial_t\mathcal{V}_h^{\pm})(t,k)\\
=&~{}\mp\frac{i}{2} \sum_{(k_1,k_2,k_3,k_3)\in\mathcal{B}(k)} \mathfrak{M}(k,k_1,k_2,k_3,k_4) e^{ \pm it\Phi_h^4(k,k_1,k_2,k_3,k_4) } \widehat{\mathcal{V}_h^{\pm}}(t,k_1) \widehat{\mathcal{V}_h^{\pm}}(t,k_2) \widehat{\mathcal{V}_h^{\pm}}(t,k_3) \widehat{\mathcal{V}_h^{\pm}}(t,k_4)\\
=:&~{}\mathcal{F}_h\textup{B}_4^{\pm}(\mathcal{V}_h^{\pm},\mathcal{V}_h^{\pm},\mathcal{V}_h^{\pm},\mathcal{V}_h^{\pm})(t,k),
\end{aligned}\]
where
\[\mathfrak{M}(k,k_1,k_2,k_3,k_4) = \frac{\cos\left(\frac{hk}{4}\right)\cos\left(\frac{h(k_{1+k_2})}{4}\right)\cos\left(\frac{h(k_3+k_4)}{4}\right)}{\Phi_h^3(k,k_1+k_2,k_3,k_4)} + \frac{\cos\left(\frac{hk}{4}\right)\cos\left(\frac{h(k_1+k_2+k_4)}{4}\right)\sin\left(\frac{h(k_1+k_2)}{2}\right) }{\sin\left(\frac{hk_3}{4}\right)\Phi_h^3(k,k_1+k_2,k_3,k_4)}\]
Collecting all, we conclude that

\begin{equation}\label{eq:renormalized FPU}
\begin{aligned}
\widehat{\mathcal{V}_h^{\pm}}(t,k) =&~{}  \widehat{u_{h,0}^{\pm}}(k)-\frac14\mathcal F_h\textup{B}_2^{\pm}(\mathcal{V}_h^{\pm},\mathcal{V}_h^{\pm})(t,k) + \frac14\mathcal F_h\textup{B}_2^{\pm}(\mathcal{V}_h^{\pm},\mathcal{V}_h^{\pm})(0,k) \\
&~{}+\int_0^t (\mathcal F_h\mathcal{N}_{h,3}^{\pm})_\mathcal{R}(\mathcal{V}_h^{\pm},\mathcal{V}_h^{\pm},\mathcal{V}_h^{\pm})(t',k) \; dt'\\
&~{}+ \frac12\mathcal{F}_h{\textup{B}_3^{\pm}}(\mathcal{V}_h^{\pm},\mathcal{V}_h^{\pm},\mathcal{V}_h^{\pm})(0,k)
  +\frac12\mathcal{F}_h{\textup{B}_3^{\pm}}(\mathcal{V}_h^{\pm},\mathcal{V}_h^{\pm},\mathcal{V}_h^{\pm})(t,k) \\ 
&~{}-\frac12 \int_0^t\mathcal{F}_h\textup{B}_4^{\pm}(\mathcal{V}_h^{\pm},\mathcal{V}_h^{\pm},\mathcal{V}_h^{\pm},\mathcal{V}_h^{\pm})(t',k) \; dt'.
\end{aligned}
\end{equation}

\subsection{Regularizing the KdV equation}
We repeat the argument as in the previous subsection for KdV equation, already done by \cite{Babin2011}, but we briefly arrange the computation for the sake of self-containedness and reader's convenience.

Let us define a profile $\mathcal{W}^{\pm}$ by 
\begin{equation}\label{eq:renormalization of w}
  \widehat{\mathcal{W}^{\pm} }(t,k)=e^{\mp it\frac{k^3}{24}}\widehat{w}^{\pm}(t,k), \quad \text{ equivalently, } \quad \mathcal{W}^{\pm}(t)=S^{\pm}(-t)w^{\pm}(t).
\end{equation}
Note that 
\[\widehat{\mathcal{W}^{\pm}}(t,0) 
  =\widehat{w^{\pm}}(t,0) 
  =\int_{\T} w^{\pm}(t,x) \; dx =0 \; \text{ for all } t.\]
Then, $\mathcal{W}^{\pm}$ (indeed, $\widehat{\mathcal{W}^{\pm}}$) solves the differential equation
\begin{equation}\label{time derivative of w}
  \partial_t   \widehat{\mathcal{W}^{\pm}}(t,k) =  \mp \frac14 (ik) \sum_{\substack{k=k_1+k_2 \\ k_1k_2\neq0}}  e^{\pm it \Psi^2(k,k_1,k_2)}   \widehat{\mathcal{W}^{\pm}}(t,k_1)   \widehat{\mathcal{W}^{\pm}}(t,k_2) =:\mathcal F \textup{D}_1^{\pm}(\mathcal{W}^{\pm},\mathcal{W}^{\pm})(t,k),
\end{equation}
or the integral equation
\begin{equation}\label{eq:FT of W}
\begin{aligned}
  \widehat{\mathcal{W}^{\pm}}(t,k) &=  \widehat{\mathcal{L}_h u_{h,0}^{\pm} }(k) \mp  \frac14 (ik) \int_0^t \sum_{ \substack{k=k_1+k_2 \\ k_1k_2\neq0}}  e^{\pm it' \Psi^2(k,k_1,k_2) }   \widehat{\mathcal{W}^{\pm}}(t',k_1)  \widehat{\mathcal{W}^{\pm}}(t',k_2) \; dt',
\end{aligned}
\end{equation}
where
\begin{equation}\label{eq:Resonance_2 of W}
\Psi^2(k,k_1,k_2) =-\frac{1}{24}(k^3 - k_1^3 - k_2^3) = -\frac18kk_1k_2.
\end{equation}
Note that $\Psi^2$ is symmetric over the variables. Let us define a bilinear form
\[\mathcal F \textup{D}_2^{\pm}(\mathcal{W}^{\pm},\mathcal{W}^{\pm})(t,k) = \sum_{\substack{k=k_1+k_2 \\ k_1k_2 \neq 0}}\frac{ik}{i\Psi^2(k,k_1,k_2)}e^{\pm it\Psi^2(k,k_1,k_2)}\wh{\mathcal{W}^{\pm}}(t,k_1)\wh{\mathcal{W}^{\pm}}(t,k_2).\]
Then, by symmetry, a direct computation gives
\[\begin{aligned}
\widehat{\mathcal{W}^{\pm}}(t,k) =&~{}  \widehat{\mathcal{L}_h u_{h,0}^{\pm}}(k) -\frac14\mathcal F\textup{D}_2^{\pm}(\mathcal{W}^{\pm},\mathcal{W}^{\pm})(t,k) + \frac14\mathcal F\textup{D}_2^{\pm}(\mathcal{W}^{\pm},\mathcal{W}^{\pm})(0,k) \\
&~{}+ \frac12\int_0^t \mathcal F\textup{D}_2^{\pm}(\mathcal{W}^{\pm}, \partial_t \mathcal{W}^{\pm})(t',k) \; dt'.
\end{aligned}\]
Using \eqref{time derivative of w}  and \eqref{eq:Resonance_2 of W}, one has  
\[\begin{aligned}
&~{}\int_0^t\mathcal{F}\textup{D}_2^{\pm}(\mathcal{W}^{\pm},\partial_t\mathcal{W}^{\pm})(t',k)dt'\\
=&~{} \int_0^t \sum_{\substack{k=k_1+k_2 \\ k_1k_2 \neq 0}}\frac{ik}{i\Psi^2(k,k_1,k_2)}e^{\pm it'\Psi^2(k,k_1,k_2)}\widehat{\mathcal{W}^{\pm}}(t',k_1)\partial_t\widehat{\mathcal{W}^{\pm}}(t',k_2) \; dt'\\
=&~{}\mp \frac14 \int_0^t \sum_{\substack{k=k_1+k_2 \\ k_1k_2 \neq 0}}\frac{ik}{i\Psi^2(k,k_1,k_2)}e^{\pm it'\Psi^2(k,k_1,k_2)}\widehat{\mathcal{W}^{\pm}}(t',k_1)\\
&~{} \hspace{3em}\times (ik_2) \sum_{\substack{k_2=k_{21}+k_{22} \\ k_{21}k_{22}\neq0}}  e^{\pm it' \Psi^2(k_2,k_{21},k_{22})}   \widehat{\mathcal{W}^{\pm}}(t',k_{21})   \widehat{\mathcal{W}^{\pm}}(t',k_{22}) \; dt' \\
=&~{}\int_0^t \sum_{\substack{k=k_1+k_2+k_3\\k_1k_2k_3 \neq 0 \\ k_2 + k_3 \neq 0}} \frac{\pm 2i}{k_1}e^{\pm it' \Psi^3(k,k_1,k_2,k_3)}\widehat{\mathcal{W}^{\pm}}(t',k_1)\widehat{\mathcal{W}^{\pm}}(t',k_3)\widehat{\mathcal{W}^{\pm}}(t',k_3)\;dt'\\
=:&~{} \int_0^t \mathcal{F} \mathcal{N}_3^{\pm}(\mathcal{W}^{\pm},\mathcal{W}^{\pm},\mathcal{W}^{\pm})(t',k) \; dt',
\end{aligned}\]
where 
\begin{equation}\label{eq:resonance_3 of W}
\Psi^3(k,k_1,k_2,k_3) := \Psi^2(k,k_1,k_2+k_3) + \Psi^2(k_2+k_3,k_2,k_3) = -\frac18(k_1+k_2)(k_2+k_3)(k_1+k_3).
\end{equation}
Note that $\Psi_h^3$ is symmetric over the variables. We further decompose the summation in $\mathcal{F} \mathcal{N}_3^{\pm}(\mathcal{W}^{\pm},\mathcal{W}^{\pm},\mathcal{W}^{\pm})$ into two parts:
\[\sum_{\substack{k=k_1+k_2+k_3\\k_1k_2k_3 \neq 0 \\ k_2 + k_3 \neq 0}} = \sum_{\substack{k=k_1+k_2+k_3,  \;k_1k_2k_3\neq0 \\ k_2+k_3\neq0 ,  \; (k_1+k_2)(k_1+k_3)=0}} \; + \; \sum_{\substack{k=k_1+k_2+k_3,  \;k_1k_2k_3\neq0 \\  (k_1+k_2)(k_2+k_3)(k_1+k_3)\neq0 }},\]
and their corresponding parts of $\mathcal{F} \mathcal{N}_3^{\pm}(\mathcal{W}^{\pm},\mathcal{W}^{\pm},\mathcal{W}^{\pm})$ are respectively denoted by $(\mathcal{F}\mathcal{N}_3^{\pm})_\mathcal{R}$ and $(\mathcal{F}\mathcal{N}_3^{\pm})_\mathcal{NR}$, thus, we write
\begin{equation}\label{decomposition of B of W}
  \mathcal{F}\mathcal{N}_3^{\pm}(\mathcal{W}^{\pm},\mathcal{W}^{\pm},\mathcal{W}^{\pm})(t',k) = (\mathcal{F}\mathcal{N}_3^{\pm})_\mathcal{R}(\mathcal{W}^{\pm},\mathcal{W}^{\pm},\mathcal{W}^{\pm})(t',k)
  +(\mathcal{F}\mathcal{N}_3^{\pm})_{\mathcal{NR}}(\mathcal{W}^{\pm},\mathcal{W}^{\pm},\mathcal{W}^{\pm})(t',k).
\end{equation}
Here, by the subscript $\mathcal{R}$ and $\mathcal{NR}$, we mean the resonant and nonresonant terms, respectively.

First, consider the resonant term $(\mathcal{F}\mathcal{N}_3^{\pm})_\mathcal{R}$. The set of frequencies under
\[k_2+k_3\neq0, (k_1+k_2)(k_1+k_3)=0\]
can be divided into the following three sets
\[\begin{aligned}
  \mathcal{R}_1 & =\{ k_1+k_2=0 \} \cap \{ k_1+k_3=0 \} \cap \{ k_2+k_3\neq 0 \}, \\ 
  \mathcal{R}_2 & =\{ k_1+k_2=0 \} \cap \{ k_1+k_3\neq 0 \} \cap \{ k_2+k_3\neq 0 \}, \\ 
  \mathcal{R}_3& =\{ k_1+k_2\neq 0 \} \cap \{ k_1+k_3=0 \} \cap \{ k_2+k_3\neq 0 \}.
\end{aligned}\]

Over $\mathcal{R}_1$, we know $k_1=-k_2=-k_3$ and $k=-k_1$, which assure 
\[(\mathcal{F}\mathcal{N}_3^{\pm})_\mathcal{R}(\mathcal{W}^{\pm},\mathcal{W}^{\pm},\mathcal{W}^{\pm})(t,k)|_{\mathcal R_1} = \mp\frac{2i}{k}\left|\widehat{\mathcal{W}^{\pm}}(t,k)\right|^2\widehat{\mathcal{W}^{\pm}}(t,k).\]

Over $\mathcal{R}_2$, we know $k_1=-k_2\neq \pm k_3$ and $k=k_3$, which assure 
\[(\mathcal{F}\mathcal{N}_3^{\pm})_\mathcal{R}(\mathcal{W}^{\pm},\mathcal{W}^{\pm},\mathcal{W}^{\pm})(t,k)|_{\mathcal R_2} = \pm 2i \widehat{\mathcal{W}^{\pm}}(t,k)\sum_{\substack{k_1 \in \Z \\ k_1 \neq \pm k}} \frac{1}{k_1}|\widehat{\mathcal{W}^{\pm}}(t,k_1)|^2 = 0,\]
due to the symmetric in $k_1$. Similarly, we have 
\[(\mathcal{F}\mathcal{N}_3^{\pm})_\mathcal{R}(\mathcal{W}^{\pm},\mathcal{W}^{\pm},\mathcal{W}^{\pm})(t,k)|_{\mathcal R_3} = \pm 2i \widehat{\mathcal{W}^{\pm}}(t,k)\sum_{\substack{k_1 \in \Z \\ k_1 \neq \pm k}} \frac{1}{k_1}|\widehat{\mathcal{W}^{\pm}}(t,k_1)|^2 = 0,\]
Collecting all, we obtain 
\[ (\mathcal{F}\mathcal{N}_3^{\pm})_\mathcal{R}(\mathcal{W}^{\pm},\mathcal{W}^{\pm},\mathcal{W}^{\pm})(t,k)= \mp\frac{2i}{k}\left|\widehat{\mathcal{W}^{\pm}}(t,k)\right|^2\widehat{\mathcal{W}^{\pm}}(t,k).\]

Now, we consider the nonresonant term $(\mathcal{F}\mathcal{N}_3^{\pm})_{\mathcal{NR}}$ in \eqref{decomposition of B of W}. Let us define the (cubic) non-resonant set
\[\Lambda(k)=\big\{ (k_1,k_2,k_3)\in \Z^3: \; k=k_1+k_2+k_3,\; k_1k_2k_3\neq0, \; (k_1+k_2)(k_2+k_3)(k_1+k_3)\neq0\; \big\}\]
and a trilinear form
\[\mathcal{F} \textup{D}_3^{\pm}(\mathcal{W}^{\pm},\mathcal{W}^{\pm},\mathcal{W}^{\pm})(t,k) := \sum_{(k_1,k_2,k_3)\in\Lambda(k)}
  \frac{2i}
  {ik_1\Psi^3(k,k_1,k_2,k_3)} e^{ \pm it\Psi^3(k,k_1,k_2,k_3)}\prod_{j=1}^3\widehat{\mathcal{W}^{\pm}}(t,k_j).\]
Then, a direct computation under the symmetric on $k_2$ and $k_2$ variables in the multiplier gives 
\[\begin{aligned}
 \int_0^t (\mathcal{F}\mathcal{N}_3^{\pm})_{\mathcal{NR}}(\mathcal{W}^{\pm},\mathcal{W}^{\pm},\mathcal{W}^{\pm})(t',k)dt'
  =&~{}\mathcal{F}{\textup{D}_3^{\pm}}(\mathcal{W}^{\pm},\mathcal{W}^{\pm},\mathcal{W}^{\pm})(0,k)
  +\mathcal{F}{\textup{D}_3^{\pm}}(\mathcal{W}^{\pm},\mathcal{W}^{\pm},\mathcal{W}^{\pm})(t,k) \\ 
  &~{}- \int_0^t\mathcal{F}{\textup{D}_3^{\pm}}(\partial_t\mathcal{W}^{\pm},\mathcal{W}^{\pm},\mathcal{W}^{\pm})(t',k) \; dt'\\
&~{}-2 \int_0^t\mathcal{F}{\textup{D}_3^{\pm}}(\mathcal{W}^{\pm},\mathcal{W}^{\pm},\partial_t\mathcal{W}^{\pm})(t',k) \; dt'.
\end{aligned}\]
Using \eqref{time derivative of w} and \eqref{eq:resonance_3 of W}, one has 
\[\begin{aligned}
  \mathcal{F}\textup{D}_3^{\pm}(\partial_t\mathcal{W}^{\pm},\mathcal{W}^{\pm},\mathcal{W}^{\pm})(t,k)
   &=\mp\frac{i}{2} \sum_{(k_1,k_2,k_3,k_3)\in \Gamma(k)}\frac{1}{\Psi^3(k,k_1+k_2,k_3,k_4)} \\   
&~{}\times e^{ \pm it\Psi^4(k,k_1,k_2,k_3,k_4) } \widehat{\mathcal{W}^{\pm}}(t,k_1) \widehat{\mathcal{W}^{\pm}}(t,k_2) \widehat{\mathcal{W}^{\pm}}(t,k_3) \widehat{\mathcal{W}^{\pm}}(t,k_4),
 \end{aligned}\]
 where the set $\Gamma(k)$ of frequencies is given by
\[\Gamma(k)=  \Bigg\{
(k_1,k_2,k_3,k_4) \in \Z^4 : \begin{array}{l}
k= k_1+k_2+k_3+k_4, \;  k_1k_2k_3k_4 \neq 0,\\
(k_1+k_2)(k_1+k_2+k_3)(k_1+k_2+k_4)(k_3+k_4)\neq0
\end{array} \Bigg\}, \] 
and the quartic resonant function is given by
\[
\Psi^4(k,k_1,k_2,k_3,k_4)= \Psi^3(k,k_1+k_2,k_3,k_4) + \Psi^2(k_1+k_2,k_1,k_2).\\
\]   
Similarly, we also have (by changing the variables)
\[\begin{aligned}
  \mathcal{F}\textup{D}_3^{\pm}(\mathcal{W}^{\pm},\mathcal{W}^{\pm},\partial_t\mathcal{W}^{\pm})(t,k)
   &=\mp \frac{i}{2} \sum_{(k_1,k_2,k_3,k_3)\in \Gamma(k)}
   \frac{k_1+k_2}{k_3\Psi^3(k,k_1+k_2,k_3,k_4)} \\   
&~{}\times e^{ \pm it\Psi^4(k,k_1,k_2,k_3,k_4) } \widehat{\mathcal{W}^{\pm}}(t,k_1) \widehat{\mathcal{W}^{\pm}}(t,k_2) \widehat{\mathcal{W}^{\pm}}(t,k_3) \widehat{\mathcal{W}^{\pm}}(t,k_4).
 \end{aligned}\]
Thus,
\[\begin{aligned}
&~{}\mathcal{F}\textup{D}_3^{\pm}(\partial_t\mathcal{W}^{\pm},\mathcal{W}^{\pm},\mathcal{W}^{\pm})(t,k) + 2\mathcal{F}\textup{D}_3^{\pm}(\mathcal{W}^{\pm},\mathcal{W}^{\pm},\partial_t\mathcal{W}^{\pm})(t,k)\\
=&~{}\mp\frac{i}{2} \sum_{(k_1,k_2,k_3,k_3)\in \Gamma(k)}\frac{2(k_1+k_2) + k_3}{k_3\Psi^3(k,k_1+k_2,k_3,k_4)}e^{ \pm it\Psi^4(k,k_1,k_2,k_3,k_4) } \prod_{j=1}^4\widehat{\mathcal{W}^{\pm}}(t,k_j)\\
=:&~{}\mathcal{F}\textup{D}_4^{\pm}(\mathcal{W}^{\pm},\mathcal{W}^{\pm},\mathcal{W}^{\pm},\mathcal{W}^{\pm})(t,k).
\end{aligned}\]
Collecting all, we conclude that
\begin{equation}\label{eq:renormalized KdV}
\begin{aligned}
\widehat{\mathcal{W}^{\pm}}(t,k) =&~{}  \widehat{\mathcal{L}_h u_{h,0}^{\pm}}(k) -\frac14\mathcal F\textup{D}_2^{\pm}(\mathcal{W}^{\pm},\mathcal{W}^{\pm})(t,k) + \frac14\mathcal F\textup{D}_2^{\pm}(\mathcal{W}^{\pm},\mathcal{W}^{\pm})(0,k) \\
&~{} +\int_0^t (\mathcal F\mathcal{N}_3^{\pm})_\mathcal{R}(\mathcal{W}^{\pm},\mathcal{W}^{\pm},\mathcal{W}^{\pm})(t',k) \; dt'\\
&~{}+ \frac12\mathcal{F}{\textup{D}_3^{\pm}}(\mathcal{W}^{\pm},\mathcal{W}^{\pm},\mathcal{W}^{\pm})(0,k)
  +\frac12\mathcal{F}{\textup{D}_3^{\pm}}(\mathcal{W}^{\pm},\mathcal{W}^{\pm},\mathcal{W}^{\pm})(t,k) \\ 
&~{}-\frac12 \int_0^t\mathcal{F}\textup{D}_4^{\pm}(\mathcal{W}^{\pm},\mathcal{W}^{\pm},\mathcal{W}^{\pm},\mathcal{W}^{\pm})(t',k) \; dt'.
\end{aligned}
\end{equation}

\begin{remark}
Comparing \eqref{eq:renormalized KdV} with \eqref{eq:renormalized FPU}, the resonant term
\[ \mp\frac{2i}{k} \int_0^t\left|\widehat{\mathcal{W}^{\pm}}(t',k)\right|^2\widehat{\mathcal{W}^{\pm}}(t',k) \; dt'\]
corresponds to 
\[\mp2i\frac{\cos\left(\frac{hk}{2}\right)\cos\left(\frac{hk}{4}\right)}{\frac{4}{h}\sin\left(\frac{hk}{4}\right)}\int_0^t \left|\widehat{\mathcal{V}_h^{\pm}}(t',k)\right|^2\widehat{\mathcal{V}_h^{\pm}}(t',k) \; dt',\]
while the rest
\[\pm \frac{ih}{2}\sin\left(\frac{hk}{2}\right) \int_0^t\widehat{\mathcal{V}_h^{\pm}}(t,k)\sum_{\substack{k_1 \in (\T_h)^* \\ k_1 \neq \pm k}} |\widehat{\mathcal{V}_h^{\pm}}(t,k_1)|^2 \; dt'\]
itself is negligible in $L^2$ due to an additional $h$.
\end{remark}

\section{Continuum limit of the Decoupled FPU system to the KdV system}\label{sec:part2}

\subsection{KdV equation}
In this section, we are going to introduce some interesting results concerned with KdV equation.

\begin{lemma}[$L^4$-Strichartz estimates \cite{B-1993KdV}]\label{lem:L4 KdV}
For $f \in X^{0,\frac13}_{\tau = k^3}$, we have
\[\|f\|_{L^4_{t,x}(\R \times \T)} \lesssim \|f\|_{X_{\tau = k^3}^{0,\frac13}}.\]
\end{lemma}

\begin{proof}
The proof is analogous in the proof of Proposition \ref{prop:L4}, and it can be found in \cite{B-1993KdV} as well as \cite{T-2001, T-2006, Oh-2012}.
\end{proof}

\begin{proposition}[$L^2$ well-posedness \cite{B-1993KdV}]\label{prop:KdV well-posed}
The KdV system 
\begin{equation}\label{eq:KdV system}
\begin{aligned}
 \pa_t w_\pm \pm \frac{1}{24} \pa_x^3 w_\pm \mp \frac{1}{4}\pa_x(w_\pm^2)&=0,\\
  w_\pm(0)&=w_{\pm,0}.
\end{aligned}
\end{equation}
is well-posed in $L^2(\T)$.
\end{proposition}

\begin{remark}
Low regularity results for KdV equation have been extensively studied. We, for instance, refer to \cite{KPV-1996, CKSTT-2003, KT-2006}.
\end{remark}

As an immediate corollary, we have
\begin{corollary}[Uniform bounds for KdV solutions]\label{cor:uniform bound for KdV}
Let $s\ge0$.  For given initial data $w_{\pm,0}  \in H^s(\T)$, let $w_{\pm}(t)$ be the solutions to the KdV system \eqref{eq:KdV system} given in Proposition~\ref{prop:KdV well-posed}. Then,
\[\left\| w_{\pm}(t) \right\|_{C([-T,T]:H_x^s(\T))} \ls \| w_{\pm,0}  \|_{H^s(\T)}, \qquad \| w_{\pm}\|_{L^4([0,T] \times \T)} \ls \|w_{\pm,0}\|_{H^s(\T)}.\]
\end{corollary}

\begin{proof}
It follows from Proposition \ref{prop:KdV well-posed} and Lemma \ref{lem:L4 KdV}.
\end{proof}

\subsection{Some preliminaries}

For $M \in \mathbb N$, let $P_{\le M}$ denote the projection operator on $L^2(\T)$ defined by
\[P_{\le M}f(x) = \frac{1}{\sqrt{2\pi}}\sum_{|k| \le M} \widehat{f}(k).\]
With this, we define $P_{\ge M}$ and $P_{L \le \cdot \le M}$ by
\[P_{\ge M} = I - P_{< M} \quad \mbox{and} \quad P_{L \le \cdot \le M} = P_{\le M} - P_{< L},\]
where $I$ is the identity operator on $L^2(\T)$.

We first address some preliminaries.
\begin{lemma}\label{lem:multilinear estimate}
Let $M > 0$. Then,
\begin{align}
  \|P_{\le M} \textup{D}_1^{\pm}(f_1,f_2) \|_{L^2(\T)} 
  &\ls M^\frac32 \|f_1\|_{L^2(\T)}\| f_2\|_{L^2(\T)}, \label{ineq:D0} \\ 
  \|P_{\ge M} \textup{D}_2^{\pm}(f_1,f_2)  \|_{L^2(\T)}
  &\ls \frac{1}{M}\|f_1\|_{L^2(\T)}\| f_2\|_{L^2(\T)} \label{ineq:D2} ,\\ 
  \|P_{\ge M} \textup{D}_3^{\pm}(f_1,f_2,f_3)  \|_{L^2(\T)}
  &\ls \frac{1}{M^{\frac12}}\| f_1\|_{L^2(\T)}\| f_2\|_{L^2(\T)}\| f_3\|_{L^2(\T)} \label{ineq:D3}  \\ 
  \|P_{\ge M} \textup{D}_4^{\pm}(f_1,f_2,f_3,f_4)  \|_{L^2(\T)}
  &\ls \| f_1\|_{L^2(\T)}\| f_2\|_{L^2(\T)}\|f_3\|_{L^2(\T)}\|_{L^2(\T)}\| f_4\|_{L^2(\T)}, \label{ineq:D4} 
\end{align}
for $f_j \in L^2(\T)$, $1 \le j \le 4$. 
\end{lemma}
\begin{proof}
\eqref{ineq:D0} just follows from the duality and the Cauchy-Schwarz inequality.

For \eqref{ineq:D2}, we know
\[ |\mathcal F \textup{D}_2^{\pm}(f_1,f_2)(k)|
  \ls \sum_{\substack{k=k_1+k_2 \\  k_1k_2\neq0}} \frac{1}{|k_1k_2|}\big|\hat{f}_1(k_1)\big|
  \big|\hat{f}_2(k_2)\big|.\]
Since $k=k_1+k_2$ with $k_1k_2 \neq 0$, we have $M \le |k|\lesssim \max(|k_1|,|k_2|)$. Without loss of generality, we assume that $|k_1| \ge |k_2|$, Using $|k| \le |k_1(k-k_1)|$, the Cauchy-Schwarz inequality, and the summation of $|k|^{-2}$ over $|k| \ge M$, one has
\[\begin{aligned}
  \|P_{\ge M} \textup{D}_2^{\pm}(f_1,f_2)  \|_{L^2(\T)} \lesssim&~{} \left(\sum_{|k| \ge M} \left|\sum_{k_1 \neq 0}\frac{1}{|k_1(k-k_1)|}|\hat{f}_1(k_1)\hat{f}_2(k-k_1)| \right|^2\right)^{\frac12}\\
\lesssim&~{} \frac{1}{M}\|f_1\|_{L^2(\T)}\|f_1\|_{L^2(\T)}.
\end{aligned}\]

For \eqref{ineq:D3}, let $\widehat{F}_1(k_1) = |k_1|^{-1}|\widehat{f}_1(k_1)|$ and $\widehat{F}_j(k_j) = |k_j|^{-\frac14}|\widehat{f}_j(k_j)|$, $j=2,3$. From
   \[ |(k_1+k_2)(k_1+k_3)(k_2+k_3)| \gs \max(|k_1|,|k_2|,|k_3|, |k|) \gs |k_2|^{\frac14} |k_3|^{\frac14} |k|^{\frac12},\]
on $\Lambda(k)$, we have 
\[\begin{aligned}
  \|P_{\ge M} \mathcal{F} \textup{D}_3^{\pm}(f_1,f_2,f_3) \|_{L^2(\T)} =&~{} \left(\sum_{|k| \ge M}\left|\sum_{(k_1,k_2,k_3)\in\Lambda(k)}
  \frac{2i}
  {ik_1\Psi^3(k,k_1,k_2,k_3)} e^{ \pm it\Psi^3(k,k_1,k_2,k_3)}\prod_{j=1}^3\widehat{f}_j(k_j)\right|^2\right)^{\frac12}\\
\ls&~{} \frac{1}{M^{\frac12}}\left\|\sum_{(k_1,k_2,k_3)\in\Lambda(k)} \widehat{F}_1(k_1)\widehat{F}_2(k_2)\widehat{F}_3(k_3)\right\|_{\ell_k^2}\\ 
  \lesssim&~{} \frac{1}{M^{\frac12}}\|F_1F_2F_3\|_{L^2(\T)} \\ 
  \lesssim&~{} \frac{1}{M^{\frac12}}\|F_1\|_{L^\infty(\T)}\|F_2\|_{L^4(\T)}\|F_3\|_{L^4(\T)}\\
  \lesssim&~{} \frac{1}{M^{\frac12}}\| f_1\|_{L^2(\T)}\| f_2\|_{L^2(\T)}\| f_3\|_{L^2(\T)}.
\end{aligned}\]

For \eqref{ineq:D4}, observe that 
\[\begin{aligned}
&~{}\frac{2(k_1+k_2)+k_3}{k_3(k_1+k_2+k_3)(k_1+k_2+k_4)(k_3+k_4)}\\
=&~{}\frac{2}{k_3(k_1+k_2+k_4)(k_3+k_4)} - \frac{1}{(k_1+k_2+k_3)(k_1+k_2+k_4)(k_3+k_4)},
\end{aligned}\]
which yields
\[\begin{aligned}
&~{}\left|\mathcal{F}\textup{D}_4^{\pm}(f_1,f_2,f_3,f_4)(k)\right|\\
\lesssim&~{} \sum_{(k_1,k_2,k_3,k_3)\in \Gamma(k)}\frac{1}{k_3(k_1+k_2+k_4)(k_3+k_4)} \prod_{j=1}^4\widehat{f}_j(k_j)\\
+&~{} \sum_{(k_1,k_2,k_3,k_3)\in \Gamma(k)}\frac{1}{(k_1+k_2+k_3)(k_1+k_2+k_4)(k_3+k_4)}\prod_{j=1}^4\widehat{f}_j(k_j).
\end{aligned}\]
By duality, it suffices to show
\begin{equation}\label{eq:quartiliniear-1}
\sum_{\Pi_1(k)}\frac{\widehat{f}_1(k_1)\widehat{f}_2(k_2)\widehat{f}_3(k_3)\widehat{f}_4(k-k_1-k_2-k_3)\widehat{g}(-k)}{k_3(k-k_3)(k-k_1-k_2)} \lesssim \prod_{j=1}^4\|f_j\|_{L^2(\T)}\|g\|_{L^2(\T)}
\end{equation}
and
\begin{equation}\label{eq:quartiliniear-2}
\sum_{\Pi_2(k)}\frac{\widehat{f}_1(k_1)\widehat{f}_2(k_2)\widehat{f}_3(k_3)\widehat{f}_4(k-k_1-k_2-k_3)\widehat{g}(-k)}{(k_1+k_2+k_3)(k-k_3)(k-k_1+k_2)} \lesssim \prod_{j=1}^4\|f_j\|_{L^2(\T)}\|g\|_{L^2(\T)}
\end{equation}
where $\Pi_1(k)$ and $\Pi_2(k)$ are the sets of frequencies given by
\[\Pi_1(k) = \{(k_1,k_2,k_3,k_4) \in (\Z)^4 : k_1+k_2+k_3+k_4 = k, k_3(k_1+k_2+k_4)(k_3+k_4)\neq 0\}\] and 
\[\Pi_2(k) = \{(k_1,k_2,k_3,k_4) \in (\Z)^4 : k_1+k_2+k_3+k_4 = k, (k_1+k_2+k_3)(k_1+k_2+k_4)(k_3+k_4) \neq 0\},\]
respectively. By Cauchy-Schwarz inequality, we have
\[\begin{aligned}
\mbox{LHS of } \eqref{eq:quartiliniear-1} \lesssim&~{} \left(\sum_{\Pi_1(k)}\frac{|\widehat{f}_1(k_1)\widehat{f}_3(k_3)|^2}{|(k-k_3)(k-k_1-k_2)|^2}\right)^{\frac12}\left(\sum_{\Pi_1(k)}\frac{|\widehat{f}_2(k_2)\widehat{f}_4(k-k_1-k_2-k_3)\widehat{g}(-k)|^2}{|k_3|^2}\right)^{\frac12}\\
\lesssim&~{}\|f_1\|_{L^2(\T)}\|f_2\|_{L^2(\T)}\|f_3\|_{L^2(\T)}\|f_4\|_{L^2(\T)}\|g\|_{L^2(\T)}
\end{aligned}\]
and
\[\begin{aligned}
\mbox{LHS of } \eqref{eq:quartiliniear-2} \lesssim&~{} \left(\sum_{\Pi_2(k)}\frac{|\widehat{f}_1(k_1)\widehat{f}_3(k_3)|^2}{|(k_1+k_2+k_3)(k-k_3)|^2}\right)^{\frac12}\left(\sum_{\Pi_2(k)}\frac{|\widehat{f}_2(k_2)\widehat{f}_4(k-k_1-k_2-k_3)\widehat{g}(-k)|^2}{|k-k_1-k_2|^2}\right)^{\frac12}\\
\lesssim&~{}\|f_1\|_{L^2(\T)}\|f_2\|_{L^2(\T)}\|f_3\|_{L^2(\T)}\|f_4\|_{L^2(\T)}\|g\|_{L^2(\T)}.
\end{aligned}\]
\end{proof}
\begin{lemma}\label{Lem:M}
Let $0<h\le1$ and $0\le \alpha \le 1$. Let
\[\mathcal M_{h,0}^{2,\pm}(t,k,k_1,k_2) = ke^{\pm it\Psi^2(k,k_1,k_2)}- \frac{8}{h^3k^2}\sin^3\left(\frac{hk}{2}\right)e^{\pm it \Phi_h^2(k,k_1,k_2)}.\]
Then, for $|k|,|k_1|,|k_2|\le \frac{\pi}{h}$, we have
\[  |\mathcal M_{h,0}^{2,\pm}(t,k,k_1,k_2)|\ls |k|^{1+\alpha}h^\alpha(1+|t|)^\alpha \left( |k_1k_2|\max(|k|,|k_1|,|k_2|)\right)^\alpha.\]
  \end{lemma}
  \begin{proof}
By the mean-value theorem, the Taylor remainder theorem, and \eqref{Lh} (for $\alpha=1$), we have
\[\begin{aligned}
\left| e^{\pm it\Phi_h^2(k,k_1,k_2)} -e^{\pm it\Psi^2(k,k_1,k_2)}\right| \lesssim&~{}  |t| \left| \Phi_h^2(k,k_1,k_2) - \Psi^2(k,k_1,k_2) \right| \\ 
\lesssim&~{} |t|\left| \frac{4}{h}\sin\left(\frac{hk}{4}\right) -k\right| \left| \frac{2}{h^2}\sin\left(\frac{hk_1}{4}\right) \sin\left(\frac{hk_2}{4}\right)\right| \\
&~{}+ |t|\left|\frac{4}{h}\sin\left(\frac{hk_1}{4}\right) -k_1\right| \left| \frac{k}{2h}\sin\left( \frac{hk_2}{4}\right)\right|\\
&~{}+|t|\left| \frac{4}{h}\sin\left(\frac{hk_2}{4}\right) -k_2\right| \left| \frac{kk_1}{8}\right| \\ 
\lesssim&~{} |t|h(|k|+|k_1|+|k_2|) |kk_1k_2|.
\end{aligned}\] 
Interpolating with the following trivial bound 
\[    \left| e^{\pm it \Phi_h^2(k,k_1,k_2)} -e^{\pm it \Psi^2(k,k_1,k_2)}\right| \lesssim 1,\]
one has
  \begin{equation}\label{exponential differnece}
\left| e^{\pm it\Phi_h^2(k,k_1,k_2)} -e^{\pm it \Psi^2(k,k_1,k_2)}\right| \lesssim \left( h|t||kk_1k_2|\max(|k|,|k_1|,|k_2|)\right)^\alpha,
  \end{equation}
for $0 \le \alpha \le 1$. For any $0 \le \alpha \le 1$, collecting \eqref{Lh}, \eqref{Lh-1} and \eqref{exponential differnece}, we conclude that
  \[\begin{aligned}
    |\mathcal M_{h,0}^{2,\pm}(t,k,k_1,k_2)| \ls&~{} \left|\frac{2}{h}\sin\left(\frac{hk}{2}\right) -k\right|\left|\frac{4}{h^2k^2}\sin^2\left(\frac{hk}{2}\right)e^{\pm it \Phi_h^2(k,k_1,k_2)}\right|\\
&~{}+ |k|\left|\frac{4}{h^2k^2}\sin^2\left(\frac{hk}{2}\right)-1\right|\left|e^{\pm it \Phi_h^2(k,k_1,k_2)}\right|\\
&~{}+ |k|\left|e^{\pm it\Phi_h^2(k,k_1,k_2)} -e^{\pm it \Psi^2(k,k_1,k_2)}\right|\\
\lesssim&~{} |k||hk|^\alpha + |k|\left( h|t||kk_1k_2|\max(|k|,|k_1|,|k_2|)\right)^\alpha,
  \end{aligned}\]
which completes the proof.
  \end{proof}

\begin{lemma}\label{Lem:M1}
Let $0<h \le 1$ and $0\le \alpha \le1$. Let
\[\mathcal{M}_{h,1}^{2,\pm}(t,k,k_1,k_2) := \frac{k}{\Psi^2(k,k_1,k_2)}e^{\pm it \Psi^2(k,k_1,k_2)}-\frac{8\sin^3\left(\frac{hk}{2}\right)}{h^3k^2\Phi_h^2(k,k_1,k_2)}e^{\pm it \Phi_h^2(k,k_1,k_2)}\]
Then, for $|k|,|k_1|,|k_2|\le \frac{\pi}{h}$ with $k_1k_2 \neq 0$, we have
\[  | \mathcal{M}_{h,1}^{\pm}(t,k,k_1,k_2) |\ls   \frac{h^\alpha(1+|t|^\alpha)}{|k_1k_2|}\left(|kk_1k_2|\max(|k|,|k_1|,|k_2|) \right)^\alpha.\]
\end{lemma}
\begin{proof}
By \eqref{Lh-1}, we know
\[\begin{aligned}
\left|\left(\frac{4}{h^2k^2}\sin^2\left(\frac{hk}{2}\right)-1\right)\frac{\frac{2}{h}\sin\left(\frac{hk}{2}\right)}{\Phi_h^2(k,k_1,k_2)}\right| \lesssim&~{} (h|k|)^{\alpha}\left|\frac{\frac{4}{h}\sin\left(\frac{hk}{4}\right)\cos\left(\frac{hk}{4}\right)}{\Phi_h^2(k,k_1,k_2)}\right|\\
\lesssim&~{} (h|k|)^{\alpha}\left|\frac{h^2}{\sin\left(\frac{hk_1}{4}\right)\sin\left(\frac{hk_2}{4}\right)}\right|\\
\lesssim&~{} \frac{(h|k|)^{\alpha}}{|k_1k_2|},
\end{aligned}\]
for any $0 \le \alpha \le 2$. Thus, it suffices to deal with
\begin{equation}\label{eq:M1-1}
\left|\frac{k}{\Psi^2(k,k_1,k_2)}e^{\pm it \Psi^2(k,k_1,k_2)}-\frac{2\sin\left(\frac{hk}{2}\right)}{h\Phi_h^2(k,k_1,k_2)}e^{\pm it \Phi_h^2(k,k_1,k_2)}\right|.
\end{equation}
Let $\displaystyle \mathcal K_j = \frac{4}{h}\sin\left(\frac{hk_j}{4}\right)$, $j=1,2$. Then, it is known that $|\mathcal K_j| \sim |k_j|$, $j=1,2$. Note that similarly as \eqref{Lh}, we have
\begin{equation}\label{eq:Lh-cosine}
\left|\cos\left(\frac{hk}{4}\right)-1\right| \lesssim (h|k|)^{\alpha},
\end{equation}
for any $0 \le \alpha \le 2$. By \eqref{Lh}, we obtain 
\[\begin{aligned}
\left|\frac{\cos\left(\frac{hk}{4}\right)}{\mathcal K_1 \mathcal K_2} - \frac{1}{k_1k_2}\right| =&~{} \left|\frac{\cos\left(\frac{hk}{4}\right)(k_1-\mathcal K_1)k_2}{\mathcal K_1 \mathcal K_2 k_1 k_2}+\frac{\left(\cos\left(\frac{hk}{4}\right)-1\right)\mathcal K_1k_2}{\mathcal K_1 \mathcal K_2 k_1 k_2}+\frac{\mathcal K_1\left(k_2 - \mathcal K_2\right)}{\mathcal K_1 \mathcal K_2 k_1 k_2}\right|\\
\lesssim&~{} \left|\frac{(h|k_1|)^{\alpha}}{\mathcal K_1 \mathcal K_2}\right| + \left|\frac{(h|k|)^{\alpha}}{\mathcal K_2k_1}\right| + \left|\frac{(h|k_2|)^{\alpha}}{\mathcal K_2k_1}\right|\\
\sim&~{} \frac{h^{\alpha}}{|k_1k_2|}\left(|k|^{\alpha} + |k_1|^{\alpha} + |k_2|^{\alpha}\right),
\end{aligned}\]
for any $0 \le \alpha \le 2$. Thus, by \eqref{exponential differnece}, we obtain for $0 \le \alpha \le 1$ that
\[\begin{aligned}
\eqref{eq:M1-1} \le&~{} \left|\frac{k}{\Psi^2(k,k_1,k_2)}\left(e^{\pm it \Psi^2(k,k_1,k_2)} - e^{\pm it \Phi_h^2(k,k_1,k_2)}\right)\right| \\
&~{}+ \left|\left(\frac{2\sin\left(\frac{hk}{2}\right)}{h\Phi_h^2(k,k_1,k_2)} - \frac{k}{\Psi^2(k,k_1,k_2)}\right)e^{\pm it \Phi_h^2(k,k_1,k_2)}\right|\\
\lesssim&~{} \frac{1}{|k_1k_2|}\left( h|t||kk_1k_2|\max(|k|,|k_1|,|k_2|)\right)^\alpha + \frac{h^s}{|k_1k_2|}\left(|k|^\alpha + |k_1|^\alpha + |k_2|^\alpha\right),
\end{aligned}\]
which completes the proof.
\end{proof}

\begin{lemma}\label{Lem:M2}
Let $0<h \le 1$ and $0\le \alpha \le1$. Let
\[\begin{aligned}
\mathcal{M}_{h}^{3,\pm}(t,k,k_1,k_2,k_3) :=&~{} \frac{2}{k_1\Psi^3(k,k_1,k_2,k_3)}e^{\pm it \Psi^3(k,k_1,k_2,k_3)}\\
&~{}-\frac{2\sin^2\left(\frac{hk}{2}\right)\cos\left(\frac{hk}{4}\right)\cos\left(\frac{h(k_2+k_3)}{4}\right)}{hk^2\sin\left(\frac{hk_1}{4}\right)\Phi_h^3(k,k_1,k_2,k_3)}e^{\pm it \Phi_h^3(k,k_1,k_2,k_3)}
\end{aligned}\]
over $\mathcal A(k)$ defined as in \eqref{eq:A(k)}. Then, we have
\[  | \mathcal{M}_{h}^{3,\pm}(t,k,k_1,k_2,k_3) |\ls \frac{h^\alpha(1+|t|)^\alpha(\max(|k|,|k_1|,|k_2|,|k_3|))^{2\alpha}}{|k_1|\max(|k_1|,|k_2|,|k_3|)}.\]
\end{lemma}
\begin{proof}
Let
\[\mathcal{M}_{h,1}^{3,\pm}(t,k,k_1,k_2,k_3) := \frac{1}{k_1\Psi^3(k,k_1,k_2,k_3)}\left(e^{\pm it \Psi^3(k,k_1,k_2,k_3)}-e^{\pm it \Phi_h^3(k,k_1,k_2,k_3)}\right)\]
and
\[\mathcal{M}_{h,2}^{3,\pm}(t,k,k_1,k_2,k_3) := \frac{1}{k_1\Psi^3(k,k_1,k_2,k_3)} - \frac{\sin^2\left(\frac{hk}{2}\right)\cos\left(\frac{hk}{4}\right)\cos\left(\frac{h(k_2+k_3)}{4}\right)}{hk^2\sin\left(\frac{hk_1}{4}\right)\Phi_h^3(k,k_1,k_2,k_3)}.\]
Then, $\displaystyle \mathcal{M}_{h}^{3,\pm}(t,k,k_1,k_2,k_3) = 2\left(\mathcal{M}_{h,1}^{3,\pm}(t,k,k_1,k_2,k_3) + \mathcal{M}_{h,2}^{3,\pm}(t,k,k_1,k_2,k_3)e^{\pm it \Phi_h^3(k,k_1,k_2,k_3)}\right)$, and thus it suffices to control $|\mathcal{M}_{h,1}^{3,\pm}(t,k,k_1,k_2,k_3)|$ and $|\mathcal{M}_{h,2}^{3,\pm}(t,k,k_1,k_2,k_3)|$. Let $\displaystyle \mathcal K_{j\ell} = \frac{4}{h}\sin\left(\frac{h(k_j+k_\ell)}{4}\right)$ and $k_{j\ell} = k_j+k_\ell$, $j,\ell=1,2,3$. Then, it is known that $|\mathcal K_{j\ell}| \sim |k_{j\ell}|$, $j,\ell=1,2,3$. Similarly as the proof of \eqref{exponential differnece}, we have for $0 \le \alpha \le 1$ that
\begin{equation}\label{eq:M^3-1}
\left|e^{\pm it \Psi^3(k,k_1,k_2,k_3)}-e^{\pm it \Phi_h^3(k,k_1,k_2,k_3)}\right|\lesssim h^\alpha(1+|t|)^\alpha \left(|k_{12}||k_{13}||k_{23}| \max(|k_{12}|,|k_{13}|,|k_{23}|)\right)^{\alpha}.
\end{equation}
Note that $\max(|k_{12}|,|k_{13}|,|k_{23}|) \lesssim \max(|k_1|,|k_2|,|k_3|)$ and $|k_{12}k_{13}k_{23}| \gtrsim \max(|k_1|,|k_2|,|k_3|)$. Then, \eqref{eq:M^3-1} guarantees
\[|\mathcal{M}_{h,1}^{3,\pm}(t,k,k_1,k_2,k_3)| \lesssim \frac{h^\alpha(1+|t|)^\alpha(\max(|k_1|,|k_2|,|k_3|))^\alpha}{|k_1|(|k_{12}k_{13}k_{23}|)^{1-\alpha}} \lesssim \frac{h^\alpha(1+|t|)^\alpha(\max(|k_1|,|k_2|,|k_3|))^{2\alpha}}{|k_1|\max(|k_1|,|k_2|,|k_3|)}.\]
On the other hand, note by \eqref{eq:Lh-cosine} that
\begin{equation}\label{eq:M^3-2}
\left|1-\cos\left(\frac{hk}{4}\right)\cos\left(\frac{hk_{23}}{4}\right)\right| \lesssim (h|k| + h|k_{23}|)^{\alpha} \lesssim h^{\alpha}(\max(|k|,|k_2|,|k_3|))^{\alpha},
\end{equation}
for any $0 \le \alpha \le 2$. Moreover, by \eqref{Lh}, we have
\[\begin{aligned}
\left|\frac{1}{k_1\Psi^3(k,k_1,k_2,k_3)}-\frac{1}{\mathcal K_1\Phi_h^3(k,k_1,k_2,k_3)}\right|\lesssim&~{} \left|\frac{\mathcal K_1\mathcal K_{12}\mathcal K_{13}\mathcal K_{23}-k_1k_{12}k_{13}k_{23}}{k_1k_{12}k_{13}k_{23}\mathcal K_1\mathcal K_{12}\mathcal K_{13}\mathcal K_{23}}\right|\\
\lesssim&~{} \left|\frac{(\mathcal K_1-k_1)\mathcal K_{12}\mathcal K_{13}\mathcal K_{23}}{k_1k_{12}k_{13}k_{23}\mathcal K_1\mathcal K_{12}\mathcal K_{13}\mathcal K_{23}}\right| \\
&~{}+ \left|\frac{k_1(\mathcal K_{12}-k_{12})\mathcal K_{13}\mathcal K_{23}}{k_1k_{12}k_{13}k_{23}\mathcal K_1\mathcal K_{12}\mathcal K_{13}\mathcal K_{23}}\right| \\
&~{}+  \left|\frac{k_1k_{12}(\mathcal K_{13}-k_{13})\mathcal K_{23}}{k_1k_{12}k_{13}k_{23}\mathcal K_1\mathcal K_{12}\mathcal K_{13}\mathcal K_{23}}\right| \\
&~{}+ \left|\frac{k_1k_{12}\mathcal K_{13}(\mathcal K_{23}-k_{23})}{k_1k_{12}k_{13}k_{23}\mathcal K_1\mathcal K_{12}\mathcal K_{13}\mathcal K_{23}}\right|\\
\lesssim&~{}  \frac{h^{\alpha}(|k_1|^{\alpha}+|k_{12}|^{\alpha}+|k_{13}|^{\alpha}+|k_{23}|^{\alpha})}{|k_1||k_{12}||k_{13}||k_{23}|}\\
\lesssim&~{} \frac{h^{\alpha}(\max(|k_1|,|k_2|,|k_3|))^{\alpha}}{|k_1|\max(|k_1|,|k_2|,|k_3|)},
\end{aligned}\]
for any $0 \le \alpha \le 2$, and where $\mathcal K_1$ is defined in the proof of Lemma \ref{Lem:M1}. Together with \eqref{Lh-1} and \eqref{eq:M^3-2}, we obtain
\[\begin{aligned}
\mathcal{M}_{h,2}^{3,\pm}(t,k,k_1,k_2,k_3)\lesssim&~{}\left|\left(1-\frac{4\sin^2\left(\frac{hk}{2}\right)}{h^2k^2}\right)\frac{\cos\left(\frac{hk}{4}\right)\cos\left(\frac{hk_{23}}{4}\right)}{\mathcal K_1\Phi_h^3(k,k_1,k_2,k_3)}\right|\\
&~{}+\left|\left(1-\cos\left(\frac{hk}{4}\right)\cos\left(\frac{hk_{23}}{4}\right)\right)\frac{1}{\mathcal K_1\Phi_h^3(k,k_1,k_2,k_3)}\right|\\
&~{}+\left|\frac{1}{k_1\Psi^3(k,k_1,k_2,k_3)}-\frac{1}{\mathcal K_1\Phi_h^3(k,k_1,k_2,k_3)}\right|\\
\lesssim&~{} \frac{h^\alpha(\max(|k|,|k_1|,|k_2|,|k_3|))^\alpha}{|k_1|\max(|k_1|,|k_2|,|k_3|)},
\end{aligned}\]
for any $0 \le \alpha\le 1$. Since $(\max(|k|,|k_1|,|k_2|,|k_3|))^{\alpha} \lesssim (\max(|k|,|k_1|,|k_2|,|k_3|))^{2\alpha}$ for $\alpha \ge 0$, we complete the proof.
\end{proof}

\begin{lemma}\label{Lem:Mh4}
Let $0<h \le 1$ and $0\le \alpha \le1$. Let
\[\begin{aligned}
\mathcal{M}_{h,1}^{4,\pm}(t,k,k_1,k_2,k_3,k_4) :=&~{} \frac{2(k_1+k_2)}{k_3\Psi^3(k,k_1+k_2,k_3,k_4)}e^{\pm it \Psi^4(k,k_1,k_2,k_3,k_4)}\\
&~{}-\frac{4\sin^2\left(\frac{hk}{2}\right)}{h^2k^2}\frac{\sin\left(\frac{h(k_1+k_2)}{2}\right)\cos\left(\frac{hk}{4}\right)\cos\left(\frac{h(k_1+k_2+k_4)}{4}\right)}{\sin\left(\frac{hk_3}{4}\right)\Phi_h^3(k,k_1+k_2,k_3,k_4)}e^{\pm it \Phi_h^4(k,k_1,k_2,k_3,k_4)}
\end{aligned}\]
and 
\[\begin{aligned}
\mathcal{M}_{h,2}^{4,\pm}(t,k,k_1,k_2,k_3,k_4) :=&~{} \frac{1}{\Psi^3(k,k_1+k_2,k_3,k_4)}e^{\pm it \Psi^4(k,k_1,k_2,k_3,k_4)}\\
&~{}-\frac{4\sin^2\left(\frac{hk}{2}\right)}{h^2k^2}\frac{\cos\left(\frac{hk}{4}\right)\cos\left(\frac{h(k_1+k_2)}{4}\right)\cos\left(\frac{h(k_3+k_4)}{4}\right)}{\Phi_h^3(k,k_1+k_2,k_3,k_4)}e^{\pm it \Phi_h^4(k,k_1,k_2,k_3,k_4)}
\end{aligned}\]
over $\mathcal B(k)$ defined as in \eqref{eq:B(k)}. Let
\[\mathcal{M}_{h}^{4,\pm}(t,k,k_1,k_2,k_3,k_4) = \mathcal{M}_{h,1}^{4,\pm}(t,k,k_1,k_2,k_3,k_4) + \mathcal{M}_{h,2}^{4,\pm}(t,k,k_1,k_2,k_3,k_4).\]
Then, we have
\[\begin{aligned}
&~{}| \mathcal{M}_{h}^{4,\pm}(t,k,k_1,k_2,k_3,k_4)|\\
\ls&~{} \frac{h^\alpha(1+|t|)^\alpha|k_{12}|}{|k_3k_{123}k_{124}k_{34}|}\left(\max(|k|, |k_1|, |k_2|, |k_3|, |k_{12}|, |k_{34}|, |k_{123}|, |k_{124}|)\right)^\alpha \left(|k_{123}k_{124}k_{34}|^\alpha + |k_{12}k_1k_2|^\alpha\right).
\end{aligned}\]
\end{lemma}

  \begin{proof}
Let $\mathcal K_j$, $\mathcal K_{j\ell}$, and $k_{j\ell}$ be defined as in the proofs of Lemma \ref{Lem:M1} and \ref{Lem:M2} for $j,\ell = 1,2,3,4$. Additionally, let $\displaystyle \mathcal K_{j\ell m} = \frac{4}{h}\sin\left(\frac{h(k_j+k_\ell+k_m)}{4}\right)$ and $k_{j\ell m} = k_j + k_\ell + k_m$, $j,\ell,m = 1,2,3,4$. Similarly as the proof of \eqref{exponential differnece}, we have 
\begin{equation}\label{eq:M^4-1}
\begin{aligned}
&~{}\left|e^{\pm it \Psi^4(k,k_1,k_2,k_3,k_4)}-e^{\pm it \Phi_h^4(k,k_1,k_2,k_3,k_4)}\right| \\
\le&~{} \left|e^{\pm it \Psi^3(k,k_{12},k_3,k_4)}-e^{\pm it \Phi_h^3(k,k_{12},k_3,k_4)}\right| +\left|e^{\pm it \Psi^2(k_{12},k_1,k_2)}-e^{\pm it \Phi_h^2(k_{12},k_1,k_2)}\right|\\
\lesssim&~{} h^\alpha(1+|t|)^\alpha\left((|k_{123}k_{124}k_{34}|\max(|k_{123}|,|k_{124}|,|k_{34}|))^\alpha + (|k_{12}k_{1}k_{2}|\max(|k_{12}|,|k_{1}|,|k_{2}|))^\alpha\right),
\end{aligned}
\end{equation} 
for $0 \le \alpha \le 1$. Note by \eqref{eq:Lh-cosine} that
\begin{equation}\label{eq:M^4-0}
\begin{aligned}
\left|1-\cos\left(\frac{hk_{12}}{4}\right)\cos\left(\frac{hk}{4}\right)\cos\left(\frac{hk_{124}}{4}\right)\right|\lesssim&~{} (h|k_{12}|+h|k|+h|k_{124}|)^{\alpha}\\
\lesssim&~{}h^{\alpha}(\max(|k_{12}|,|k|,|k_{124}|))^{\alpha}, 
\end{aligned}
\end{equation}
for any $0 \le \alpha \le 2$. Moreover, by \eqref{Lh}, we have
\[\begin{aligned}
\left|\frac{k_{12}}{k_3k_{123}k_{124}k_{34}} - \frac{\mathcal K_{12}}{\mathcal K_{3}\mathcal K_{123}\mathcal K_{124}\mathcal K_{34}}\right| =&~{} \left|\frac{k_{12}\mathcal K_{3}\mathcal K_{123}\mathcal K_{124}\mathcal K_{34} - \mathcal K_{12}k_3k_{123}k_{124}k_{34}}{k_3k_{123}k_{124}k_{34}\mathcal K_{3}\mathcal K_{123}\mathcal K_{124}\mathcal K_{34}}\right|\\
\le&~{} \left|\frac{(k_{12}-\mathcal K_{12})\mathcal K_{3}\mathcal K_{123}\mathcal K_{124}\mathcal K_{34}}{k_3k_{123}k_{124}k_{34}\mathcal K_{3}\mathcal K_{123}\mathcal K_{124}\mathcal K_{34}}\right|\\
&~{}+\left|\frac{\mathcal K_{12}(\mathcal K_{3}-k_3)\mathcal K_{123}\mathcal K_{124}\mathcal K_{34}}{k_3k_{123}k_{124}k_{34}\mathcal K_{3}\mathcal K_{123}\mathcal K_{124}\mathcal K_{34}}\right|\\
&~{}+\left|\frac{\mathcal K_{12}k_3(\mathcal K_{123} - k_{123})\mathcal K_{124}\mathcal K_{34}}{k_3k_{123}k_{124}k_{34}\mathcal K_{3}\mathcal K_{123}\mathcal K_{124}\mathcal K_{34}}\right|\\
&~{}+\left|\frac{\mathcal K_{12}k_3k_{123}(\mathcal K_{124}-k_{124})\mathcal K_{34}}{k_3k_{123}k_{124}k_{34}\mathcal K_{3}\mathcal K_{123}\mathcal K_{124}\mathcal K_{34}}\right|\\
&~{}+\left|\frac{\mathcal K_{12}k_3k_{123}k_{124}(\mathcal K_{34}-k_{34})}{k_3k_{123}k_{124}k_{34}\mathcal K_{3}\mathcal K_{123}\mathcal K_{124}\mathcal K_{34}}\right|\\
\lesssim&~{}\frac{|k_{12}|h^\alpha(|k_{12}|^{\alpha}+|k_{3}|^{\alpha}+|k_{123}|^{\alpha}+|k_{124}|^{\alpha}+|k_{34}|^{\alpha})}{|k_3||k_{123}||k_{124}||k_{34}|}\\
\end{aligned}\]
for any $0 \le \alpha \le 2$. Together with \eqref{eq:M^4-1}, \eqref{Lh-1} and \eqref{eq:M^4-0}, we obtain
\[\begin{aligned}
&~{}\left|\mathcal{M}_{h,1}^{4,\pm}(t,k,k_1,k_2,k_3,k_4)\right| \\
\lesssim&~{} \left|\frac{k_{12}}{k_3\Psi^3(k,k_{12},k_3,k_4)}\left(e^{\pm it \Psi^4(k,k_1,k_2,k_3,k_4)}-e^{\pm it \Phi_h^4(k,k_1,k_2,k_3,k_4)}\right)\right|\\ &~{}+\left|\left(1-\frac{4\sin^2\left(\frac{hk}{2}\right)}{h^2k^2}\right)\frac{\sin\left(\frac{hk_{12}}{2}\right)\cos\left(\frac{hk}{4}\right)\cos\left(\frac{hk_{124}}{4}\right)}{\sin\left(\frac{hk_3}{4}\right)\Phi_h^3(k,k_{12},k_3,k_4)}\right|\\
&~{}+\left|\left(1-\cos\left(\frac{hk_{12}}{4}\right)\cos\left(\frac{hk}{4}\right)\cos\left(\frac{hk_{124}}{4}\right)\right)\frac{\sin\left(\frac{hk_{12}}{4}\right)}{\sin\left(\frac{hk_3}{4}\right)\Phi_h^3(k,k_{12},k_3,k_4)}\right|\\
&~{}+\left|\frac{k_{12}}{k_3k_{123}k_{124}k_{34}} - \frac{\mathcal K_{12}}{\mathcal K_{3}\mathcal K_{123}\mathcal K_{124}\mathcal K_{34}}\right|\\
\lesssim&~{} \frac{h^\alpha(1+|t|)^\alpha|k_{12}|}{|k_3k_{123}k_{124}k_{34}|}\left(\max(|k|, |k_1|, |k_2|, |k_3|, |k_{12}|, |k_{34}|, |k_{123}|, |k_{124}|)\right)^\alpha \left(|k_{123}k_{124}k_{34}|^\alpha + |k_{12}k_1k_2|^\alpha\right),
\end{aligned}\]
for any $0 \le \alpha \le 1$. Analogously, we also have
\[\begin{aligned}
&~{}\left|\mathcal{M}_{h,2}^{4,\pm}(t,k,k_1,k_2,k_3,k_4)\right| \\
\lesssim&~{} \left|\frac{1}{\Psi^3(k,k_{12},k_3,k_4)}\left(e^{\pm it \Psi^4(k,k_1,k_2,k_3,k_4)}-e^{\pm it \Phi_h^4(k,k_1,k_2,k_3,k_4)}\right)\right|\\ &~{}+\left|\left(1-\frac{4\sin^2\left(\frac{hk}{2}\right)}{h^2k^2}\right)\frac{\cos\left(\frac{hk}{4}\right)\cos\left(\frac{hk_{12}}{4}\right)\cos\left(\frac{hk_{34}}{4}\right)}{\Phi_h^3(k,k_{12},k_3,k_4)}\right|\\
&~{}+\left|\left(1-\cos\left(\frac{hk}{4}\right)\cos\left(\frac{hk_{12}}{4}\right)\cos\left(\frac{hk_{34}}{4}\right)\right)\frac{1}{\Phi_h^3(k,k_{12},k_3,k_4)}\right|\\
&~{}+\left|\frac{1}{k_{123}k_{124}k_{34}} - \frac{1}{\mathcal K_{123}\mathcal K_{124}\mathcal K_{34}}\right|\\
\lesssim&~{} \frac{h^\alpha(1+|t|)^\alpha|k_{12}|}{|k_3k_{123}k_{124}k_{34}|}\left(\max(|k|, |k_1|, |k_2|, |k_3|, |k_{12}|, |k_{34}|, |k_{123}|, |k_{124}|)\right)^\alpha \left(|k_{123}k_{124}k_{34}|^\alpha + |k_{12}k_1k_2|^\alpha\right),
\end{aligned}\]
for any $0 \le \alpha \le 1$. Collecting all, we complete the proof.
\end{proof}

\subsection{Continuum limit to KdV: Proof of Proposition~\ref{prop:from decoupled to kdv}}
We finally prove the convergence of decoupled FPU \eqref{decoupled FPU'} to KdV \eqref{kdv integral form}. Let $T=T(\| u_{h,0}^{\pm}\|_{H^s(\T_h)})>0$ be the common existence time for the solution $v_h^{\pm}(t)$ (resp. $w^{\pm}(t)$) to the decoupled FPU with initial data $u_{h,0}^{\pm}$ constructed in Proposition~\ref{prop:uniform bound} (resp. the KdV equation  with initial data $\mathcal{L}_hu_{h,0}^{\pm}$ constructed in Proposition~\ref{prop:KdV well-posed}).

\begin{proposition}\label{prop:exterior estimates}
Let $0 \le s \le 1$. Let $f_h$ and $g$ be any $H^s(\T_h)$ and $H^s(\T)$ functions, respectively. Then,
\[\| P_{\ge \frac{\pi}{h}}\mathcal L_h f_h\|_{L^2(\T)} \lesssim h^s\|f_h\|_{H^s(\T_h)} \quad \mbox{and} \quad \| P_{\ge \frac{\pi}{h}}g\|_{L^2(\T)} \lesssim h^s\|g\|_{H^s(\T)},\]
where $\mathcal L_h$ is the linear interpolation defined as in \eqref{def:linear interpolation_0}.
\end{proposition}

\begin{proof}
It immediately follows from Bernstein's inequality and the boundedness of the linear interpolation operator.
\end{proof}

\begin{proposition}[Comparison between linear FPU and Airy flows]\label{prop:commutator of interpolation and propagator}
Let $0 \le s \le 5$ and $|t| \le 1$ be fixed. Let $f_h$ be any $H^s(\T_h)$ function. Then,
\[\left\| P_{\le \frac{\pi}{h}}\left(S^{\pm}(t)\mathcal L_h f_h(t)  - \mathcal L_h S_h^{\pm}(t) f_h(t)\right)\right\|_{L^2(\T)} \lesssim h^{\frac{2s}{5}}\|f_h\|_{H^s(\T_h)}.\]
\end{proposition}

\begin{proof}
Note that for $0 \le s \le 5$, by the mean-value theorem and the Taylor remainder theorem, we have
\[\begin{aligned}
\left| e^{\pm it\frac{k^3}{24}} - e^{\pm it\frac{1}{h^2}\left( k - \frac{2}{h}\sin(\frac{hk}{2})\right) } \right| \lesssim&~{} \min\left(1,|t|\left|\frac{k^3}{24}-\frac{1}{h^2}\left(k-\frac{2}{h}\sin\left(\frac{hk}{2}\right)\right)\right|\right) \\
\lesssim&~{} \min\left(1,|t|h^2|k|^5\right) \\
\lesssim&~{} |t|^{\frac{s}{5}}h^{\frac{2s}{5}}|k|^s.
\end{aligned}\]
Thus, we conclude for $|t| \le 1$ that
\begin{equation}\label{difference2-1:symbol difference}
\begin{aligned}
  &~{}\left\| P_{\le \frac{\pi}{h}}\left(S^{\pm}(t)\mathcal L_h f_h(t)  - \mathcal L_h S_h^{\pm}(t) f_h(t)\right) \right\|_{L^2(\T)} \\ 
  =&~{}\left\| \left( e^{\pm it\frac{k^3}{24}} - e^{\pm it\frac{1}{h^2}\left(k- \frac{2}{h}\sin\left(\frac{hk}{2}\right)\right) } \right)\frac{4}{h^2k^2}\sin^2\left(\frac{hk}{2}\right)  \widehat{f}_h(t,k) \right \|_{L^2(|k|\le \frac{\pi}{h})} \\ 
  \lesssim&~{} h^{\frac{2s}{5}} \|f_h(t)\|_{H^s(\T_h)}.
\end{aligned}
\end{equation}
\end{proof}

\begin{proposition}\label{prop:differnece estimates-1}
Let $0 \le s \le 1$ and $0 < h \le 1$. For $1 < M < \frac{\pi}{h}$, we have 
\begin{equation}\label{lowfreq:esti}
\begin{aligned}
&~{}\|P_{\le M}(\mathcal L_{h} \mathcal V_{h}^{\pm} - \mathcal W^{\pm} )(t)\|_{L^2(\T)}\\
\lesssim&~{}  h^{\frac{s}{2}}M^{\frac{3}{2}+\frac{s}{2}} \int_0^t (1+|t'|)^{\frac{s}{2}}\|\mathcal{V}_h^{\pm}(t')\|_{H^{s}(\T_h)}^2 \; dt'\\
&~{}+ M^{\frac{3}{2}} \int_0^t \|(\mathcal L_{h} \mathcal V_{h}^{\pm} - \mathcal W^{\pm})(t')\|_{L^2(\T)} \left(\|\mathcal V_{h}^{\pm}(t')\|_{H^s(\T)}+\|\mathcal W^{\pm}(t')\|_{H^s(\T)}\right) \; dt',
\end{aligned}
\end{equation}
for any $\mathcal V_{h}^{\pm} \in H^s(\T_h)$ and $\mathcal W^{\pm} \in H^s(\T)$. The implicit constant does not depend on $h\in(0,1]$.
\end{proposition}

\begin{proof}
From \eqref{eq:FT of V} and \eqref{eq:FT of W}, we write for $|k|\le M$ that
\[\begin{aligned}
(\mathcal L_{h} \mathcal V_{h}^{\pm} - \mathcal W^{\pm} )(t,k)
&=\int_0^t \mathcal L_{h} B_1^{\pm}(\mathcal V_{h}^{\pm},\mathcal V_{h}^{\pm})(t',k) - D_1^{\pm}( \mathcal W^{\pm} , \mathcal W^{\pm} )(t',k) \; dt'\\ 
&=\int_0^t \left(\mathcal L_{h} B_1^{\pm}(\mathcal V_{h}^{\pm},\mathcal V_{h}^{\pm})-D_1^{\pm}( \mathcal{L}_{h}\mathcal V_{h}^{\pm} ,  \mathcal{L}_{h}\mathcal V_{h}^{\pm}  ) \right)(t',k) \; dt'\\
&~{}+ \int_0^t D_1^{\pm}(\mathcal L_{h} \mathcal V_{h}^{\pm} - \mathcal W^{\pm},\mathcal L_{h} \mathcal V_{h}^{\pm} + \mathcal W^{\pm} )(t',k) \; dt'.
\end{aligned}\]
By \eqref{ineq:D0} and the Minkowski inequality, we first have
\[\begin{aligned}
&~{}\|P_{\le M}D_1^{\pm}(\mathcal L_{h} \mathcal V_{h}^{\pm} - \mathcal W^{\pm},\mathcal L_{h} \mathcal V_{h}^{\pm} + \mathcal W^{\pm} )(t')\|_{L^2(\T)}\\
 \lesssim&~{} M^{\frac32}\|(\mathcal L_{h} \mathcal V_{h}^{\pm} - \mathcal W^{\pm})(t')\|_{L^2(\T)} \left(\|\mathcal L_{h} \mathcal V_{h}^{\pm}(t')\|_{L^2(\T)}+\|\mathcal W^{\pm}(t')\|_{L^2(\T)}\right).
\end{aligned}\]
On the other hand, by \eqref{linear interpolation multiplier} and the symmetry over $k_1$ and $k_2$, a direct computation for $|k| \le M$ gives
\[\mathcal{F} \left( \mathcal L_{h} B_1^{\pm}(\mathcal V_{h}^{\pm},\mathcal V_{h}^{\pm})(t',k) - D_1^{\pm}( \mathcal{L}_{h}\mathcal V_{h}^{\pm} ,  \mathcal{L}_{h}\mathcal V_{h}^{\pm}  ))\right)(t',k) = I_1(t',k) + I_2(t',k) + I_3(t',k),\]
where
\[I_1(t',k) =\mp i \sum_{\substack{k=k_1+k_2 \\ |k_1|, |k_2| \le \frac{\pi}{h} \\ k_1k_2 \neq 0}}\mathcal M_{h,0}^{2,\pm}(t',k,k_1,k_2)\widehat{\mathcal{L}_{h}\mathcal V_{h}^{\pm}}(t',k_1)\widehat{\mathcal{L}_{h}\mathcal V_{h}^{\pm}}(t',k_2),\]
here $\mathcal M_{h,0}^{2,\pm}(t',k,k_1,k_2)$ is introduced in Lemma \ref{Lem:M},
\[I_2(t',k) = \mp\frac{ik}{4} \sum_{\substack{k=k_1+k_2 \\ |k_1|, |k_2| \le \frac{\pi}{h} \\ k_1k_2 \neq 0}} e^{\pm it' \Psi^2(k,k_1,k_2)} \left(\widehat{\mathcal{L}_{h}\mathcal V_{h}^{\pm}} - \widehat{\mathcal{V}_h^{\pm}}  \right)(t',k_1) \left(\widehat{\mathcal{L}_{h}\mathcal V_{h}^{\pm}} + \widehat{\mathcal{V}_h^{\pm}}\right)(t',k_2),\]
and
\[I_3(t',k) = \mp\frac{ik}{4} \sum_{\substack{k=k_1+k_2 \\ |k_1| > \frac{\pi}{h}  \; \text{or} \; |k_2| > \frac{\pi}{h} \\ k_1k_2 \neq 0}} e^{\pm it' \Psi^2(k,k_1,k_2)} \widehat{\mathcal{L}_{h}\mathcal V_{h}^{\pm}}(t',k_1)\widehat{\mathcal{L}_{h}\mathcal V_{h}^{\pm}}(t',k_2).\]
By Lemma~\ref{Lem:M} and the Cauchy-Schwarz inequality, we have 
\[   \|P_{\le M}I_1(t')\|_{L^2(\T)} \lesssim h^\alpha (1+|t'|)^{\alpha} M^{\frac32+\alpha} \|\mathcal{V}_h^{\pm}(t')\|_{H^{2\alpha}(\T_h)}\|\mathcal{V}_h^{\pm}(t')\|_{H^{\alpha}(\T_h)},\]
for any $0 \le \alpha \le 1$. On the other hand, by the Cauchy-Schwarz inequality and Lemma \ref{lem:LOWLh}, we have
\begin{equation}\label{eq:I_2}
\begin{aligned}
\|P_{\le M}I_2(t')\|_{L^2(\T)} \lesssim&~{} h^\alpha M^{\frac32}\left(\sum_{|k| \le \frac{\pi}{h}}\left|\widehat{\mathcal{L}_{h}\mathcal V_{h}^{\pm}} - \widehat{\mathcal{V}_h^{\pm}} \right|^2\right)^{\frac12} \left(\sum_{|k| \le \frac{\pi}{h}}\left|\widehat{\mathcal{L}_{h}\mathcal V_{h}^{\pm}} + \widehat{\mathcal{V}_h^{\pm}} \right|^2\right)^{\frac12}\\
\lesssim&~{}h^\alpha M^{\frac32}\|\mathcal{V}_h^{\pm}(t')\|_{H^{\alpha}(\T_h)} \left(\|\mathcal{L}_h \mathcal{V}_h^{\pm}(t')\|_{L^2(\T)} + \|\mathcal{V}_h^{\pm}(t')\|_{L^2(\T_h)}\right),
\end{aligned}
\end{equation}
for any $0 \le \alpha \le 2$. Moreover, by the Cauchy-Schwarz inequality, we have
\begin{equation}\label{eq:I_3}
\begin{aligned}
\|P_{\le M}I_3(t')\|_{L^2(\T)} \lesssim&~{} M^{\frac32}\|P_{> \frac{\pi}{h}}\mathcal{L}_{h}\mathcal V_{h}^{\pm}(t')\|_{L^2(\T)} \|\mathcal{L}_h\mathcal{V}_h^{\pm}(t')\|_{L^2(\T)}\\
\lesssim&~{}h^\alpha M^{\frac32}\|\mathcal{L}_h \mathcal{V}_h^{\pm}(t')\|_{H^{\alpha}(\T)}\|\mathcal{L}_h \mathcal{V}_h^{\pm}(t')\|_{L^2(\T)},
\end{aligned}
\end{equation}
for any $\alpha \ge 0$. Collecting all in addition to Lemma \ref{Lem:discretization linearization inequality}, we obtain
\[\begin{aligned}
&~{}\|P_{\le M}(\mathcal L_{h} \mathcal V_{h}^{\pm} - \mathcal W^{\pm} )(t)\|_{L^2(\T)}\\ 
\lesssim&~{}  h^{\alpha}M^{\frac{3}{2}+\alpha} \int_0^t (1+|t'|)^{\alpha}\|\mathcal{V}_h^{\pm}(t')\|_{H^{2\alpha}(\T_h)}\|\mathcal{V}_h^{\pm}(t')\|_{H^{\alpha}(\T_h)} \; dt'\\
&~{}+ h^{\alpha}M^{\frac{3}{2}} \int_0^t \|\mathcal{V}_h^{\pm}(t')\|_{H^{\alpha}(\T_h)}\|\mathcal{V}_h^{\pm}(t')\|_{L^2(\T_h)} \; dt'\\
&~{}+ M^{\frac{3}{2}} \int_0^t \|(\mathcal L_{h} \mathcal V_{h}^{\pm} - \mathcal W^{\pm})(t')\|_{L^2(\T)} \left(\|\mathcal V_{h}^{\pm}(t')\|_{L^2(\T_h)}+\|\mathcal W^{\pm}(t')\|_{L^2(\T)}\right) \; dt'.
\end{aligned}\]
Taking $2\alpha = s$, we complete the proof.
\end{proof}

\begin{proposition}\label{prop:differnece estimates-2}
Let $0 \le s \le 1$ and $0 < h \le 1$. For $1 < M < \frac{\pi}{h}$, we have 
\begin{equation}\label{lowfreq:esti}
\begin{aligned}
&\left\|P_{ M \le \cdot \le \frac{\pi}{h}} \left( \mathcal{L}_h\mathcal{V}^{\pm}_h- \mathcal{W}^{\pm}\right)(t)
  \right\|_{L^2(\T)}  \\ 
\lesssim&~{} \frac{1}{M^{\frac12}} \sup_{t'\in[0,t]}\left(\|( \mathcal{L}_h\mathcal{V}_h^{\pm}-\mathcal{W}^{\pm})(t')\|_{L^{2}(\T)}\big(1+\|\mathcal{V}_h^{\pm}(t')\|_{H^s(\T_h)}+\|\mathcal{W}^{\pm}(t')\|_{H^s(\T)} \big)^2\right)\\ 
&~{}+h^{\frac{s}{2}}(1+|t'|)^{1+\frac{s}{2}} \sup_{t'\in[0,t]}\left(\left(1+\|\mathcal{V}_h^{\pm}(t')\|_{H^s(\T_h)}\right)\|\mathcal{V}_h^{\pm}(t')\|_{H^s(\T_h)}^2\right)\\
&~{}+  \int_0^t \| (\mathcal{L}_h\mathcal{V}_h^{\pm}-\mathcal{W}^{\pm})(t')\|_{L^2(\T)}
  \left( 1+ \|\mathcal{V}_h^{\pm}(t')\|_{H^s(\T_h)} + \| \mathcal{W}^{\pm}(t')\|_{H^s(\T)} \right)^3 dt' \\ 
\end{aligned}
\end{equation}
for any $\mathcal V_{h}^{\pm} \in H^s(\T_h)$ and $\mathcal W^{\pm} \in H^s(\T)$. The implicit constant does not depend on $h\in(0,1)$.
\end{proposition}
\begin{proof}
From \eqref{eq:renormalized FPU} and \eqref{eq:renormalized KdV} in addition to \eqref{linear interpolation multiplier}, we write for $M<|k|\le\frac{\pi}{h}$ that 
\[  \mathcal F(\mathcal{L}_h\mathcal{V}_h^{\pm} - \mathcal{W}^{\pm})(t,k) = (\widehat{J}_1+\widehat{J}_2+\widehat{J}_3+ \widehat{J}_4)(t,k),\]
where 
\[\widehat{J}_1(t,k) =-\frac14 \sum_{t'= 0,t}\mathcal F\left(\mathcal{L}_h\textup{B}_2^{\pm}(\mathcal{V}_h^{\pm},\mathcal{V}_h^{\pm})-\textup D_2^{\pm}(\mathcal{W}^{\pm},\mathcal{W}^{\pm})\right)(t',k),\]
\[\widehat{J}_2(t,k) =\frac12 \sum_{t'=0,t} \mathcal F\left(\mathcal{L}_h\textup{B}_3^{\pm}(\mathcal{V}_h^{\pm},\mathcal{V}_h^{\pm},\mathcal{V}_h^{\pm}) - \textup D_3^{\pm}(\mathcal{W}^{\pm},\mathcal{W}^{\pm},\mathcal{W}^{\pm})\right)(t',k),\]
\[\widehat{J}_3(t,k) = \int_0^t \left((\mathcal F_h(\mathcal L_h \mathcal{N}_{h,3}^{\pm}))_\mathcal{R}(\mathcal{V}_h^{\pm},\mathcal{V}_h^{\pm},\mathcal{V}_h^{\pm}) -(\mathcal F\mathcal{N}_3^{\pm})_\mathcal{R}(\mathcal{W}^{\pm},\mathcal{W}^{\pm},\mathcal{W}^{\pm})\right)(t',k)  \; dt',\]
and
\[\widehat{J}_4(t,k) = - \frac{1}{2}\int_0^t \mathcal F\left( \mathcal{L}_h\textup{B}_4^{\pm}(\mathcal{V}_h^{\pm},\mathcal{V}_h^{\pm},\mathcal{V}_h^{\pm},\mathcal{V}_h^{\pm}) - \textup{D}_4^{\pm}(\mathcal{W}^{\pm},\mathcal{W}^{\pm},\mathcal{W}^{\pm},\mathcal{W}^{\pm})\right)(t',k) \; dt'.\]

\medskip

\underline{{\bf $J_1$ estimate}}. We write 
\[\begin{aligned}
\left(\mathcal{L}_h\textup{B}_2^{\pm}(\mathcal{V}_h^{\pm},\mathcal{V}_h^{\pm})-\textup D_2^{\pm}(\mathcal{W}^{\pm},\mathcal{W}^{\pm})\right)(t') =&~{} \left(\mathcal{L}_h \textup{B}_2^{\pm}(\mathcal{V}_h^{\pm},\mathcal{V}_h^{\pm})-\textup D_2^{\pm}(\mathcal L_h\mathcal{V}_h^{\pm},\mathcal L_h\mathcal{V}_h^{\pm}) \right)(t')\\
&~{}+  \textup D_2^{\pm}(\mathcal L_h\mathcal{V}_h^{\pm}-\mathcal{W}^{\pm},\mathcal L_h\mathcal{V}_h^{\pm}+\mathcal{W}^{\pm})(t').
\end{aligned}\]
By \eqref{ineq:D2}, we immediately know
\[\begin{aligned}
&~{}\left\|P_{ M \le \cdot \le \frac{\pi}{h}}\left(\textup D_2^{\pm}(\mathcal L_h\mathcal{V}_h^{\pm}-\mathcal{W}^{\pm},\mathcal L_h\mathcal{V}_h^{\pm}+\mathcal{W}^{\pm})(t')\right)\right\|_{L^2(\T)}\\
\lesssim&~{} \frac1M\| (\mathcal L_h\mathcal{V}_h^{\pm}-\mathcal{W}^{\pm})(t') \|_{L^2(\T)}\left(\|\mathcal L_h\mathcal{V}_h^{\pm}(t')\|_{L^2(\T)} + \|\mathcal{W}^{\pm}(t')\|_{L^2(\T)}\right).
\end{aligned}\]
For the rest, by the symmetry over $k_1$ and $k_2$, we write
\[\begin{aligned}
&~{}\mathcal F\left(\mathcal{L}_h \textup{B}_2^{\pm}(\mathcal{V}_h^{\pm},\mathcal{V}_h^{\pm})-\textup D_2^{\pm}(\mathcal L_h\mathcal{V}_h^{\pm},\mathcal L_h\mathcal{V}_h^{\pm}) \right)(t',k)\\
=&~{}-\sum_{\substack{k=k_1+k_2 \\ k_1k_2 \neq 0 \\ |k_1|, |k_2| \le \frac{\pi}{h}}} \mathcal{M}_{h,1}^{2,\pm}(t',k,k_1,k_2) \widehat{\mathcal{V}_h^{\pm}}(t',k_1)  \widehat{\mathcal{V}_h^{\pm}}(t',k_2)\\
&~{}+\sum_{\substack{k=k_1+k_2 \\ k_1k_2 \neq 0 \\ |k_1|, |k_2| \le \frac{\pi}{h}}}\frac{k}{\Psi^2(k,k_1,k_2)}e^{\pm i t' \Psi^2(k,k_1,k_2)}\left( \widehat{\mathcal{L}_{h}\mathcal V_{h}^{\pm}} - \widehat{\mathcal{V}_h^{\pm}}  \right)(t',k_1)
  \left( \widehat{\mathcal{L}_{h}\mathcal V_{h}^{\pm}} + \widehat{\mathcal{V}_h^{\pm}} \right)(t',k_2)\\
&~{}+\sum_{\substack{k=k_1+k_2 \\ k_1k_2 \neq 0 \\ |k_1| > \frac{\pi}{h} \text{or} |k_2| > \frac{\pi}{h}}}\frac{k}{\Psi^2(k,k_1,k_2)}e^{\pm i t' \Psi^2(k,k_1,k_2)}\widehat{\mathcal{L}_{h}\mathcal V_{h}^{\pm}}(t',k_1)
\widehat{\mathcal{L}_{h}\mathcal V_{h}^{\pm}}(t',k_2)\\
=:&~{} \widehat{J}_{11}+\widehat{J}_{12}+\widehat{J}_{13},
\end{aligned}\]
where $\mathcal M_{h,1}^{2,\pm}(t',k,k_1,k_2)$ is introduced in Lemma \ref{Lem:M1}. For $J_{11}$, by the symmetry on $k_1$ and $k_2$, we may assume that $|k_1| \ge |k_2|$. Since $k=k_1+k_2$, we have $|k| \lesssim |k_1|$, and thus, by Lemma \ref{Lem:M1}, we have for $0 \le \alpha \le \frac12$ that
\[\begin{aligned}
  &~{}\| P_{ M \le \cdot \le \frac{\pi}{h}} J_{11}(t')\|_{L^2(\T)} \\ 
  \ls&~{} h^\alpha(1+|t'|)^\alpha\Bigg(\sum_{M \le |k| \le \frac{\pi}{h}}\Bigg|\sum_{\substack{0 < |k_1| \le \frac{\pi}{h}\\k_1 \neq k}} \frac{|k_1|^{\alpha}|\widehat{\mathcal{V}_h^{\pm}}(t',k_1)||k-k_1|^{\alpha}|\widehat{\mathcal{V}_h^{\pm}}(t',k-k_1)|}{|k_1|^{1-2\alpha}|k-k_1|}\Bigg|^2\Bigg)^{\frac12}\\
\lesssim&~{}h^\alpha(1+|t'|)^\alpha\Bigg(\sum_{M \le |k| \le \frac{\pi}{h}}\Bigg(\sum_{\substack{0 < |k_1| \le \frac{\pi}{h}\\k_1 \neq k}}\frac{1}{|k-k_1|^2}\Bigg) \Bigg(\sum_{\substack{0 < |k_1| \le \frac{\pi}{h}\\k_1 \neq k}}|k_1|^{2\alpha}|\widehat{\mathcal{V}_h^{\pm}}(t',k_1)|^2|k-k_1|^{2\alpha}|\widehat{\mathcal{V}_h^{\pm}}(t',k-k_1)|^2\Bigg)\Bigg)^{\frac12}\\
\lesssim&~{}h^\alpha(1+|t'|)^\alpha\|\mathcal{V}_h^{\pm}(t')\|_{H^\alpha(\T_h)}\|\mathcal{V}_h^{\pm}(t')\|_{H^\alpha(\T_h)}.
\end{aligned}\]
For $J_{12}$ and $J_{13}$, similarly as \eqref{eq:I_2} and \eqref{eq:I_3}, respectively, but using \eqref{ineq:D2}, we obtain
\[\begin{aligned}
\|P_{ M \le \cdot \le \frac{\pi}{h}} J_{12}(t')\|_{L^2(\T)} \lesssim&~{} \frac{1}{M}\left(\sum_{|k| \le \frac{\pi}{h}}\left|\widehat{\mathcal{L}_{h}\mathcal V_{h}^{\pm}} - \widehat{\mathcal{V}_h^{\pm}} \right|^2\right)^{\frac12} \left(\sum_{|k| \le \frac{\pi}{h}}\left|\widehat{\mathcal{L}_{h}\mathcal V_{h}^{\pm}} + \widehat{\mathcal{V}_h^{\pm}} \right|^2\right)^{\frac12}\\
\lesssim&~{} \frac{h^\alpha}{M}\|\mathcal{V}_h^{\pm}(t')\|_{H^{\alpha}(\T_h)} \left(\|\mathcal{L}_h \mathcal{V}_h^{\pm}(t')\|_{L^2(\T)} + \|\mathcal{V}_h^{\pm}(t')\|_{L^2(\T_h)}\right),
\end{aligned}\]
for any $0 \le \alpha \le 2$, and
\[\begin{aligned}
  \|P_{ M \le \cdot \le \frac{\pi}{h}} J_{13}(t')\|_{L^2(\T)}\lesssim&~{} \frac{1}{M}\|P_{> \frac{\pi}{h}}\mathcal{L}_{h}\mathcal V_{h}^{\pm}(t')\|_{L^2(\T)} \|\mathcal{L}_h\mathcal{V}_h^{\pm}(t')\|_{L^2(\T)}\\
\lesssim&~{}\frac{h^{\alpha}}{M}\|\mathcal{L}_h \mathcal{V}_h^{\pm}(t')\|_{H^{\alpha}(\T)}\|\mathcal{L}_h \mathcal{V}_h^{\pm}(t')\|_{L^2(\T)},
\end{aligned}\]
for any $\alpha \ge 0$. Collecting all and taking $2\alpha = s$ in addition to Lemma \ref{Lem:discretization linearization inequality}, we obtain for $t'=0, t$ that
\begin{equation}\label{eq:J_1 term}
\begin{aligned}
\|P_{ M \le \cdot \le \frac{\pi}{h}} J_1(t')\|_{L^2(\T)} \lesssim&~{} \frac1M\| (\mathcal L_h\mathcal{V}_h^{\pm}-\mathcal{W}^{\pm})(t') \|_{L^2(\T)}\left(\|\mathcal{V}_h^{\pm}(t')\|_{H^s(\T_h)} + \|\mathcal{W}^{\pm}(t')\|_{H^s(\T)}\right)\\
&~{}+h^{\frac{s}{2}}(1+|t'|)^{\frac{s}{2}}\|\mathcal{V}_h^{\pm}(t')\|_{H^s(\T_h)}^2+\frac{h^{\frac{s}{2}}}{M}\|\mathcal{V}_h^{\pm}(t')\|_{H^s(\T_h)}^2.
\end{aligned}
\end{equation}

\medskip

\underline{{\bf $J_2$ estimate}}. We write 
\[\begin{aligned}
\mathcal{L}_h\textup{B}_3^{\pm}(\mathcal{V}_h^{\pm},\mathcal{V}_h^{\pm},\mathcal{V}_h^{\pm})-\textup{D}_3^{\pm}(\mathcal{W}^{\pm},\mathcal{W}^{\pm},\mathcal{W}^{\pm})(t') =&~{} \left(\mathcal{L}_h\textup{B}_3^{\pm}(\mathcal{V}_h^{\pm},\mathcal{V}_h^{\pm},\mathcal{V}_h^{\pm})
-\textup{D}_3^{\pm}(\mathcal{L}_h\mathcal{V}_h^{\pm},\mathcal{L}_h\mathcal{V}_h^{\pm},\mathcal{L}_h\mathcal{V}_h^{\pm})\right)(t')\\
&~{}+\textup{D}_3^{\pm}\left(\mathcal{L}_h\mathcal{V}_h^{\pm}-\mathcal{W}^{\pm},\mathcal{L}_h\mathcal{V}_h^{\pm},\mathcal{L}_h\mathcal{V}_h^{\pm}\right))(t')\\
&~{}+\textup{D}_3^{\pm}\left(\mathcal{W}^{\pm},\mathcal{L}_h\mathcal{V}_h^{\pm}-\mathcal{W}^{\pm},\mathcal{L}_h\mathcal{V}_h^{\pm}\right)(t')\\ 
&~{}+\textup{D}_3^{\pm}\left(\mathcal{W}^{\pm},\mathcal{W}^{\pm},\mathcal{L}_h\mathcal{V}_h^{\pm}-\mathcal{W}^{\pm}\right)(t').
\end{aligned}\]
By \eqref{ineq:D3}, we immediately know 
\[\begin{aligned}
\|P_{ M \le \cdot \le \frac{\pi}{h}}\textup{D}_3^{\pm}\left(\mathcal{L}_h\mathcal{V}_h^{\pm}-\mathcal{W}^{\pm},\mathcal{L}_h\mathcal{V}_h^{\pm},\mathcal{L}_h\mathcal{V}_h^{\pm}\right))(t')\|_{L^{2}(\T)} \lesssim \frac{1}{M^{\frac12}}\|(\mathcal{L}_h\mathcal{V}_h^{\pm}-\mathcal{W}^{\pm})(t')\|_{L^2(\T)}\|\mathcal{L}_h\mathcal{V}_h^{\pm}(t')\|_{L^2(\T)}^2.
\end{aligned}\]
Analogously, we have
\[\begin{aligned}
&~{}\|P_{ M \le \cdot \le \frac{\pi}{h}}\textup{D}_3^{\pm}\left(\mathcal{W}^{\pm},\mathcal{L}_h\mathcal{V}_h^{\pm}-\mathcal{W}^{\pm},\mathcal{L}_h\mathcal{V}_h^{\pm}\right)(t')\|_{L^{2}(\T)}\\
\lesssim&~{} \frac{1}{M^{\frac12}}\|(\mathcal{L}_h\mathcal{V}_h^{\pm}-\mathcal{W}^{\pm})(t')\|_{L^2(\T)}\|\mathcal{L}_h\mathcal{V}_h^{\pm}(t')\|_{L^2(\T)}\|\mathcal{W}^{\pm}(t')\|_{L^2(\T)}
\end{aligned}\]
and
\[\begin{aligned}
\|P_{ M \le \cdot \le \frac{\pi}{h}}\textup{D}_3^{\pm}\left(\mathcal{W}^{\pm},\mathcal{W}^{\pm},\mathcal{L}_h\mathcal{V}_h^{\pm}-\mathcal{W}^{\pm}\right)(t')\|_{L^{2}(\T)} \lesssim \frac{1}{M^{\frac12}}\|(\mathcal{L}_h\mathcal{V}_h^{\pm}-\mathcal{W}^{\pm})(t')\|_{L^2(\T)}\|\mathcal{W}^{\pm}(t')\|_{L^2(\T)}^2.
\end{aligned}\]
For the rest, we write for $M<|k|\le\frac{\pi}{h}$ that
\[\mathcal F\left(\mathcal{L}_h\textup{B}_3^{\pm}(\mathcal{V}_h^{\pm},\mathcal{V}_h^{\pm},\mathcal{V}_h^{\pm})
-\textup{D}_3^{\pm}(\mathcal{L}_h\mathcal{V}_h^{\pm},\mathcal{L}_h\mathcal{V}_h^{\pm},\mathcal{L}_h\mathcal{V}_h^{\pm})\right)(t',k) = \widehat{J}_{21}+\widehat{J}_{22}+\widehat{J}_{23},\]
where
\[\widehat{J}_{21} = -\sum_{(k_1,k_2,k_3)\in\mathcal{A}(k)}\mathcal M_{h}^{3,\pm}(t',k,k_{1},k_{2},k_{3})\widehat{\mathcal{V}_h^{\pm}}(t,k_1)  \widehat{\mathcal{V}_h^{\pm}}(t,k_2)\widehat{ \mathcal{V}_h^{\pm}}(t,k_3),\]
here $\mathcal M_{h}^{3,\pm}(t',k,k_{1},k_{2},k_{3})$ and $\mathcal{A}(k)$ are introduced in Lemma \ref{Lem:M2} and \eqref{eq:A(k)}, respectively,
\[\begin{aligned}
\widehat{J}_{22} = -\sum_{(k_1,k_2,k_3)\in\mathcal{A}(k)}\frac{2e^{\pm it'\Psi^3(k,k_1,k_2,k_3)}}{k_1\Psi^3(k,k_1,k_2,k_3)}&~{}\Big(\big(\widehat{\mathcal{L}_h\mathcal{V}_h^{\pm}} - \widehat{\mathcal{V}_{h}^{\pm}}\big)(t',k_1)  \widehat{\mathcal{V}_{h}^{\pm}}(t',k_2) 
  \widehat{\mathcal{V}_{h}^{\pm}}(t',k_3)\\
&~{}+\widehat{\mathcal{L}_h\mathcal{V}_h^{\pm}}(t',k_1)\big(\widehat{\mathcal{L}_h\mathcal{V}_h^{\pm}} - \widehat{\mathcal{V}_{h}^{\pm}}\big)(t',k_2)  \widehat{\mathcal{V}_{h}^{\pm}}(t',k_3)\\
&~{}+\widehat{\mathcal{L}_h\mathcal{V}_h^{\pm}}(t',k_1)\widehat{\mathcal{L}_h\mathcal{V}_h^{\pm}}(t',k_2)\big(\widehat{\mathcal{L}_h\mathcal{V}_h^{\pm}} - \widehat{\mathcal{V}_{h}^{\pm}}\big)(t',k_3)\Big),
\end{aligned}\]
and
\[\widehat{J}_{23} =-\sum_{\substack{(k_1,k_2,k_3)\in \Lambda(k) \\ \max(|k_1|,|k_2|,|k_3|) > \frac{\pi}{h}}}\frac{2e^{\pm it'\Psi^3(k,k_1,k_2,k_3)}}{k_1\Psi^3(k,k_1,k_2,k_3)} \widehat{\mathcal{L}_h\mathcal{V}_h^{\pm}}(t',k_1)\widehat{\mathcal{L}_h\mathcal{V}_h^{\pm}}(t',k_2)\widehat{\mathcal{L}_h\mathcal{V}_h^{\pm}}(t',k_3).\]
For $J_{21}$, note that $|k_3| \le \max(|k_1|,|k_2|,|k_3|)$. If $\max(|k|,|k_1|,|k_2|,|k_3|) = |k_2|$\footnote{The estimate of $J_{21}$ does not depend on the choice $\max(|k|,|k_1|,|k_2|,|k_3|) = |k_2|$.}, by Lemma~\ref{Lem:M2}, we have
\[\begin{aligned}
  &~{}\| P_{ M \le \cdot \le \frac{\pi}{h}} J_{21}(t')\|_{L^2(\T)} \\ 
  \ls&~{} h^\alpha(1+|t'|)^\alpha\Bigg(\sum_{M \le |k| \le \frac{\pi}{h}}\Bigg|\sum_{\substack{0 < |k_1|, |k_2| \le \frac{\pi}{h}\\k_1+k_2 \neq k}} \frac{|\widehat{\mathcal{V}_h^{\pm}}(t',k_1)||k_2|^{2\alpha}|\widehat{\mathcal{V}_h^{\pm}}(t',k_2)||\widehat{\mathcal{V}_h^{\pm}}(t',k-k_1-k_2)|}{|k_1||k-k_1-k_2|}\Bigg|^2\Bigg)^{\frac12}\\
\lesssim&~{}h^\alpha(1+|t'|)^\alpha\left(\sum_{\substack{0 < |k_1|, |k_2| \le \frac{\pi}{h}\\k_1+k_2 \neq k}}\frac{1}{|k_1|^2|k-k_1-k_2|^2} \right)^{\frac12}\\
&~{}\times \left(\sum_{M \le |k| \le \frac{\pi}{h}}\sum_{\substack{0 < |k_1|, |k_2| \le \frac{\pi}{h}\\k_1+k_2 \neq k}}|\widehat{\mathcal{V}_h^{\pm}}(t',k_1)|^2|k_2|^{4\alpha}|\widehat{\mathcal{V}_h^{\pm}}(t',k_2)|^2|\widehat{\mathcal{V}_h^{\pm}}(t',k-k_1-k_2)|^2\right)^{\frac12}\\
\lesssim&~{}h^\alpha(1+|t'|)^\alpha\|\mathcal{V}_h^{\pm}(t')\|_{H^{2\alpha}(\T_h)}\|\mathcal{V}_h^{\pm}(t')\|_{L^2(\T_h)}^2,
\end{aligned}\]
for any $0 \le \alpha \le 1$. For the other cases, we also have the same result. For $J_{22}$ and $J_{23}$, similarly as \eqref{eq:I_2} and \eqref{eq:I_3}, respectively, but using \eqref{ineq:D3}, we obtain
\[\begin{aligned}
\|P_{ M \le \cdot \le \frac{\pi}{h}} J_{22}(t')\|_{L^2(\T)} \lesssim&~{} \frac{1}{M^{\frac12}}\left(\sum_{|k| \le \frac{\pi}{h}}\left|\widehat{\mathcal{L}_{h}\mathcal V_{h}^{\pm}} - \widehat{\mathcal{V}_h^{\pm}} \right|^2\right)^{\frac12} \\
&~{}\times \left(\|\mathcal{L}_h\mathcal{V}_h^{\pm}(t')\|_{L^2(\T)}^2+\|\mathcal{L}_h\mathcal{V}_h^{\pm}(t')\|_{L^2(\T)}\|\mathcal{V}_h^{\pm}(t')\|_{L^2(\T_h)}+\|\mathcal{V}_h^{\pm}(t')\|_{L^2(\T_h)}^2\right)\\
\lesssim&~{} \frac{h^\alpha}{M^{\frac12}}\|\mathcal{V}_h^{\pm}(t')\|_{H^{\alpha}(\T_h)} \left(\|\mathcal{L}_h\mathcal{V}_h^{\pm}(t')\|_{L^2(\T)}+\|\mathcal{V}_h^{\pm}(t')\|_{L^2(\T_h)}\right)^2,
\end{aligned}\]
for any $0 \le \alpha \le 2$, and
\[\begin{aligned}
  \|P_{ M \le \cdot \le \frac{\pi}{h}} J_{23}(t')\|_{L^2(\T)}\lesssim&~{} \frac{1}{M^{\frac12}}\|P_{> \frac{\pi}{h}}\mathcal{L}_{h}\mathcal V_{h}^{\pm}(t')\|_{L^2(\T)} \|\mathcal{L}_h\mathcal{V}_h^{\pm}(t')\|_{L^2(\T)}^2\\
\lesssim&~{}\frac{h^{\alpha}}{M^{\frac12}}\|\mathcal{L}_h \mathcal{V}_h^{\pm}(t')\|_{H^{\alpha}(\T)}\|\mathcal{L}_h \mathcal{V}_h^{\pm}(t')\|_{L^2(\T)}^2,
\end{aligned}\]
for any $\alpha \ge 0$. Collecting all and taking $2\alpha = s$ in addition to Lemma \ref{Lem:discretization linearization inequality}, we obtain for $t'=0, t$ that
\begin{equation}\label{eq:J_2 term}
\begin{aligned}
\|P_{ M \le \cdot \le \frac{\pi}{h}} J_2(t')\|_{L^2(\T)} \lesssim&~{} \frac{1}{M^{\frac12}}\| (\mathcal L_h\mathcal{V}_h^{\pm}-\mathcal{W}^{\pm})(t') \|_{L^2(\T)}\left(\|\mathcal{V}_h^{\pm}(t')\|_{H^s(\T_h)} + \|\mathcal{W}^{\pm}(t')\|_{H^s(\T)}\right)^2\\
&~{}+h^{\frac{s}{2}}(1+|t'|)^{\frac{s}{2}}\|\mathcal{V}_h^{\pm}(t')\|_{H^s(\T_h)}^3+\frac{h^{\frac{s}{2}}}{M^{\frac12}}\|\mathcal{V}_h^{\pm}(t')\|_{H^s(\T_h)}^3.
\end{aligned}
\end{equation}

\medskip
\underline{{\bf $J_3$ estimate}}. We write 
\[\left((\mathcal L_h\mathcal{N}_{h,3}^{\pm})_\mathcal{R}(\mathcal{V}_h^{\pm},\mathcal{V}_h^{\pm},\mathcal{V}_h^{\pm}) -(\mathcal{N}_3^{\pm})_\mathcal{R}(\mathcal{W}^{\pm},\mathcal{W}^{\pm},\mathcal{W}^{\pm})\right)(t') = J_{31}+J_{32}+J_{33}+J_{34},\]
where
\[\widehat{J}_{31} = \mp 2i \left(\frac{4\sin\left(\frac{hk}{2}\right)}{h^2k^2}\frac{\cos\left(\frac{hk}{2}\right)\cos\left(\frac{hk}{4}\right)}{\frac{4}{h}\sin\left(\frac{hk}{4}\right)}-\frac1k \right)|\widehat{\mathcal{V}_h^{\pm}}(t',k)|^2\widehat{\mathcal{V}_h^{\pm}}(t',k) ,\]
\[\widehat{J}_{32} = \mp \frac{2i}{k} \left(|\widehat{\mathcal{V}_h^{\pm}}|^2\widehat{\mathcal{V}_h^{\pm}} - |\widehat{\mathcal L_h\mathcal{V}_h^{\pm}}|^2\widehat{\mathcal L_h\mathcal{V}_h^{\pm}}\right)(t',k),\]
\[\widehat{J}_{33} = \mp\frac{2i}{k}\left(|\widehat{\mathcal L_h\mathcal{V}_h^{\pm}}|^2\widehat{\mathcal L_h\mathcal{V}_h^{\pm}}-|\widehat{\mathcal{W}^{\pm}}|^2\widehat{\mathcal{W}^{\pm}}\right)(t',k),\]
and
\[\widehat{J}_{34} = \pm \frac{2ih\sin^3\left(\frac{hk}{2}\right)}{h^2k^2}\widehat{\mathcal{V}_h^{\pm}}(t',k)\sum_{\substack{k_1 \in (\T_h)^* \\ k_1 \neq \pm k}} |\widehat{\mathcal{V}_h^{\pm}}(t',k_1)|^2.\]
One can easily see that
\[\begin{aligned}
\|P_{ M \le \cdot \le \frac{\pi}{h}} J_{32}(t')\|_{L^2(\T)} \lesssim&~{} \frac{1}{M}\left(\sum_{|k| \le \frac{\pi}{h}}\left|\widehat{\mathcal{L}_{h}\mathcal V_{h}^{\pm}}(k) - \widehat{\mathcal{V}_h^{\pm}}(k) \right|^2\right)^{\frac12}\left(\|\mathcal L_h\mathcal{V}_h^{\pm}(t')\|_{L^2(\T)} + \|\mathcal{V}_h^{\pm}(t')\|_{L^2(\T_h)}\right)^2\\
\lesssim&~{}\frac{h^{\alpha}}{M}\|\mathcal{V}_h^{\pm}(t')\|_{H^{\alpha}(\T_h)}\left(\|\mathcal L_h\mathcal{V}_h^{\pm}(t')\|_{L^2(\T)} + \|\mathcal{V}_h^{\pm}(t')\|_{L^2(\T_h)}\right)^2,
\end{aligned}\]
for any $0 \le \alpha \le 2$ due to Lemma \ref{lem:LOWLh}, and 
\[\|P_{ M \le \cdot \le \frac{\pi}{h}} J_{33}(t')\|_{L^2(\T)} \lesssim \frac{1}{M}\|(\mathcal L_h\mathcal{V}_h^{\pm}-\mathcal{W}^{\pm})(t') \|_{L^2(\T)}\left(\|\mathcal L_h\mathcal{V}_h^{\pm}(t')\|_{L^2(\T)} + \|\mathcal{W}^{\pm}(t')\|_{L^2(\T)}\right)^2.\]
Moreover, since 
\[\left|\frac{2ih\sin^3\left(\frac{hk}{2}\right)}{h^2k^2}\right| \lesssim h \quad \mbox{and} \quad \sum_{\substack{k_1 \in (\T_h)^* \\ k_1 \neq \pm k}} |\widehat{\mathcal{V}_h^{\pm}}(t',k_1)|^2 \lesssim \|\mathcal{V}_h^{\pm}(t')\|_{L^2(\T_h)}^2,\]
we have
\[\|P_{ M \le \cdot \le \frac{\pi}{h}} J_{34}(t')\|_{L^2(\T)} \lesssim h\|\mathcal{V}_h^{\pm}(t')\|_{L^2(\T_h)}^3.\]
For the rest $J_{31}$, by \eqref{Lh-1}, \eqref{eq:M^3-2} and \eqref{Lh}, we have
\[\begin{aligned}
&~{}\|P_{ M \le \cdot \le \frac{\pi}{h}} J_{31}(t')\|_{L^2(\T)}\\
&~{}\lesssim
\left(\sum_{M \le |k| \le \frac{\pi}{h}}\left|\frac{4\sin\left(\frac{hk}{2}\right)}{h^2k^2}\frac{\cos\left(\frac{hk}{2}\right)\cos\left(\frac{hk}{4}\right)}{\frac{4}{h}\sin\left(\frac{hk}{4}\right)}-\frac1k \right|^2|\widehat{\mathcal{V}_h^{\pm}}(t',k)|^6\right)^{\frac12}\\
\lesssim&~{} \left(\sum_{M \le |k| \le \frac{\pi}{h}}\left(\frac{h^{\alpha}|k|^{\alpha}}{|k|}\right)^2|\widehat{\mathcal{V}_h^{\pm}}(t',k)|^6\right)^{\frac12}\\
\lesssim&~{}\frac{h^{\alpha}}{M}\|\mathcal{V}_h^{\pm}(t')\|_{H^{\alpha}(\T_h)}\|\mathcal{V}_h^{\pm}(t')\|_{L^2(\T_h)}^2,
\end{aligned}\]
for any $0 \le \alpha \le 2$. Collecting all and taking $\alpha = s$ in addition to Lemma \ref{Lem:discretization linearization inequality}, we obtain
\begin{equation}\label{eq:J_3 term}
\begin{aligned}
&~{}\|P_{ M \le \cdot \le \frac{\pi}{h}} J_3(t')\|_{L^2(\T)} \\
\lesssim&~{} \int_0^t \Bigg(\frac{1}{M}\|(\mathcal{V}_h^{\pm}-\mathcal{W}^{\pm})(t') \|_{L^2(\T)}\left(\|\mathcal L_h\mathcal{V}_h^{\pm}(t')\|_{H^s(\T)} + \|\mathcal{W}^{\pm}(t')\|_{H^s(\T)}\right)^2\\
&~{}\hspace{2em}+\frac{h^s}{M}\|\mathcal{V}_h^{\pm}(t')\|_{H^{s}(\T_h)}^3 + h\|\mathcal{V}_h^{\pm}(t')\|_{H^{s}(\T_h)}^3\Bigg)\;dt'.
\end{aligned}
\end{equation}

\medskip
\underline{{\bf $J_4$ estimate}}. We write 
\[\begin{aligned}
&~{}\left(\mathcal{L}_h\textup{B}_4^{\pm}(\mathcal{V}_h^{\pm},\mathcal{V}_h^{\pm},\mathcal{V}_h^{\pm},\mathcal{V}_h^{\pm}) - \textup{D}_4^{\pm}(\mathcal{W}^{\pm},\mathcal{W}^{\pm},\mathcal{W}^{\pm},\mathcal{W}^{\pm})\right)(t')\\
=&~{} \left(\mathcal{L}_h\textup{B}_4^{\pm}(\mathcal{V}_h^{\pm},\mathcal{V}_h^{\pm},\mathcal{V}_h^{\pm},\mathcal{V}_h^{\pm}) - \textup{D}_4^{\pm}(\mathcal L_h\mathcal{V}_h^{\pm},\mathcal L_h\mathcal{V}_h^{\pm},\mathcal L_h\mathcal{V}_h^{\pm},\mathcal L_h\mathcal{V}_h^{\pm})\right)(t') \\
&~{}+\textup{D}_4^{\pm}(\mathcal{L}_h\mathcal{V}_h^{\pm}-\mathcal{W}^{\pm},\mathcal{L}_h\mathcal{V}_h^{\pm},\mathcal{L}_h\mathcal{V}_h^{\pm},\mathcal{L}_h\mathcal{V}_h^{\pm})(t')\\
&~{}+\textup{D}_4^{\pm}(\mathcal{W}^{\pm},\mathcal{L}_h\mathcal{V}_h^{\pm}-\mathcal{W}^{\pm},\mathcal{L}_h\mathcal{V}_h^{\pm},\mathcal{L}_h\mathcal{V}_h^{\pm})(t')\\
&~{}+\textup{D}_4^{\pm}(\mathcal{W}^{\pm},\mathcal{W}^{\pm},\mathcal{L}_h\mathcal{V}_h^{\pm}-\mathcal{W}^{\pm},\mathcal{L}_h\mathcal{V}_h^{\pm})(t')\\
&~{}+\textup{D}_4^{\pm}(\mathcal{W}^{\pm},\mathcal{W}^{\pm},\mathcal{W}^{\pm},\mathcal{L}_h\mathcal{V}_h^{\pm}-\mathcal{W}^{\pm})(t').
\end{aligned}\]
By \eqref{ineq:D4}, we immediately know 
\[\begin{aligned}
&~{}\|P_{ M \le \cdot \le \frac{\pi}{h}}\textup{D}_4^{\pm}(\mathcal{L}_h\mathcal{V}_h^{\pm}-\mathcal{W}^{\pm},\mathcal{L}_h\mathcal{V}_h^{\pm},\mathcal{L}_h\mathcal{V}_h^{\pm},\mathcal{L}_h\mathcal{V}_h^{\pm})(t')\|_{L^{2}(\T)} \\
\lesssim&~{} \|(\mathcal{L}_h\mathcal{V}_h^{\pm}-\mathcal{W}^{\pm})(t')\|_{L^2(\T)}\|\mathcal{L}_h\mathcal{V}_h^{\pm}(t')\|_{L^2(\T)}^3.
\end{aligned}\]
Analogously, we have
\[\begin{aligned}
&~{}\|P_{ M \le \cdot \le \frac{\pi}{h}}\textup{D}_4^{\pm}(\mathcal{W}^{\pm},\mathcal{L}_h\mathcal{V}_h^{\pm}-\mathcal{W}^{\pm},\mathcal{L}_h\mathcal{V}_h^{\pm},\mathcal{L}_h\mathcal{V}_h^{\pm})(t')\|_{L^{2}(\T)}\\
\lesssim&~{} \|(\mathcal{L}_h\mathcal{V}_h^{\pm}-\mathcal{W}^{\pm})(t')\|_{L^2(\T)}\|\mathcal{W}^{\pm}(t')\|_{L^2(\T)}\|\mathcal{L}_h\mathcal{V}_h^{\pm}(t')\|_{L^2(\T)}^2,
\end{aligned}\]
\[\begin{aligned}
&~{}\|P_{ M \le \cdot \le \frac{\pi}{h}}\textup{D}_4^{\pm}(\mathcal{W}^{\pm},\mathcal{W}^{\pm},\mathcal{L}_h\mathcal{V}_h^{\pm}-\mathcal{W}^{\pm},\mathcal{L}_h\mathcal{V}_h^{\pm})(t')\|_{L^{2}(\T)}\\
\lesssim&~{} \|(\mathcal{L}_h\mathcal{V}_h^{\pm}-\mathcal{W}^{\pm})(t')\|_{L^2(\T)}\|\mathcal{W}^{\pm}(t')\|_{L^2(\T)}^2\|\mathcal{L}_h\mathcal{V}_h^{\pm}(t')\|_{L^2(\T)},
\end{aligned}\]
and
\[\begin{aligned}
\|P_{ M \le \cdot \le \frac{\pi}{h}}\textup{D}_4^{\pm}(\mathcal{W}^{\pm},\mathcal{W}^{\pm},\mathcal{W}^{\pm},\mathcal{L}_h\mathcal{V}_h^{\pm}-\mathcal{W}^{\pm})(t')\|_{L^{2}(\T)} \lesssim \|(\mathcal{L}_h\mathcal{V}_h^{\pm}-\mathcal{W}^{\pm})(t')\|_{L^2(\T)}\|\mathcal{W}^{\pm}(t')\|_{L^2(\T)}^3.
\end{aligned}\]
For the rest, we write for $M < |k| \le \frac{\pi}{h}$ that
\[\mathcal F\left(\mathcal{L}_h\textup{B}_4^{\pm}(\mathcal{V}_h^{\pm},\mathcal{V}_h^{\pm},\mathcal{V}_h^{\pm},\mathcal{V}_h^{\pm}) - \textup{D}_4^{\pm}(\mathcal L_h\mathcal{V}_h^{\pm},\mathcal L_h\mathcal{V}_h^{\pm},\mathcal L_h\mathcal{V}_h^{\pm},\mathcal L_h\mathcal{V}_h^{\pm})\right)(t',k) = \widehat{J}_{41}+\widehat{J}_{42}+\widehat{J}_{43},\]
where
\[\widehat{J}_{41} = \pm \frac{i}{2}\sum_{(k_1,k_2,k_3,k_4) \in \mathcal B(k)}\mathcal M_{h}^{4,\pm}(t',k,k_{1},k_{2},k_{3},k_4)\widehat{ \mathcal{V}_h^{\pm}}(t',k_1)\widehat{ \mathcal{V}_h^{\pm}}(t',k_2)\widehat{ \mathcal{V}_h^{\pm}}(t',k_3)\widehat{ \mathcal{V}_h^{\pm}}(t',k_4),\]
here $\mathcal M_{h}^{4,\pm}(t',k,k_{1},k_{2},k_{3},k_4)$ and $\mathcal{B}(k)$ are introduced in Lemma \ref{Lem:Mh4} and \eqref{eq:B(k)}, respectively,
\[\begin{aligned}
\widehat{J}_{42} = \pm \frac{i}{2}\sum_{(k_1,k_2,k_3,k_4) \in \mathcal B(k)}&~{}\frac{2(k_1+k_2)+k_3}{k_3\Psi^3(k,k_1+k_2,k_3,k_4)}e^{\pm it' \Psi^4(k,k_1,k_2,k_3,k_4)}\\
\times &~{}\Big(\big(\widehat{\mathcal{L}_h\mathcal{V}_h^{\pm}} - \widehat{\mathcal{V}_{h}^{\pm}}\big)(t',k_1)  \widehat{\mathcal{V}_{h}^{\pm}}(t',k_2) 
  \widehat{\mathcal{V}_{h}^{\pm}}(t',k_3)\widehat{\mathcal{V}_{h}^{\pm}}(t',k_4)\\
&~{}+\widehat{\mathcal{L}_h\mathcal{V}_h^{\pm}}(t',k_1)\big(\widehat{\mathcal{L}_h\mathcal{V}_h^{\pm}} - \widehat{\mathcal{V}_{h}^{\pm}}\big)(t',k_2)  \widehat{\mathcal{V}_{h}^{\pm}}(t',k_3) \widehat{\mathcal{V}_{h}^{\pm}}(t',k_4)\\
&~{}+\widehat{\mathcal{L}_h\mathcal{V}_h^{\pm}}(t',k_1)\widehat{\mathcal{L}_h\mathcal{V}_h^{\pm}}(t',k_2)\big(\widehat{\mathcal{L}_h\mathcal{V}_h^{\pm}} - \widehat{\mathcal{V}_{h}^{\pm}}\big)(t',k_3)\widehat{\mathcal{V}_{h}^{\pm}}(t',k_4)\\
&~{}+\widehat{\mathcal{L}_h\mathcal{V}_h^{\pm}}(t',k_1)\widehat{\mathcal{L}_h\mathcal{V}_h^{\pm}}(t',k_2)\widehat{\mathcal{L}_h\mathcal{V}_h^{\pm}}(t',k_3)\big(\widehat{\mathcal{L}_h\mathcal{V}_h^{\pm}} - \widehat{\mathcal{V}_{h}^{\pm}}\big)(t',k_4)\Big),
\end{aligned}\]
and
\[\begin{aligned}
\widehat{J}_{43} = \pm \frac{i}{2}\sum_{\substack{(k_1,k_2,k_3,k_4) \in  \mathcal B(k) \\ \max(|k_j|:j=1,2,3,4) > \frac{\pi}{h}}}&~{}\frac{2(k_1+k_2)+k_3}{k_3\Psi^3(k,k_1+k_2,k_3,k_4)}e^{\pm it' \Psi^4(k,k_1,k_2,k_3,k_4)}\\
&~{}\times \widehat{\mathcal{L}_h\mathcal{V}_h^{\pm}}(t',k_1)\widehat{\mathcal{L}_h\mathcal{V}_h^{\pm}}(t',k_2)\widehat{\mathcal{L}_h\mathcal{V}_h^{\pm}}(t',k_3)\widehat{\mathcal{L}_h\mathcal{V}_h^{\pm}}(t',k_4).
\end{aligned}\]
For $J_{41}$, let $k_{max} = \max(|k_1|, |k_2|, |k_3|, |k_{12}|, |k_{34}|, |k_{123}|, |k_{124}|)$. By Lemma \ref{Lem:Mh4} and $k_{12} = k_{123} - k_3$, $J_{41}$ is divided by the following four terms: 
\begin{equation}\label{eq:j_41-1}
\sum_{(k_1,k_2,k_3,k_4) \in \mathcal B(k)}\frac{k_{max}^{\alpha}|k_{123}k_{124}k_{34}|^\alpha}{|k_{123}k_{124}k_{34}|}\widehat{ \mathcal{V}_h^{\pm}}(t',k_1)\widehat{ \mathcal{V}_h^{\pm}}(t',k_2)\widehat{ \mathcal{V}_h^{\pm}}(t',k_3)\widehat{ \mathcal{V}_h^{\pm}}(t',k_4),
\end{equation}
\begin{equation}\label{eq:j_41-2}
\sum_{(k_1,k_2,k_3,k_4) \in \mathcal B(k)}\frac{k_{max}^{\alpha}|k_{123}k_{124}k_{34}|^\alpha}{|k_{3}k_{124}k_{34}|}\widehat{ \mathcal{V}_h^{\pm}}(t',k_1)\widehat{ \mathcal{V}_h^{\pm}}(t',k_2)\widehat{ \mathcal{V}_h^{\pm}}(t',k_3)\widehat{ \mathcal{V}_h^{\pm}}(t',k_4),
\end{equation}
\begin{equation}\label{eq:j_41-3}
\sum_{(k_1,k_2,k_3,k_4) \in \mathcal B(k)}\frac{k_{max}^{\alpha}|k_{12}k_{1}k_{2}|^\alpha}{|k_{123}k_{124}k_{34}|}\widehat{ \mathcal{V}_h^{\pm}}(t',k_1)\widehat{ \mathcal{V}_h^{\pm}}(t',k_2)\widehat{ \mathcal{V}_h^{\pm}}(t',k_3)\widehat{ \mathcal{V}_h^{\pm}}(t',k_4),
\end{equation}
and
\begin{equation}\label{eq:j_41-4}
\sum_{(k_1,k_2,k_3,k_4) \in \mathcal B(k)}\frac{k_{max}^{\alpha}|k_{12}k_{1}k_{2}|^\alpha}{|k_{3}k_{124}k_{34}|}\widehat{ \mathcal{V}_h^{\pm}}(t',k_1)\widehat{ \mathcal{V}_h^{\pm}}(t',k_2)\widehat{ \mathcal{V}_h^{\pm}}(t',k_3)\widehat{ \mathcal{V}_h^{\pm}}(t',k_4),
\end{equation}
whenever $0 \le \alpha \le \frac12$. Let $k_{123}^* = k-k_{123}$. 

For \eqref{eq:j_41-1}, we first fix $\alpha < \frac12$. Note that
\[\sum_{\substack{0 < |k_1|,|k_2|,|k_3| < \frac{\pi}{h} \\ (k-k_3)(k-k_1-k_2) \neq 0}}\frac{|k_2|^{4\alpha}|\widehat{ \mathcal{V}_h^{\pm}}(t',k_2)|^2}{|(k-k_3)(k-k_1-k_2)|^{2-2\alpha}} \lesssim \|\mathcal{V}_h^{\pm}(t')\|_{H^{2\alpha}(\T_h)}^2\]
and 
\[\begin{aligned}
&~{}\sum_{M \le |k| \le \frac{\pi}{h}}\sum_{\substack{0 < |k_1|,|k_2|,|k_3| < \frac{\pi}{h} \\ k_{123} \neq 0}}\frac{|k_1|^{4\alpha}|\widehat{ \mathcal{V}_h^{\pm}}(t',k_1)|^2|k_3|^{4\alpha}|\widehat{ \mathcal{V}_h^{\pm}}(t',k_3)|^2|k_{123}^*|^{4\alpha}|\widehat{ \mathcal{V}_h^{\pm}}(t',k_{123}^*)|^2}{|k_1+k_2+k_3|^{2-2\alpha}}\\
\lesssim&~{}\|\mathcal{V}_h^{\pm}(t')\|_{H^{2\alpha}(\T_h)}^6.
\end{aligned}\]
Consequently,
\begin{equation}\label{eq:j_4-1}
\begin{aligned}
&~{}\left(\sum_{M \le |k| \le \frac{\pi}{h}}|\eqref{eq:j_41-1}|^2\right)^{\frac12}\\
\lesssim&~{}\Bigg(\sum_{M \le |k| \le \frac{\pi}{h}} \Bigg(\sum_{\substack{0 < |k_1|,|k_2|,|k_3| < \frac{\pi}{h} \\ (k-k_3)(k-k_1-k_2) \neq 0}}\frac{|k_2|^{4\alpha}|\widehat{ \mathcal{V}_h^{\pm}}(t',k_2)|^2}{|(k-k_3)(k-k_1-k_2)|^{2-2\alpha}}\Bigg)\\
&~{}\hspace{4em}\times \Bigg(\sum_{\substack{0 < |k_1|,|k_2|,|k_3| < \frac{\pi}{h} \\ k_{123} \neq 0}}\frac{|k_1|^{4\alpha}|\widehat{ \mathcal{V}_h^{\pm}}(t',k_1)|^2|k_3|^{4\alpha}|\widehat{ \mathcal{V}_h^{\pm}}(t',k_3)|^2|k_{123}^*|^{4\alpha}|\widehat{ \mathcal{V}_h^{\pm}}(t',k_{123}^*)|^2}{|k_1+k_2+k_3|^{2-2\alpha}}\Bigg)\Bigg)^{\frac12}\\
\lesssim&~{}\|\mathcal{V}_h^{\pm}(t')\|_{H^{2\alpha}(\T_h)}^4.
\end{aligned}
\end{equation} 
Now we consider the case  $\alpha = \frac12$. It is not difficult to see that 
\[\sum_{\substack{0 < |k_1|,|k_2|,|k_3| < \frac{\pi}{h} \\ k_{123}^*k_{123}(k-k_3)(k-k_1-k_2) \neq 0}}\frac{k_{max}}{|k_{123}(k-k_3)(k-k_1-k_2)||k_1k_2k_3k_{123}^*|^2} \lesssim 1,\]
which yields
\[\begin{aligned}
&~{}\left(\sum_{M \le |k| \le \frac{\pi}{h}}|\eqref{eq:j_41-1}|^2\right)^{\frac12}\\
\lesssim&~{}\Bigg(\sum_{M \le |k| \le \frac{\pi}{h}} \Bigg(\sum_{\substack{0 < |k_1|,|k_2|,|k_3| < \frac{\pi}{h} \\ k_{123}^*k_{123}(k-k_3)(k-k_1-k_2) \neq 0}}\frac{k_{max}}{|k_{123}(k-k_3)(k-k_1-k_2)||k_1k_2k_3k_{123}^*|^2}\Bigg)\\
&~{}\times \Bigg(\sum_{0 < |k_1|,|k_2|,|k_3| < \frac{\pi}{h}}|k_1|^{2}|\widehat{ \mathcal{V}_h^{\pm}}(t',k_1)|^2|k_2|^{2}|\widehat{ \mathcal{V}_h^{\pm}}(t',k_2)|^2|k_3|^{2}|\widehat{ \mathcal{V}_h^{\pm}}(t',k_3)|^2|k_{123}^*|^{2}|\widehat{ \mathcal{V}_h^{\pm}}(t',k_{123}^*)|^2\Bigg)\Bigg)^{\frac12}\\
\lesssim&~{}\|\mathcal{V}_h^{\pm}(t')\|_{H^{1}(\T_h)}^4.
\end{aligned}\]

For \eqref{eq:j_41-2}, we first fix $\alpha < \frac12$. Note that 
\[\sum_{0 < |k_1|,|k_2|,|k_3| < \frac{\pi}{h}}\frac{|k_1|^{4\alpha}|\widehat{ \mathcal{V}_h^{\pm}}(t',k_1)|^2|k_{123}^*|^{4\alpha}|\widehat{ \mathcal{V}_h^{\pm}}(t',k_{123}^*)|^2}{|k_3|} \lesssim \|\mathcal{V}_h^{\pm}(t')\|_{H^{2\alpha}(\T_h)}^4\]
and 
\[\sum_{M \le |k| \le \frac{\pi}{h}}\sum_{\substack{0 < |k_1|,|k_2|,|k_3| < \frac{\pi}{h} \\ (k-k_3)(k-k_1-k_2) \neq 0}}\frac{|k_2|^{4\alpha}|\widehat{ \mathcal{V}_h^{\pm}}(t',k_2)|^2|k_3|^{4\alpha}|\widehat{ \mathcal{V}_h^{\pm}}(t',k_3)|^2}{|(k-k_3)(k-k_1-k_2)|^{2-2\alpha}}\lesssim\|\mathcal{V}_h^{\pm}(t')\|_{H^{2\alpha}(\T_h)}^4.
\]
Consequently,
\[\begin{aligned}
&~{}\left(\sum_{M \le |k| \le \frac{\pi}{h}}|\eqref{eq:j_41-2}|^2\right)^{\frac12}\\
\lesssim&~{}\Bigg(\sum_{M \le |k| \le \frac{\pi}{h}} \Bigg(\sum_{0 < |k_1|,|k_2|,|k_3| < \frac{\pi}{h}}\frac{|k_1|^{4\alpha}|\widehat{ \mathcal{V}_h^{\pm}}(t',k_1)|^2|k_{123}^*|^{4\alpha}|\widehat{ \mathcal{V}_h^{\pm}}(t',k_{123}^*)|^2}{|k_3|}\Bigg)\\
&~{}\hspace{4em}\times \Bigg(\sum_{\substack{0 < |k_1|,|k_2|,|k_3| < \frac{\pi}{h} \\ (k-k_3)(k-k_1-k_2) \neq 0}}\frac{|k_2|^{4\alpha}|\widehat{ \mathcal{V}_h^{\pm}}(t',k_2)|^2|k_3|^{4\alpha}|\widehat{ \mathcal{V}_h^{\pm}}(t',k_3)|^2}{|(k-k_3)(k-k_1-k_2)|^{2-2\alpha}}\Bigg)\Bigg)^{\frac12}\\
\lesssim&~{}\|\mathcal{V}_h^{\pm}(t')\|_{H^{2\alpha}(\T_h)}^4.
\end{aligned}\] 
On the other hand, it is not difficult to see that 
\[\sum_{\substack{0 < |k_1|,|k_2|,|k_3| < \frac{\pi}{h} \\ k_{123}^*(k-k_3)(k-k_1-k_2) \neq 0}}\frac{k_{max}^2}{|k_3|^2|(k-k_3)(k-k_1-k_2)||k_1k_2k_3k_{123}^*|^2} \lesssim 1,\]
which yields
\[\begin{aligned}
&~{}\left(\sum_{M \le |k| \le \frac{\pi}{h}}|\eqref{eq:j_41-2}|^2\right)^{\frac12}\\
\lesssim&~{}\Bigg(\sum_{M \le |k| \le \frac{\pi}{h}} \Bigg(\sum_{\substack{0 < |k_1|,|k_2|,|k_3| < \frac{\pi}{h} \\ k_{123}^*(k-k_3)(k-k_1-k_2) \neq 0}}\frac{k_{max}^2}{|k_3|^2|(k-k_3)(k-k_1-k_2)||k_1k_2k_3k_{123}^*|^2}\Bigg)\\
&~{}\times \Bigg(\sum_{0 < |k_1|,|k_2|,|k_3| < \frac{\pi}{h}}|k_1|^{2}|\widehat{ \mathcal{V}_h^{\pm}}(t',k_1)|^2|k_2|^{2}|\widehat{ \mathcal{V}_h^{\pm}}(t',k_2)|^2|k_3|^{2}|\widehat{ \mathcal{V}_h^{\pm}}(t',k_3)|^2|k_{123}^*|^{2}|\widehat{ \mathcal{V}_h^{\pm}}(t',k_{123}^*)|^2\Bigg)\Bigg)^{\frac12}\\
\lesssim&~{}\|\mathcal{V}_h^{\pm}(t')\|_{H^{1}(\T_h)}^4.
\end{aligned}\]
This proves the case $\alpha = \frac12$.

For \eqref{eq:j_41-3}, we first fix $\alpha < \frac12$. Using $\displaystyle k_{12} = \frac12(k_{123} + k_{124} - k_{34})$, one sees that 
\[\frac{|k_{12}|^{\alpha}}{|k_{123}k_{124}k_{34}|} \lesssim \frac{1}{|k_{123}k_{124}k_{34}|^{1-\alpha}}.\]
An analogous argument to \eqref{eq:j_4-1} guarantees
\[\left(\sum_{M \le |k| \le \frac{\pi}{h}}|\eqref{eq:j_41-3}|^2\right)^{\frac12} \lesssim \|\mathcal{V}_h^{\pm}(t')\|_{H^{2\alpha}(\T_h)}^4.\]
Now we consider the case $\alpha = \frac12$. Note that $|k_{12}|\le \min\left(\max(|k_1|,|k_2|), \max(|k_{123}|, |k_{124}|, |k_{34}|)\right)$. Considering all cases of $k_{max}$, it follows that
\[\frac{k_{max}^{\frac12}|k_{12}k_{1}k_{2}|^{\frac12}}{|k_{123}k_{124}k_{34}||k_1k_2k_3k_4|} \lesssim \frac{1}{|k_{123}k_{124}k_4|}+\frac{1}{|k_{123}k_{34}k_4|}+\frac{1}{|k_{124}k_{34}k_3|}.\]
Then, \eqref{eq:j_41-3} with each multiplier on the right-hand side is estimated as follows: 
\[\begin{aligned}
&~{}\left(\sum_{M \le |k| \le \frac{\pi}{h}}|\eqref{eq:j_41-3}|^2\right)^{\frac12}\\
\lesssim&~{}\Bigg(\sum_{M \le |k| \le \frac{\pi}{h}} \Bigg(\sum_{\substack{0 < |k_1|,|k_2|,|k_3| < \frac{\pi}{h} \\ k_{123}^*(k-k_3) \neq 0}}\frac{|k_1|^{2}|\widehat{ \mathcal{V}_h^{\pm}}(t',k_1)|^2}{|k_{123}^*(k-k_3)|^2}\Bigg)\\
&~{}\times \Bigg(\sum_{\substack{0 < |k_1|,|k_2|,|k_3| < \frac{\pi}{h} \\ k_1+k_2+k_3 \neq 0}}\frac{|k_2|^{2}|\widehat{ \mathcal{V}_h^{\pm}}(t',k_2)|^2|k_3|^{2}|\widehat{ \mathcal{V}_h^{\pm}}(t',k_3)|^2|k_{123}^*|^{2}|\widehat{ \mathcal{V}_h^{\pm}}(t',k_{123}^*)|^2}{|k_1+k_2+k_3|^2}\Bigg)\Bigg)^{\frac12}\\
\lesssim&~{}\|\mathcal{V}_h^{\pm}(t')\|_{H^{1}(\T_h)}^4,
\end{aligned}\]
\[\begin{aligned}
&~{}\left(\sum_{M \le |k| \le \frac{\pi}{h}}|\eqref{eq:j_41-3}|^2\right)^{\frac12}\\
\lesssim&~{}\Bigg(\sum_{M \le |k| \le \frac{\pi}{h}} \Bigg(\sum_{\substack{0 < |k_1|,|k_2|,|k_3| < \frac{\pi}{h} \\ k_{123}^*(k-k_1-k_2) \neq 0}}\frac{|k_2|^{2}|\widehat{ \mathcal{V}_h^{\pm}}(t',k_2)|^2}{|k_{123}^*(k-k_1-k_2)|^2}\Bigg)\\
&~{}\times \Bigg(\sum_{\substack{0 < |k_1|,|k_2|,|k_3| < \frac{\pi}{h} \\ (k_1+k_2+k_3) \neq 0}}\frac{|k_1|^{2}|\widehat{ \mathcal{V}_h^{\pm}}(t',k_1)|^2|k_3|^{2}|\widehat{ \mathcal{V}_h^{\pm}}(t',k_3)|^2|k_{123}^*|^{2}|\widehat{ \mathcal{V}_h^{\pm}}(t',k_{123}^*)|^2}{|k_1+k_2+k_3|^2}\Bigg)\Bigg)^{\frac12}\\
\lesssim&~{}\|\mathcal{V}_h^{\pm}(t')\|_{H^{1}(\T_h)}^4,
\end{aligned}\]
and
\begin{equation}\label{eq:j_4-2}
\begin{aligned}
&~{}\left(\sum_{M \le |k| \le \frac{\pi}{h}}|\eqref{eq:j_41-3}|^2\right)^{\frac12}\\
\lesssim&~{}\Bigg(\sum_{M \le |k| \le \frac{\pi}{h}} \Bigg(\sum_{\substack{0 < |k_1|,|k_2|,|k_3| < \frac{\pi}{h} \\ k-k_1-k_2 \neq 0}}\frac{|k_1|^{2}|\widehat{ \mathcal{V}_h^{\pm}}(t',k_1)|^2|k_3|^{2}|\widehat{ \mathcal{V}_h^{\pm}}(t',k_3)|^2}{|k-k_1-k_2|^2}\Bigg)\\
&~{}\times \Bigg(\sum_{\substack{0 < |k_1|,|k_2|,|k_3| < \frac{\pi}{h} \\ k-k_3 \neq 0}}\frac{|k_2|^{2}|\widehat{ \mathcal{V}_h^{\pm}}(t',k_2)|^2|k_{123}^*|^{2}|\widehat{ \mathcal{V}_h^{\pm}}(t',k_{123}^*)|^2}{|(k-k_3)k_3|^2}\Bigg)\Bigg)^{\frac12}\\
\lesssim&~{}\|\mathcal{V}_h^{\pm}(t')\|_{H^{1}(\T_h)}^4.
\end{aligned}
\end{equation}

For \eqref{eq:j_41-4}, we first fix $\alpha < \frac12$. It follows from $\displaystyle k_{12} = k_{124} - k_{34} + k_{3}$ that 
\[\frac{|k_{12}|^{\alpha}}{|k_3k_{124}k_{34}|} \lesssim \frac{1}{|k_3k_{124}k_{34}|^{1-\alpha}}.\]
Similarly as \eqref{eq:j_4-2}, we obtain
\[\left(\sum_{M \le |k| \le \frac{\pi}{h}}|\eqref{eq:j_41-4}|^2\right)^{\frac12} \lesssim \|\mathcal{V}_h^{\pm}(t')\|_{H^{2\alpha}(\T_h)}^4.\]
On the other hand, note that $|k_{12}|\le \min\left(\max(|k_1|,|k_2|), \max(|k_{124}|, |k_{34}|, |k_3|)\right)$. By considering all possible cases of $k_{max}$, we find
\[\frac{k_{max}^{\frac12}|k_{12}k_{1}k_{2}|^{\frac12}}{|k_3k_{124}k_{34}||k_1k_2k_3k_4|} \lesssim \frac{1}{|k_{124}k_2|^{\frac12}|k_3k_4|} + \frac{1}{|k_3k_{124}k_{34}|} + \frac{1}{|k_3k_{124}k_4|}.\]
Note that the bound $\displaystyle \frac{1}{|k_{124}k_2|^{\frac12}|k_3k_4|}$ appears only when $\max(|k_1|,|k_2|) \sim |k| \gg \min(|k_1|,|k_2|), |k_3|, |k_4|$. Let us define a set by
\[\widetilde{\mathcal B}(k) = \mathcal B(k) \cap \{(k_1,k_2,k_3,k_4) \in ((\T_h)^*)^4 : \max(|k_1|,|k_2|) \sim |k| \gg \min(|k_1|,|k_2|), |k_3|, |k_4|\},\]
and a function $g_{h,\sigma}$ by $\widehat{g}_{h,\sigma}(t',k) = |k|^{-\sigma}|k||\widehat{ \mathcal{V}_h^{\pm}}(t',k)|$ for $\sigma \in \R$. Note that $|k_1| \sim |k_{124}|$ on $\widetilde{\mathcal B}(k)$. Then, by Lemma \ref{lem:Bernstein inequality}, we have
\[\begin{aligned}
&~{}\left(\sum_{M \le |k| \le \frac{\pi}{h}}\left|\sum_{(k_1,k_2,k_3,k_4) \in \widetilde{\mathcal B}(k)}\frac{k_{max}^{\frac12}|k_{12}k_{1}k_{2}|^{\frac12}}{|k_{3}k_{124}k_{34}|}\widehat{ \mathcal{V}_h^{\pm}}(t',k_1)\widehat{ \mathcal{V}_h^{\pm}}(t',k_2)\widehat{ \mathcal{V}_h^{\pm}}(t',k_3)\widehat{ \mathcal{V}_h^{\pm}}(t',k_4)\right|^2\right)^{\frac12}\\
\lesssim&~{}\|g_{h,\frac12}g_{h,\frac12}g_{h,1}g_{h,1}\|_{L^2(\T_h)}\\
\lesssim&~{}\|g_{h,\frac12}\|_{L^4(\T_h)}\|g_{h,\frac12}\|_{L^4(\T_h)}\|g_{h,1}\|_{L^\infty(\T_h)}\|g_{h,1}\|_{L^{\infty}(\T_h)}\\
\lesssim&~{}\|g_{h,\frac12}\|_{H^{\frac12}(\T_h)}\|g_{h,\frac12}\|_{H^{\frac12}(\T_h)}\|g_{h,1}\|_{H^1(\T_h)}\|g_{h,1}\|_{H^1(\T_h)}\\
\lesssim&~{}\|\mathcal{V}_h^{\pm}(t')\|_{H^{1}(\T_h)}^4.
\end{aligned}\]
With the second bound, \eqref{eq:j_41-4} is dealt with analogously to \eqref{eq:j_4-2}. Lastly,  \eqref{eq:j_41-4} with the third bound is estimated as follows: 
\[\begin{aligned}
&~{}\left(\sum_{M \le |k| \le \frac{\pi}{h}}\left|\sum_{(k_1,k_2,k_3,k_4) \in \widetilde{\mathcal B}(k) \setminus \mathcal B(k)}\frac{k_{max}^{\frac12}|k_{12}k_{1}k_{2}|^{\frac12}}{|k_{3}k_{124}k_{34}|}\widehat{ \mathcal{V}_h^{\pm}}(t',k_1)\widehat{ \mathcal{V}_h^{\pm}}(t',k_2)\widehat{ \mathcal{V}_h^{\pm}}(t',k_3)\widehat{ \mathcal{V}_h^{\pm}}(t',k_4)\right|^2\right)^{\frac12}\\
\lesssim&~{}\Bigg(\sum_{M \le |k| \le \frac{\pi}{h}} \Bigg(\sum_{\substack{0 < |k_1|,|k_2|,|k_3| < \frac{\pi}{h} \\ k_{123}^* \neq 0}}\frac{|k_2|^{2}|\widehat{ \mathcal{V}_h^{\pm}}(t',k_2)|^2}{|k_{123}^*k_3|^2}\Bigg)\\
&~{}\times \Bigg(\sum_{\substack{0 < |k_1|,|k_2|,|k_3| < \frac{\pi}{h} \\ k-k_3 \neq 0}}\frac{|k_1|^{2}|\widehat{ \mathcal{V}_h^{\pm}}(t',k_1)|^2|k_3|^{2}|\widehat{ \mathcal{V}_h^{\pm}}(t',k_3)|^2|k_{123}^*|^{2}|\widehat{ \mathcal{V}_h^{\pm}}(t',k_{123}^*)|^2}{|k-k_3|^2}\Bigg)\Bigg)^{\frac12}\\
\lesssim&~{}\|\mathcal{V}_h^{\pm}(t')\|_{H^{1}(\T_h)}^4.
\end{aligned}\]
All these prove the case $\alpha = \frac12$.

By collecting all, we conclude that
\[\|P_{ M \le \cdot \le \frac{\pi}{h}} J_{41}(t')\|_{L^2(\T)} \lesssim h^{\alpha}(1+|t'|)^{\alpha}\|\mathcal{V}_h^{\pm}(t')\|_{H^{2\alpha}(\T_h)}^4,\]
for any $0 \le \alpha \le \frac12$. For $J_{42}$ and $J_{43}$, similarly as \eqref{eq:I_2} and \eqref{eq:I_3}, respectively, but using \eqref{ineq:D4}, we obtain
\[\begin{aligned}
\|P_{ M \le \cdot \le \frac{\pi}{h}} J_{42}(t')\|_{L^2(\T)} \lesssim&~{} \left(\sum_{|k| \le \frac{\pi}{h}}\left|\widehat{\mathcal{L}_{h}\mathcal V_{h}^{\pm}} - \widehat{\mathcal{V}_h^{\pm}} \right|^2\right)^{\frac12} \\
&~{}\times \Bigg(\|\mathcal{L}_h\mathcal{V}_h^{\pm}(t')\|_{L^2(\T)}^3+\|\mathcal{L}_h\mathcal{V}_h^{\pm}(t')\|_{L^2(\T)}^2\|\mathcal{V}_h^{\pm}(t')\|_{L^2(\T_h)}\\
&~{}\hspace{2em}+\|\mathcal{L}_h\mathcal{V}_h^{\pm}(t')\|_{L^2(\T)}\|\mathcal{V}_h^{\pm}(t')\|_{L^2(\T_h)}^2+\|\mathcal{V}_h^{\pm}(t')\|_{L^2(\T_h)}^3\Bigg)\\
\lesssim&~{} h^\alpha\|\mathcal{V}_h^{\pm}(t')\|_{H^{\alpha}(\T_h)} \left(\|\mathcal{L}_h\mathcal{V}_h^{\pm}(t')\|_{L^2(\T)}+\|\mathcal{V}_h^{\pm}(t')\|_{L^2(\T_h)}\right)^3,
\end{aligned}\]
for any $0 \le \alpha \le 2$, and
\[\begin{aligned}
  \|P_{ M \le \cdot \le \frac{\pi}{h}} J_{43}(t')\|_{L^2(\T)}\lesssim&~{} \|P_{> \frac{\pi}{h}}\mathcal{L}_{h}\mathcal V_{h}^{\pm}(t')\|_{L^2(\T)} \|\mathcal{L}_h\mathcal{V}_h^{\pm}(t')\|_{L^2(\T)}^3\\
\lesssim&~{}h^{\alpha}\|\mathcal{L}_h \mathcal{V}_h^{\pm}(t')\|_{H^{\alpha}(\T)}\|\mathcal{L}_h \mathcal{V}_h^{\pm}(t')\|_{L^2(\T)}^2,
\end{aligned}\]
for any $\alpha \ge 0$. Collecting all and taking $2\alpha = s$ in addition to Lemma \ref{Lem:discretization linearization inequality}, we obtain 
\begin{equation}\label{eq:J_4 term}
\begin{aligned}
&~{}\|P_{ M \le \cdot \le \frac{\pi}{h}} J_4(t')\|_{L^2(\T)} \\
\lesssim&~{} \int_0^t \Bigg(\| (\mathcal L_h\mathcal{V}_h^{\pm}-\mathcal{W}^{\pm})(t') \|_{L^2(\T)}\left(\|\mathcal{V}_h^{\pm}(t')\|_{H^s(\T_h)} + \|\mathcal{W}^{\pm}(t')\|_{H^s(\T)}\right)^3\\
&~{}\hspace{2em}+h^{\frac{s}{2}}(1+|t'|)^{\frac{s}{2}}\|\mathcal{V}_h^{\pm}(t')\|_{H^s(\T_h)}^4+h^{\frac{s}{2}}\|\mathcal{V}_h^{\pm}(t')\|_{H^s(\T_h)}^4\Bigg)\;dt'.
\end{aligned}
\end{equation}
Collecting \eqref{eq:J_1 term}, \eqref{eq:J_2 term}, \eqref{eq:J_3 term}, and \eqref{eq:J_4 term}, one concludes \eqref{lowfreq:esti}.
\end{proof}

\begin{proof}[Proof of Proposition \ref{prop:from decoupled to kdv}]
For given $R > 0$, let $u_{h,0}^{\pm}\in H^s(\T_h)$ denote the initial data satisfying
\[\sup_{h \in (0,1]}\big\|\big( u_{h,0}^{+},u_{h,0}^{-}\big)\big\|_{\mathbb H^s(\T_h)} \le R,\]
and let $T > 0$ be a common lifespan of solutions to \eqref{decoupled FPU'} and \eqref{kdv integral form}\footnote{One can take the minimum $T$ appearing in Corollaries~\ref{Cor:Uniform bounds} and \ref{cor:uniform bound for KdV}}. Then, by Corollaries~\ref{Cor:Uniform bounds} and \ref{cor:uniform bound for KdV} in addition to Lemma \ref{Lem:discretization linearization inequality}, the solutions to \eqref{decoupled FPU'} and \eqref{kdv integral form} are uniformly (in $h$) bounded in terms of the initial data on their lifespan, i.e.,
  \[\sup_{t\in [-T,T]}\|v_h^{\pm}(t)\|_{H^s(\T_h)}, \;  \sup_{t\in [-T,T]}\|w^{\pm}(t)\|_{H^s(\T)} \ls R,\]
as well as
\begin{equation}\label{uniform bounds of sols-1}
  \sup_{t\in [-T,T]}\|\mathcal{V}^{\pm}_h(t)\|_{H^s(\T_h)}, \;  \sup_{t\in [-T,T]}\|\mathcal{W}^{\pm}(t)\|_{H^s(\T)} \ls R,
\end{equation}
where $\mathcal{V}^{\pm}_h(t)$ and $\mathcal{W}^{\pm}(t)$ are defined as in \eqref{eq:renormalization} and \eqref{eq:renormalization of w}, respectively. Note that
\[\mathcal L_h v_h^{\pm} (t) - w^{\pm}(t) = \left(\mathcal L_hS_h^{\pm}(t)\mathcal{V}^{\pm}_h(t) - S^{\pm}(t)\mathcal L_h\mathcal{V}^{\pm}_h(t)\right) + \left(S^{\pm}(t)\mathcal L_h\mathcal{V}^{\pm}_h(t) - S^{\pm}(t)\mathcal{W}^{\pm}(t)\right).\]
By Proposition \ref{prop:commutator of interpolation and propagator} and \eqref{uniform bounds of sols-1}, we know
\[\|\left(\mathcal L_hS_h^{\pm}(t)\mathcal{V}^{\pm}_h - S^{\pm}(t)\mathcal L_h\mathcal{V}^{\pm}_h\right)(t)\|_{L^2(\T)} \lesssim_R h^{\frac{2s}{5}}.\]
Thus, it suffices to show that there exists $0 < h_0 = h_0(R) <1$ sufficiently small such that 
\begin{equation}\label{eq:proof-2}
\|\mathcal L_h\mathcal{V}^{\pm}_h(t) - \mathcal{W}^{\pm}(t)\|_{C_t([-T,T];L^2(\T))} \lesssim_R h^{\frac{2s}{5}}
\end{equation}
for all $0 < h \le h_0$. Here, time $T >0$ could be smaller than the one appearing in \eqref{uniform bounds of sols-1}. Fix $M > 0$, sufficiently large and depending only on $R$, such that
\[\frac{C}{M^{\frac12}} \left(1+R\right)^2 \le \frac14,\]
where $C>0$ is a universal constant appeared in Propositions \ref{prop:exterior estimates}, \ref{prop:differnece estimates-1}, and \ref{prop:differnece estimates-2}, and \eqref{uniform bounds of sols-1}, which is independent not only on $h$, but also on $R$. For this $M$, take $h_0$ satisfying $h_0M < \pi$. For $h \le h_0$, by Propositions \ref{prop:exterior estimates}, \ref{prop:differnece estimates-1}, and \ref{prop:differnece estimates-2} in addition to \eqref{uniform bounds of sols-1}, we obtain
\[\begin{aligned}
&~{}\|\mathcal L_h\mathcal{V}^{\pm}_h(t) - \mathcal{W}^{\pm}(t)\|_{L^2(\T)}\\
\le&~{}\| P_{\le M}(\mathcal L_h\mathcal{V}^{\pm}_h - \mathcal{W}^{\pm})(t)\|_{L^2(\T)}+ \| P_{M \le \cdot \le \frac{\pi}{h}}(\mathcal L_h\mathcal{V}^{\pm}_h - \mathcal{W}^{\pm})(t)\|_{L^2(\T)}  + \| P_{\ge \frac{\pi}{h}}(\mathcal L_h\mathcal{V}^{\pm}_h - \mathcal{W}^{\pm})(t)\|_{L^2(\T)}\\
\le&~{} Ch^{\frac{s}{2}}(1+T)^{1+\frac{s}{2}}M^{\frac{3}{2}+\frac{s}{2}}\left(1+R\right)R^2\\
&~{}+\frac{C}{M^{\frac12}} \|( \mathcal{L}_h\mathcal{V}_h^{\pm}-\mathcal{W}^{\pm})(t')\|_{C_t([-T,T];L^2(\T))}\left(1+R\right)^2\\
&~{}+ CT \|( \mathcal{L}_h\mathcal{V}_h^{\pm}-\mathcal{W}^{\pm})(t')\|_{C_t([-T,T];L^2(\T))}\left(M^{\frac{3}{2}}R + \left(1+ R \right)^3\right),
\end{aligned}\]
for any $0 \le t \le T$. 

For fixed $M>0$, we can take $T > 0$ sufficiently small such that
\[CT \left(M^{\frac{3}{2}}R + \left(1+R \right)^3\right) \le \frac14.\]
Consequently, we obtain
\[\|\mathcal L_h\mathcal{V}^{\pm}_h(t) - \mathcal{W}^{\pm}(t)\|_{L^2(\T)} \le 2Ch^{\frac{s}{2}}(1+T)^{1+\frac{s}{2}}M^{\frac{3}{2}+\frac{s}{2}}\left(1+R\right)R^2,\]
which guarantees \eqref{eq:proof-2}.
\end{proof}

\end{document}